\nonstopmode
\documentclass[10pt,reqno]{amsart}
\usepackage{graphicx}
\usepackage{latexsym}
\usepackage{fancyhdr}
\usepackage{amsmath, amssymb}
\usepackage[all]{xy}
\usepackage{pdflscape}
\usepackage{longtable}
\usepackage{rotating}
\usepackage{verbatim}
\usepackage{subfigure}
\usepackage{hyperref}
\hypersetup{
  linkcolor  = blue,
  citecolor  = blue,
  urlcolor   = blue,
  colorlinks = true,
}

\usepackage{mathrsfs}
\usepackage{tensor}
\usepackage{mdwlist}
\usepackage{etoolbox}
\usepackage{mathtools}
\usepackage{todonotes}
\usepackage{cleveref}
\usepackage{stmaryrd}

\allowdisplaybreaks

\theoremstyle{plain}
\newtheorem*{theorem*}{Theorem}
\newtheorem{theorem}{Theorem}[section]
\newtheorem*{lemma*}{Lemma}
\newtheorem{lemma}[theorem]{Lemma}
\newtheorem*{proposition*}{Proposition}
\newtheorem{proposition}[theorem]{Proposition}
\newtheorem*{corollary*}{Corollary}
\newtheorem{corollary}[theorem]{Corollary}
\newtheorem*{claim*}{Claim}

\newtheorem{question}[theorem]{Question}
\theoremstyle{definition}
\newtheorem*{assumption*}{Assumption}

\newtheorem*{definition*}{Definition}
\newtheorem{definition}[theorem]{Definition}
\newtheorem*{convention*}{Convention}
\newtheorem{convention}[theorem]{Convention}
\newtheorem*{example*}{Example}
\newtheorem{example}[theorem]{Example}

\newtheorem*{algorithm*}{Algorithm}
\newtheorem*{remark*}{Remark}
\newtheorem*{remarks*}{Remarks}
\newtheorem{remark}[theorem]{Remark}

\numberwithin{equation}{section}

\newcounter{mnotecount}[section]

\newcommand{\rmnote}[1]{}

\def\al{\alpha}
\def\be{\beta}
\def\ga{\gamma}
\def\de{\delta}
\def\ep{\epsilon}

\def\et{\eta}
\def\th{\theta}
\def\vt{\vartheta}
\def\ka{\kappa}
\def\la{\lambda}

\def\rh{\rho}

\def\si{\sigma}

\def\ta{\tau}

\def\ph{\varphi}
\def\vh{\varphi}
\def\ch{\chi}
\def\ps{\psi}
\def\om{\omega}

\def\Ga{\Gamma}

\def\La{\Lambda}
\def\Si{\Sigma}

\def\Om{\Omega}

\def\C{\mathbb{C}}

\def\K{\mathbb{K}}
\def\N{\mathbb{N}}

\def\R{\mathbb{R}}

\def\cA{\mathcal{A}}
\def\cB{\mathcal{B}}

\def\cD{\mathcal{D}}
\def\cE{\mathcal{E}}

\def\cH{\mathcal{H}}

\def\cN{\mathcal{N}}

\def\cQ{\mathcal{Q}}

\def\fG{\mathfrak{G}}

\def\fM{\mathfrak{M}}
\def\fN{\mathfrak{N}}

\def\fS{\mathfrak{S}}
\def\fT{\mathfrak{T}}

\def\fW{\mathfrak{W}}

\def\sP{\mathscr{P}}

\def\RR{\R}
\def\NN{\N}

\def\p{\partial}

\def\ol{\overline}
\def\ul{\underline}

\def\preceq{\preccurlyeq}

\def\A{\;\forall}
\def\E{\;\exists}

\def\Bmap{j^\infty_{\{0\}}}

\renewcommand{\Re}{\mathrm{Re}}
\renewcommand{\Im}{\mathrm{Im}}
\def\<{\langle}
\def\>{\rangle}
\renewcommand{\o}{\circ}

\def\supp{\on{supp}}

\let\on=\operatorname
\newcommand{\sr}[1]%
{\ifmmode{}^\dagger\else${}^\dagger$\fi\ifvmode
\vbox to 0pt{\vss
 \hbox to 0pt{\hskip\hsize\hskip1em
 \vbox{\hsize3cm\raggedright\pretolerance10000
 \noindent #1\hfill}\hss}\vss}\else
 \vadjust{\vbox to0pt{\vss%
 \hbox to 0pt{\hskip\hsize\hskip1em%
 \vbox{\hsize3cm\raggedright\pretolerance10000%
 \noindent #1\hfill}\hss}\vss}}\fi%
}

\title{Ultradifferentiable extension theorems: a survey}

\author[Armin Rainer]
{Armin Rainer}

\address{Fakult\"at f\"ur Mathematik, Universit\"at Wien,
Oskar-Morgenstern-Platz~1, A-1090 Wien, Austria}

\email{armin.rainer@univie.ac.at}

\begin{document}

\begin{abstract}
  We survey ultradifferentiable extension theorems, i.e., quantitative versions of Whitney's classical extension theorem,
  with special emphasis on the existence of continuous linear extension operators.
  The focus is on Denjoy--Carleman classes for which we develop the theory from scratch and discuss important related concepts
  such as (non-)quasianalyticity.
  It allows us to give an efficient and, to a fair extent, elementary
  introduction to Braun--Meise--Taylor classes based on their representation
  as intersections and unions of  Denjoy--Carleman classes.
\end{abstract}

\thanks{The author was supported by the Austrian Science Fund (FWF), Grant P~32905-N and START Programme Y963}
\keywords{Whitney's extension theorem, Borel map, (non-)quasianalyticity, ultradifferentiable classes, extension operators}
\subjclass[2020]{
26E10, 
30D60, 
46E10, 
46E25, 
46A63, 
58C25  
}
\date{\today}

\maketitle

\tableofcontents

\section{Introduction}

The development of differential analysis in the last century was strongly influenced by Whitney's work on differentiable functions.
Especially fruitful was Whitney's extension theorem which describes the restrictions of smooth functions on $\R^n$ to closed subsets.
For a closed non-empty set $A \subseteq \R^n$ we consider the \emph{restriction map}
\[
  j^\infty_A : C^\infty(\R^n) \to C^0(A)^{\N^n}, \quad f \mapsto (f^{(\al)}|_A)_{\al \in \N^n}.
\]
Taylor's theorem implies that the image $j^\infty_A  C^\infty(\R^n)$ is contained in the set $\cE(A)$ of \emph{Whitney jets}:
$F= (F^\al)_\al \in C^0(A)^{\N^n}$ belongs to $\cE(A)$ if for all compact subsets $K \subseteq A$, all $p \in \N$, and all $|\al| \le p$
\[
	(R^p_xF)^\al(y) = o(|x-y|^{p-|\al|}) \quad \text{ as } |x-y| \to 0, ~ x,y \in K,
\]
where
\[
	(R^p_xF)^\al(y) := F^\al(y) - \sum_{|\be|\le p -|\al|} \frac{(y-x)^\be}{\be!} F^{\al+\be}(x).
\]
Whitney's extension theorem \cite{Whitney34a} states that $j^\infty_A  C^\infty(\R^n)=\cE(A)$, i.e.,
every Whitney jet $F \in \cE(A)$ admits an extension $f \in C^\infty(\R^n)$ such that $j^\infty_Af = F$.

In this paper we shall be interested in quantitative versions of Whitney's extension theorem.

\begin{question} \label{Q1}
	Given that a Whitney jet $F \in \cE(A)$ satisfies certain uniform growth properties,
	can these properties be preserved by the extension?
	If they cannot be preserved, can the loss of regularity be controlled?
  Can the extension be performed by a continuous linear map?
\end{question}

The uniform growth properties we have in mind are bounds on the multisequence of partial derivatives imposed in terms of a suitable weight
sequence which measures the deviation from the Cauchy estimates and hence from analyticity.
They give rise to so-called \emph{ultradifferentiable classes}
which form scales of regularity classes between the real analytic and the smooth class.
More specifically, the use of a \emph{weight sequence} $M$ leads to \emph{Denjoy--Carleman classes} which originated around 1900
in work of Borel, Gevrey, Holmgren, Hadamard, etc. A different approach based on decay properties of the Fourier transform is due to
Beurling and Bj\"orck in the 1960s. An equivalent description of this second approach was later given by Braun, Meise, and Taylor.
The resulting classes are called \emph{Braun--Meise--Taylor classes}; their definition involves a \emph{weight function} $\om$.

In this survey we will first treat the extension problem extensively for Denjoy--Carleman classes.
In a second part we will discuss it for Braun--Meise--Taylor classes. Our treatment of the latter is based on
their description as intersections and unions of Denjoy--Carleman classes in terms of an associated
family of weight sequences (called \emph{weight matrix}).

Our study of \Cref{Q1} will first focus on the simplest case when the closed set $A$ is just the singleton $\{0\}$.
Then $j^\infty_{\{0\}}$ is called \emph{Borel map}, $\cE(\{0\})$ is isomorphic to
the ring of formal power series $\K[[X_1,\ldots,X_n]]$ (where $\K$ is $\R$ or $\C$), and Whitney's extension theorem reduces to Borel's lemma.
\'Emile Borel's discovery at the end of the 19th century
that there exist classes of smooth functions which have an analytic continuation property (called \emph{quasianalyticity}),
but contain functions that are nowhere analytic, led to a lot of activity. A class of smooth functions is quasianalytic
if the restriction of the Borel map to this class is injective. The famous Denjoy--Carleman theorem characterizes quasianalyticity
of Denjoy--Carleman classes in terms of the weight sequence $M$. We give a full proof based on two
elementary lemmas (\Cref{lem:sizenearflat} and \Cref{lem:specialbump}) which
will turn out to be foundational for the whole theory.
For example, we will deduce the characterization of quasianalyticity for
Braun--Meise--Taylor classes from the Denjoy--Carleman theorem.

The bounds on the derivatives defining an ultradifferentiable class naturally
induce corresponding bounds for the infinite jet of such functions at $0$.
Thus the natural codomain of the Borel map on an ultradifferentiable class is a sequence space defined by the corresponding bounds,
and one may ask about surjectivity of this map.
We will see that surjectivity on Denjoy--Carleman classes is equivalent to a condition on the weight sequence $M$
that is called \emph{strong non-quasianalyticity}.
It will turn out that this condition is indispensable for the existence of suitable (even \emph{optimal}) cutoff functions
crucial for the solution of the extension problem.
By different methods (studying the distribution of zeros of quasianalytic functions and their derivatives)
we shall also see that the Borel map is never surjective in the quasianalytic setting,
except for the real analytic class.

All ultradifferentiable classes considered in this paper come in two types,
\emph{Beurling} type and \emph{Roumieu} type.
The classes of Beurling type carry a natural locally convex topology that make them to Fr\'echet spaces, the
topology of the classes of Roumieu type is more complicated.
We shall see that the Borel map on Denjoy--Carleman classes of Beurling type is even split surjective if
$M$ is strongly non-quasianalytic. In the Roumieu case, generally, the Borel map does not admit a continuous linear
right-inverse.

Having solved the extension problem for the Borel map, the solution of the general problem for arbitrary
non-empty closed subsets $A$ of $\R^n$ depends on the existence of \emph{optimal} cutoff functions.
They are optimal in the sense that they realize necessary sharp bounds.
In conjunction with a family of Whitney cubes for $A$, they yield optimal partitions of unity with the help
of which the local extensions (provided by the solution for the Borel map) can be glued to a global extension
of Whitney ultrajets on $A$. The construction of these cutoff functions involves some delicate properties of
auxiliary functions associated with the weight sequence $M$.

The existence of extensions in the Beurling case follows by a reduction argument from the Roumieu case.
But we will also find optimal cutoff functions of Beurling type and utilize them for a direct proof of the
split surjectivity of the restriction map $j^\infty_A$ on Denjoy--Carleman classes of Beurling type.
In this way we give an elementary constructive proof for the existence of extension operators in this setting,
without relying on the abstract splitting theorem for Fr\'echet spaces, which we shall however discuss briefly.

If strong non-quasianalyticity is lacking, and hence extension preserving the class is impossible,
one is led to the problem of controlling the loss of regularity.
This problem is solved by describing the pairs
of weight sequences $(M,N)$ (later called \emph{admissible pairs})
such that Whitney jets on a closed set $A$ satisfying $M$-bounds admit extensions with $N$-bounds, both in the Beurling and
Roumieu case. Technically, this requires an interesting tool inspired by Dyn'kin's theory of almost analytic functions:
instead of the local extensions (for the singleton) one uses the Taylor polynomials of the jet
to higher and higher degree as the closed set $A$ is approached.

The second part of the survey is dedicated to a concise introduction to Braun--Meise--Taylor classes
and a discussion of the extension problem in that framework.
Our approach is based on the description of these classes as suitable intersections and unions of Denjoy--Carleman classes
which allows (to a fair extent) for a swift and elementary treatment, building on the extensive study of the latter in part one.
The discussion of the extension problem for Braun--Meise--Taylor classes will be more expository (compared to part one).
We will present the state of the art of this area mostly without proofs but indicating
the important ideas and methods involved.
Interestingly, there appear phenomena that are not present in the framework of Denjoy--Carleman classes.
For instance, if extensions preserving the class are possible, even in the Beurling case
they cannot always be realized by extension operators.
The existence of extension operators (always in the Beurling case) depends on the geometry of the set $A$ and
on the weight function $\om$.
The singleton $\{0\}$, for example, admits an extension operator if and only if $\om$ has a certain additional property
(namely, it is a so-called \emph{(DN)-weight}).
And, if $\om$ has this property, then every closed set $A$ has an extension operator.
Given that $\om$ lacks that property, then the existence of an extension operator on a compact set $K$ is characterized by the
linear topological invariant (DN) of the space of Whitney $\om$-ultrajets of Beurling type on $K$.
A compact set may or may not satisfy this condition: sets with real analytic boundary do
and sets with sharp (i.e., infinitely flat) cusps do not.

In order to make the exposition not too technical we do not strive for the utmost generality of the results.
But we provide appropriate references for the interested reader.

This survey article arose from notes for a mini-course I gave at the
\emph{School of Real Geometry in Fortaleza}, Brazil, May 24-28, 2021.\footnote{\url{https://sites.google.com/view/scregefor2020/}}

\subsection*{Notation}

We use standard multiindex notation.

The supremum norm is denoted by $\|u\|_K := \sup_{x\in K} |u(x)|$.

A sequence $(a_k)$ is called \emph{increasing} if $a_k \le a_{k+1}$
and \emph{strictly increasing} if $a_k<a_{k+1}$ for all $k$;
analogously with \emph{decreasing}. That an increasing sequence $(a_k)$ tends to infinity
is abbreviated by $a_k \nearrow \infty$.

For two non-negative sequences $a=(a_k)$ and $b=(b_k)$ we write $a \lesssim b$ if there is a constant $C>0$
such that $a_k \le C b_k$ for all $k$; similarly for functions.
We will also use $a \preceq b$ for $a_k^{1/k} \lesssim b_k^{1/k}$ and
$a \lhd b$ for $a_k^{1/k}/ b_k^{1/k} \to 0$.

The indicator function of a subset $A \subseteq \R^n$ is denoted by $\mathbf{1}_A$.
If $A$ is non-empty, then $\on{diam} A := \sup\{|a- b| : a,b \in A\}$ is the diameter of $A$.
If $B \subseteq \R^n$ is another non-empty set, then
$\on{dist}(A,B) := \inf\{|a-b| : a \in A, \, b \in B\}$ is the Euclidean distance between $A$ and $B$.
In particular, $d_A(z) = d(z,A) = \on{dist}(z,A) := \on{dist}(\{z\},A)$ for $z \in \R^n$.

For open subsets $A,B \subseteq \R^n$ we use the symbol $A \Subset B$ to indicate that $A$ is relatively compact in $B$.
And we write $K \subseteq_{cp} A$ if $K$ is a compact subset of $A$.

\part{Denjoy--Carleman classes}

\section{Ultradifferentiable functions}

In this section we discuss natural growth conditions for the infinite sequence of derivatives of smooth functions.

\subsection{Denjoy--Carleman classes}

Let $M= (M_k)_{k \in \mathbb{N}}$ be a sequence of positive real numbers.
Let $U$ be an open subset of $\mathbb{R}^n$.
For $f \in C^\infty(U)$, $\rho >0$, and compact $K \subseteq U$ we consider
the seminorm
\[
	\|f\|^M_{K,\rho} := \sup_{x \in K} \sup_{\alpha \in \mathbb{N}^n} \frac{|f^{(\alpha)}(x)|}{\rho^{|\alpha|} M_{|\alpha|}}
	= \sup_{\alpha \in \mathbb{N}^n} \frac{\|f^{(\alpha)}\|_K}{\rho^{|\alpha|} M_{|\alpha|}},
\]
where $\|u\|_K := \sup_{x\in K} |u(x)|$ denotes the supremum norm.
(In the definition of $\|f\|^M_{K,\rho}$ it is not important that $K$ is compact; we will occasionally use $\|f\|^M_{U,\rho}$ for
open sets $U$.)
We define the \emph{Denjoy--Carleman class of Beurling type}
\begin{align*}
	\mathcal{E}^{(M)}(U) &:= \big\{f \in C^\infty(U) : \forall K \subseteq_{cp} U ~\forall \rh > 0 : \|f\|^M_{K,\rho} < \infty\big\}
\end{align*}
and the \emph{Denjoy--Carleman class of Roumieu type}
\begin{align*}
	\mathcal{E}^{\{M\}}(U) &:= \big\{f \in C^\infty(U) : \forall K \subseteq_{cp} U ~\exists \rh > 0 : \|f\|^M_{K,\rho} < \infty\big\}.
\end{align*}
It is convenient to consider also the \emph{global Denjoy--Carleman classes}
\begin{align*}
	\mathcal{B}^{(M)}(U) &:= \big\{f \in C^\infty(U) : \forall \rh > 0 : \|f\|^M_{U,\rho} < \infty\big\},
	\\
	\mathcal{B}^{\{M\}}(U) &:= \big\{f \in C^\infty(U) : \exists \rh > 0 : \|f\|^M_{U,\rho} < \infty\big\},
\end{align*}
as well as the Banach space $\cB^{M}_\rh(U) = \{f \in C^\infty(U) : \|f\|^M_{U,\rho} < \infty\}$.
We consider the natural locally convex topologies
on these spaces, i.e.,
\begin{align*}
  		\mathcal{B}^{(M)}(U) &= \on{proj}_{n\in \N} \cB^{M}_{1/n}(U), \quad
	\mathcal{B}^{\{M\}}(U) = \on{ind}_{n\in \N} \cB^{M}_n(U),
\end{align*}
and
\begin{align*}
	\cE^{[M]}(U) = \on{proj}_{V \Subset U} \mathcal{B}^{[M]}(V).
\end{align*}

\begin{convention} \label{convention}
  For notational convenience we use $\mathcal{E}^{[M]}$  as placeholder for either $\mathcal{E}^{(M)}$
  or $\mathcal{E}^{\{M\}}$
  with the understanding that if the placeholder appears repeatedly in a statement then it must be interpreted by
  either $\mathcal{E}^{(M)}$ or $\mathcal{E}^{\{M\}}$ at all instances.
  In an analogous fashion we use $\cB^{[M]}$, $\cD^{[M]}$, etc.
\end{convention}

The spaces $\cB^{(M)}(U)$ and $\cE^{(M)}(U)$ are Fr\'echet spaces, $\cB^{\{M\}}(U)$ is a Silva space,
and all spaces are nuclear provided that the weight sequence $M$ is derivation-closed, see \Cref{sec:stability}.

Furthermore, we need the subspaces consisting of functions with compact support: for compact $K \subseteq \R^n$ let
\[
	\cD^{[M]}(K) := \{f \in \cE^{[M]}(\R^n) : \on{supp} f \subseteq K\}
\]
with the induced topology.
If $U \subseteq\R^n$ is open,
we set
\[
	\cD^{[M]}(U) := \on{ind}_{K\subseteq_{cp} U} \cD^{[M]}(K).
\]
We shall see below that $\cD^{[M]}(U)$ can be trivial.

For the sequence $M_k = k!$, the Roumieu class $\mathcal{E}^{\{M\}}(U)$ coincides with the class of real analytic functions $C^\omega(U)$ and
the Beurling class $\mathcal{E}^{(M)}(U)$ consists of the restrictions to $U$ of the entire functions $\cH(\C^n)$ on $\mathbb{C}^n$.
This follows from the Cauchy estimates. In general the sequence $M$ describes a deviation from the Cauchy estimates.

\subsection{Weight sequences}

It is convenient to require some basic mild regularity conditions for the sequence $M$.
Most importantly we assume that $M$ is
\emph{log-convex}, that is $(\log M_k)_k$ is a convex sequence.\footnote{Generally, this assumption can be made without loss of generality.
Indeed, if $M$ is a positive sequence, then
$\mathcal{E}^{\{M\}}(U) = \mathcal{E}^{\{\underline M\}}(U)$
and $\mathcal{E}^{(M)}(U) = \mathcal{E}^{(\underline M)}(U)$,
where $\ul M$ is the log-convex minorant of $M$,
provided that the classes contain the real analytic class, respectively.
This follows from the Cartan--Gorny inequality, see \cite{Gorny39,Cartan40} and \cite[Theorem 2.15]{RainerSchindl12}.
\label{fn1}}
Equivalently, the associated sequence $\mu$ with
\[
  \mu_k := \frac{M_k}{M_{k-1}}, \quad k \ge 1,
\]
is increasing. We also assume $\mu_0:= 1 \le \mu_1$.

\begin{definition}
We call a positive log-convex sequence $M=(M_k)$ with $M_0= 1 \le M_1$ and $M_k^{1/k} \to \infty$
a \emph{weight sequence}.
\end{definition}

Log-convexity of the sequence $m$ defined by
\[
  m_k := \frac{M_k}{k!}
\]
is a stronger useful property; it means that $\mu^*$ with $\mu^*_0:=1$ and
\begin{equation*}
  \mu_k^* := \frac{m_k}{m_{k-1}} = \frac{\mu_k}{k}, \quad k\ge 1,
\end{equation*}
is increasing.
A weight sequence $M$ with this property is said to be \emph{strongly log-convex}.

\begin{lemma} \label{lem:atinfinity}
	A positive log-convex sequence $M$ with $M_0=1 \le M_1$ has the following properties:
	\begin{enumerate}
		\item $M_k^{1/k} \le \mu_k$ for all $k \ge 1$.
		\item The sequences $(M_k)$ and $(M_k^{1/k})$ are increasing.
		\item $M_j M_k \le M_{j+k}$ for all $j,k \in \N$.
		\item $M_k^{1/k} \to \infty$ if and only if $\mu_k \to \infty$.
	\end{enumerate}
\end{lemma}

\begin{proof}
By assumption, we have $1 \le \mu_1 \le \mu_2 \le \cdots$ and thus $M_k^{1/k} = (\mu_1 \mu_2 \cdots \mu_k)^{1/k} \le \mu_k$, that is (1).
An easy computation shows that $M_k^{1/k} \le M_{k+1}^{1/(k+1)}$ is equivalent to $M_k^{1/k}\le \mu_{k+1}$. So (2) follows.
To see (3) observe that
$M_jM_k\le (\mu_1\cdots\mu_j)(\mu_1\cdots\mu_k) \le \mu_1 \cdots \mu_{j+k} = M_{j+k}$.
Let us check (4).
That $M_k^{1/k} \to \infty$ implies $\mu_k \to \infty$ follows from (1).
	If $\mu_k \to \infty$, then for each positive integer $n$ there is $k_n$ such that $\mu_k \ge n$ for all $k \ge k_n$. Then
	\[
		M_{2k_n} =  M_{k_n} \mu_{k_n+1} \cdots \mu_{2k_n} \ge n^{k_n}
	\]
	and thus $M_{2k_n}^{1/(2k_n)} \ge  \sqrt{n}$ for all $n$. So $M_k^{1/k} \to \infty$, since $M_k^{1/k}$ is increasing.
\end{proof}

\subsection{Characteristic functions}

Every Roumieu class contains functions whose derivatives
max out the defining bounds.

\begin{lemma} \label{lem:characteristic}
	If $M=(M_k)$ is a weight sequence, then there exists $f \in \cB^{\{M\}}(\mathbb{R},\mathbb{C})$
	such that $f^{(k)}(0) = i^k a_k$ with $a_k \ge M_k$ for all $k \in \mathbb{N}$ and $i= \sqrt{-1}$.
\end{lemma}

\begin{proof}
	The following construction is due to \cite{Bang46}, see also \cite[Theorem 1]{Thilliez08}.
	Since $\mu_k$ is increasing, we have
	\[
		\mu_{k+1}^{j-k} \le \frac{M_j}{M_k}, \quad \text{ for all } (j,k) \in \mathbb{N}^2.
	\]
	Then the function
	\[
			f(x) := \sum_{k = 0}^\infty \frac{M_k}{(2 \mu_{k+1})^k} e^{2 i \mu_{k+1} x}
	\]
	has the required properties. Indeed, $j$-fold term-wise differentiation yields
	\[
		\Big|\sum_{k = 0}^\infty \frac{M_k}{(2 \mu_{k+1})^k} (2 i \mu_{k+1})^j  e^{2 i \mu_{k+1} x}\Big|
		\le M_j \sum_{k = 0}^\infty 2^{j-k}  = 2^{j+1} M_j
	\]
	whence the series is uniformly convergent and $f\in \cB^{\{M\}}(\mathbb{R}, \mathbb C)$. Moreover,
	\[
		\frac{f^{(j)}(0)}{i^j} =  \sum_{k = 0}^\infty \frac{M_k}{(2 \mu_{k+1})^{k-j}}  \ge M_j.
	\]
	The lemma is proved.
\end{proof}

\subsection{Inclusion relations}

Let $M=(M_k)$ and $N= (N_k)$ be positive sequences.
It follows from the definition that the inclusion $\mathcal{E}^{[M]}(U) \subseteq \mathcal{E}^{[N]}(U)$ holds
for all open $U \subseteq \mathbb{R}^n$ and all $n \ge 1$, provided that the positive sequences $M$ und $N$ satisfy
\[
		\sup_{k \in \mathbb{N}} \Big(\frac{M_k}{N_k} \Big)^{1/k} < \infty.
\]
In that case we write $M \preceq N$.
Notice that the inclusion $\mathcal{E}^{(M)}(U) \subseteq \mathcal{E}^{\{M\}}(U)$ is trivially true.
It is also clear from the definition that
\[
	\lim_{k \to \infty} \Big(\frac{M_k}{N_k} \Big)^{1/k} = 0
\]
implies the inclusion $\mathcal{E}^{\{M\}}(U) \subseteq \mathcal{E}^{(N)}(U)$
for all open $U \subseteq \mathbb{R}^n$, $n \ge 1$.
We abbreviate this relation by $M \lhd N$\index{MN@$M \lhd N$}.

\begin{lemma}
	Let $M$ be a weight sequence and $N$ a positive sequence. Then:
	\begin{enumerate}
		\item The inclusion $\mathcal{E}^{[M]}(U) \subseteq \mathcal{E}^{[N]}(U)$ for all open $U \subseteq \mathbb{R}^n$, $n\ge 1$, is equivalent to $M \preceq N$.
		\item The inclusion $\mathcal{E}^{\{M\}}(U) \subseteq \mathcal{E}^{(N)}(U)$ for all open $U \subseteq \mathbb{R}^n$, $n\ge 1$, is equivalent to $M \lhd N$.
		\item The inclusion $\mathcal{E}^{(M)}(U) \subseteq \mathcal{E}^{\{N\}}(U)$ for all open $U \subseteq \mathbb{R}^n$, $n\ge 1$, is equivalent to $M \preceq N$.
	\end{enumerate}
	Actually, for the necessity of $M \preceq N$, respectively $M \lhd N$,
	it is enough to have the respective inclusion relation for some non-empty open subset $U$
	of $\mathbb{R}$.
\end{lemma}

If the smaller space is of Roumieu type, then the necessity of the various conditions follows easily from \Cref{lem:characteristic}.
That $\mathcal{E}^{(M)}(\R) \subseteq \mathcal{E}^{(N)}(\R)$ and $\mathcal{E}^{(M)}(\R) \subseteq \mathcal{E}^{\{N\}}(\R)$
imply $M \preceq N$, respectively, can be shown by an argument of \cite{Bruna80/81} based on the closed graph theorem.

Let $M$ and $N$ be weight sequences.
The lemma implies that $\mathcal{E}^{[M]} = \mathcal{E}^{[N]}$ if and only if
$M \preceq N \preceq M$. In that case we say that the weight sequences $M$ and $N$ are
\emph{equivalent}.
(Here (and below) $\mathcal{E}^{[M]} \subseteq \mathcal{E}^{[N]}$ means that the inclusion
$\mathcal{E}^{[M]}(U) \subseteq \mathcal{E}^{[N]}(U)$ holds for all open subsets $U \subseteq \mathbb R^n$, $n \ge 1$.)

In view of $C^\om(U) \cong \cE^{\{(k!)_k\}}(U)$ and $\cH(\C^n) \cong \cE^{((k!)_k)}(U)$ one easily deduces the
following corollary.

\begin{corollary}
	Let $M$ be a positive sequence. Then:
	\begin{enumerate}
		\item The inclusions $C^\omega(U) \subseteq \mathcal{E}^{\{M\}}(U)$ and
		$\mathcal{H}(\mathbb{C}^n) \subseteq \mathcal{E}^{(M)}(U)$ for all open $U \subseteq \mathbb{R}^n$, $n \ge 1$,
		are both equivalent to the condition
		\[
			\liminf_{k\to \infty} m_k^{1/k} > 0.
		\]
		\item The inclusion $C^\omega(U) \subseteq \mathcal{E}^{(M)}(U)$ for all open $U \subseteq \mathbb{R}^n$, $n \ge 1$,
		is equivalent to the condition
		\[
			\lim_{k\to \infty} m_k^{1/k} = \infty.
		\]
		In this case the inclusion $C^\omega(U) \subseteq \mathcal{E}^{(M)}(U)$ is strict.
	\end{enumerate}
\end{corollary}

\subsection{Stability properties} \label{sec:stability}

Under suitable assumptions on the weight sequence, Denjoy--Carleman classes are stable under basic operations of smooth analysis.
Let $M=(M_k)$ be a weight sequence.
\begin{description}
	\item[Pointwise multiplication] Let $U \subseteq \R^n$ be open. Then $\cE^{[M]}(U)$ forms a ring with respect to pointwise
	multiplication of functions. This follows from the Leibniz rule and \Cref{lem:atinfinity}(3).
	\item[Analytic change of variables] Suppose that $C^\om(U) \subseteq \cE^{[M]}(U)$ and $\vh : V \to U$ is real analytic, $V \subseteq \R^m$ open.
	Then the pullback $\vh^* : \cE^{[M]}(U) \to \cE^{[M]}(V)$, $\vh^*(f) = f\o \vh$, is well-defined; cf.\ \cite{FurdosNenningRainer}.
	Hence one can consider $\cE^{[M]}$-functions on real analytic manifolds.
\end{description}
For the following properties the weight sequence $M$ must satisfy additional conditions.
\begin{description}
 	\item[Stability under differentiation] $\mathcal{E}^{[M]}(U)$ is stable under differentiation
	(i.e.\ $\partial^\al \mathcal{E}^{[M]}(U) \subseteq \mathcal{E}^{[M]}(U)$ for all $\al \in \mathbb{N}^n$)
	if and only if
	\begin{equation} \label{eq:der-closed}
		\sup_{k\ge 1} \Big(\frac{M_{k+1}}{M_k}\Big)^{1/k} < \infty.
	\end{equation}
	In this case we say that $M$ is \emph{derivation-closed}.
	It is easy to see that \eqref{eq:der-closed} implies that
	there is a constant $C\ge 1$ such that $M_k \le C^{k^2}$ for all $k$.
	Actually, also the converse implication holds if $M$ is log-convex, see \cite{Mandelbrojt52}.
	\item[Stability under composition] $\cE^{[M]}$ is stable under composition (i.e.,
		$\mathcal{E}^{[M]}(U) \o \mathcal{E}^{[M]}(V,U) \subseteq \mathcal{E}^{[M]}(V)$ for all open sets
		$U \subseteq \mathbb{R}^n$ and $V \subseteq \mathbb{R}^m$) provided that
		\begin{equation}
			\label{eq:FdB}
			M^\o
			\preceq M,
		\end{equation}
		where $m_k^\o := \max\{m_j m_{\al_1} \cdots m_{\al_j} : \al_i \in \mathbb{N}_{>0}, ~ \al_1 + \cdots + \al_j = k\}$. This follows easily from
		Fa\`a di Bruno's formula; cf.\ \cite{RainerSchindl12}.\footnote{It was shown in \cite{RainerSchindl12} that \eqref{eq:FdB} is also necessary
		for the stability under composition of $\mathcal{E}^{[M]}$ provided that the class is stable by differentiation; this is based on the characterization of inverse-closedness and
		stability of $\mathcal{E}^{[M]}$ under superposition by entire functions due to \cite{Rudin62,Bruna80/81}.}
		We call \eqref{eq:FdB} the \emph{(FdB)-property}.
 \end{description}

 \begin{remark}
 	The (FdB)-property is not easy to check, but it follows from several conditions that are easier to handle.
	If $M$ is a positive sequence, then each of the following conditions implies that $M$ has the (FdB)-property:
	\begin{enumerate}
		\item $m=(m_k)$ is log-convex (i.e.\ $M$ is strongly log-convex).
		\item $m_j m_k \le m_1 m_{j+k-1}$ for all $j,k \ge 1$.
		\item $M$ is derivation-closed and the sequence $m_k^{1/k}$ is \emph{almost increasing}\index{almost increasing}, i.e.,
		there exists $C>0$ such that $m_j^{1/j} \le C m_k^{1/k}$ for all $j \le k$.
	\end{enumerate}
	Conversely, if $M$ is a weight sequence having the (FdB)-property, then
	$m_k^{1/k}$ is almost increasing; cf.\ \cite{RainerSchindl12}.
	See \cite[3.6]{RainerSchindl12} for an example of a weight sequence $M$ such that $\cE^{[M]}$ is stable under composition, but there
	is no strongly log-convex sequence $N$ that is equivalent to $M$.
 \end{remark}

 \begin{description}
 	\item[Inverse mappings] For any $\cE^{[M]}$-mapping $f : U \to V$, where $U \subseteq \R^m$ and $V \in \R^n$ are open,
 	such that $f'(x_0) \in L(\R^m,\R^n)$ is invertible at $x_0 \in U$ there exist neighborhoods $x_0 \in U_0 \subseteq U$ and
 	$f(x_0) \in V_0 \subseteq V$ and a $\cE^{[M]}$-mapping $g : V_0 \to U_0$ such that $f \o g = \on{id}_{V_0}$, provided that
 	$C^\om \subseteq \cE^{[M]}$, $M$ is derivation-closed, and $m_k^{1/k}$ is almost increasing.
 	Under these assumptions we evidently also have the implicit mapping theorem in $\cE^{[M]}$.
 	\item[Solving ODEs] For any $\cE^{[M]}$-mapping $f : \R \times \R^n \to \R^n$ the solution
 	of the initial value problem $x' = f(t,x)$, $x(0)=x_0$, is of class $\cE^{[M]}$, where it exists,
 	 provided that
 	$C^\om \subseteq \cE^{[M]}$, $M$ is derivation-closed, and $m_k^{1/k}$ is almost increasing.
 \end{description}
 It was proved in \cite{RainerSchindl14} that, given that
 $C^\om \subseteq \cE^{[M]}$ and $M$ is derivation-closed, the condition that $m_k^{1/k}$ is almost increasing is also necessary for
 stability under composition, inverse mappings, and solving ODEs, respectively.
 As a consequence $\cE^{[M]}$ is \emph{inverse-closed}, that is $1/f \in \cE^{[M]}(U)$ if $f \in \cE^{[M]}(U)$ is non-vanishing.

\subsection{Moderate growth}
Let us finish this section by briefly discussing another property which
often plays a decisive role.
We say that a weight sequence $M$ has \emph{moderate growth} if
\begin{equation} \label{eq:mg}
  \exists C>0 ~\forall j,k \in \N : M_{j+k} \le C^{j+k}M_j M_k.
\end{equation}
 	Evidently, \eqref{eq:mg} entails \eqref{eq:der-closed}. It is not hard to see that \eqref{eq:mg} is equivalent to $\mu_{k+1} \lesssim M_k^{1/k}$
 	and in turn to $\mu_{2k} \lesssim \mu_k$; cf.\ \cite[Lemma 2.2]{RainerSchindl16a}. In terms of the spaces $\cE^{[M]}$, moderate growth of $M$ is equivalent to
 	\emph{separativity} \cite{Matsumoto84}, validity of the \emph{exponential law} \cite{KMRc,KMRq,KMRu}, and
 	\emph{stability under ultradifferential operators} \cite{Komatsu73}, respectively.
  Note that \eqref{eq:mg} holds for $M$ if and only if it holds for $m$.

 The moderate growth condition is rather restrictive. Indeed, it implies that there is a constant $C>0$ such that $\mu_{2^j} \le C \mu_{2^{j-1}}$ for all $j$, and
 thus, if $2^j \le k < 2^{j+1}$,
 \[
 \mu_k \le \mu_{2^{j+1}} \le C \mu_{2^j} \le C^{j+1} \le C^{j+1}\frac{k^s}{2^{js}}
 \]
 for any $s\ge 1$. If we choose $s\ge 1$ such that the sequence $(C^{j+1}/2^{js})_j$ is bounded, we find that $M \preceq G^s$, where
  $G^s_k := k!^s$.

 \begin{example} \label{eq:Gevrey1}
  For $s\ge 1$ the sequence $G^s = (G^s_k)$ is a strongly log-convex weight sequences of moderate growth.
  It is called the
  \emph{Gevrey sequence} of index $s$.
  We have $\cE^{\{G^1\}} = C^\om$. For
  $s>1$ a typical function in $\cE^{\{G^s\}}(\R)$ is $t \mapsto \exp(- 1/t^{1/(s-1)})$.
  It is easy to check that $(k^{sk})_k$ is an equivalent weight sequence.
\end{example}

\section{The Borel map and the Denjoy--Carleman theorem}

Coming back to \Cref{Q1} we shall now discuss our problem for the singleton $\{0\}$.
Suppose that $M$ is a weight sequence and $f \in \cE^{[M]}(U)$, where $U$ is an open connected neighborhood
of $0 \in \R^n$. Then the jet $F:= j^\infty_{\{0\}}f = (f^{(\al)}(0))_{\al \in \N^n}$
obviously satisfies
\begin{equation}
	|F|^M_{\rh}  := \sup_{\al \in \N^n} \frac{|F^{\al}|}{\rh^{|\al|} M_{|\al|}} < \infty
\end{equation}
for some $\rh>0$ or for all $\rh>0$, depending on whether we consider the Roumieu case $\cE^{\{M\}}$ or the Beurling case $\cE^{(M)}$.
We define
\begin{align*}
	\La^{\{M\}}_n &:= \{a \in \C^{\N^n} : ~\exists \rh>0 : |a|^M_\rh <\infty\},
	\\
	\La^{(M)}_n &:= \{a \in \C^{\N^n} : ~\forall \rh>0 : |a|^M_\rh <\infty\},
\end{align*}
and equip them with their natural locally convex topology.
In accordance with \Cref{convention}, we use $\La^{[M]}_n$ for both of them and
just write $\La^{[M]}$
if the dimension $n$ is clear from the context.
Now the (restriction of the) \emph{Borel map}
\begin{equation} \label{eq:Borel}
	j^\infty_{\{0\}} : \cE^{[M]}(U) \to \La^{[M]}
\end{equation}
is well-defined.
We may specify \Cref{Q1} in this setting and ask:
\begin{quote}
    \emph{When is the Borel mapping \eqref{eq:Borel} surjective or injective?}
\end{quote}

We start with discussing injectivity.

\subsection{Quasianalyticity}
Let $M$ be a weight sequence.
We say that the class $\cE^{[M]}$ is \emph{quasianalytic}
if for all open connected neighborhoods $U$ of $0 \in \mathbb{R}^n$ the Borel map $j^\infty_{\{0\}} : \cE^{[M]}(U) \to \La^{[M]}$ is injective.
The class $\cE^{[M]}$ is called \emph{non-quasianalytic} if it is not quasianalytic.
In that case it is easy to construct non-trivial functions with compact support (so-called \emph{bump functions}) in the class,
since $\cE^{[M]}$ is stable by multiplication.\footnote{The study of quasianalytic classes started
at the end of the nineteenth century with the work of
\'Emile Borel who found examples of non-trivial sets $\mathcal{C}$ of smooth functions on $\mathbb{R}$ containing nowhere analytic
elements such that all members $f$ of $\mathcal{C}$ have the unique continuation property
$\big(\forall k \in \mathbb{N} : f^{(k)}(0) = 0\big) \implies f =0$.}

\begin{remark}
	There is nothing special about the origin.
	It is easy to see that $\cE^{[M]}$ is quasianalytic if and only if
	$j^\infty_{\{a\}} : \cE^{[M]}(V) \to \La^{[M]}$ is injective,
	where $V \subseteq \R^n$ is any open connected neighborhood of an arbitrary point $a$.

	Note that it suffices to check in dimension one whether a class $\cE^{[M]}$ is quasianalytic or not.
	Indeed, if there exists	a bump function $f$ in dimension one. Then $f \otimes \cdots \otimes f$ ($n$ times)
	is a bump function in dimension $n$.
	And, if a function $f$ defined in a neighborhood of $0 \in \R^n$ satisfies
	$j^\infty_{\{0\}} f =0$,
	then $j^\infty_{\{0\}} (f\o \ell) =0$ for all lines $\ell$ through $0$.
\end{remark}

In a communication \cite{Hadamard:1912aa} to the \emph{Soci\'et\'e Math\'ematique de France}
Jacques Hadamard asked whether
quasianalyticity can be characterized in terms of a growth condition on the iterated derivatives.
This was essentially confirmed a decade later by
\cite{Denjoy21} and \cite{zbMATH02599917,zbMATH02598188}
and became known as the \emph{Denjoy--Carleman theorem}.
We shall give a proof based on the approach of \cite{MR225957} and \cite[Theorem 1.3.8]{Hoermander83I}.

\subsection{The size of a function near points of flatness}

\begin{lemma} \label{lem:sizenearflat}
	Let $M=(M_k)$ be a weight sequence.
	Let $I = (-r,r) \subseteq \mathbb{R}$ for some $r > 0$.
 	Let $f \in C^\infty(I)$ be such that $f(t) = 0$ for all $t \le 0$ and
 	$\|f^{(k)}\|_I \le M_k$ for $0 \le k \le n$.
 	Then, for all $1 \le \ell \le n$ and $0\le t \le \min\big\{r, \frac{1}{4} \sum_{k = \ell}^n \frac{1}{\mu_k} \big\}$,
	\begin{equation} \label{eq:sizenearflat}
      	|f(t)| \le  \Big( \frac{2 t}{\sum_{k = \ell}^n \frac{1}{\mu_k}}\Big)^\ell.
  \end{equation}
\end{lemma}

\begin{proof}
  Cf. \cite{MR1382568} which is based on \cite{Bang46} and \cite{MR225957}.
  Fix $\ell$ with $1 \le \ell \le n$ and
  $t = t_\ell \in (0,r)$.
  Let $t_{\ell+1}> t_{\ell+2} > \cdots >  t_{n+1} = 0$ be such that the interval $I_k := [t_{k+1},t_k]$ has length
	\[
		|I_k| = \frac{a}{\mu_k}, \quad \text{ where } \quad a := \frac{t_\ell}{\sum_{k = \ell}^n \frac{1}{\mu_k}}.
	\]
	Additionally, we set $I_{n+1} := (-r,0]$.
	We define $F(s,k) = \|f^{(s)}\|_{I_k}$ and claim
	\begin{align} \label{eq:Fsk}
	 	F(s,k) \le (2 a)^{k-s} M_s \quad \text{ for } \ell \le k \le n+1 \text{ and } 0 \le s \le k.
	 \end{align}
	The claim holds for $s = k$ and for $s \le k = n+1$ by assumption,
	and the fundamental theorem of calculus gives the bound
	\[
		F(s,k) \le F(s,k+1) + |I_k| F(s+1,k).
	\]
	Let $0 \le s < k \le n$.
	If we assume by induction, that \eqref{eq:Fsk} is true for $(s,k+1)$ and $(s+1,k)$, then
	\[
		F(s,k)
		\le (2 a)^{k+1-s} M_s +  \frac{a}{\mu_k} (2 a)^{k-s-1} M_{s+1}
		\le	(2 a)^{k-s} M_s \Big( 2a + \frac12\Big),
	\]
	since $M_{s+1} = M_s \mu_{s+1} \le M_s \mu_k$.
	Provided that $2 a \le 1/2$, claim \eqref{eq:Fsk} is proved by induction.
	In particular, for $t = t_\ell \in I_\ell$ and $s = 0$, we obtain
	\(
  |f(t)| \le F(0,\ell) \le (2a)^\ell
	\)
  which is \eqref{eq:sizenearflat}.
\end{proof}

As an immediate consequence we observe that the divergence of the series
$\sum_{k} \frac{1}{\mu_k}$
is a
sufficient condition for the injectivity of $j^\infty_{\{0\}} : \cE^{\{M\}}(\R) \to \La^{\{M\}}$.
To see that it is also necessary we aim to construct a non-trivial function with compact support in
$\mathcal{E}^{\{M\}}(\mathbb{R})$ provided that the series converges.

\subsection{Special bump functions}

For $a>0$ let us consider the step function $H_a := \frac{1}{a} \mathbf 1_{(0,a)}$.
	If $f$ is a continuous function on $\mathbb{R}$, then
	\[
		f * H_a(x) = \frac{1}{a} \int_0^a f(x-t) \, dt = \frac{1}{a} \int_{x-a}^x f(t) \, dt
	\]
	is $C^1$ with derivative
	\[
		(f * H_a)'(x) = \frac{f(x) - f(x-a)}{a}.
	\]
	Hence $f * H_a$ is $C^{k+1}$ provided that $f$ is $C^k$.

\begin{lemma} \label{lem:specialbump}
	Let $(a_k)$ be a decreasing positive sequence with $a := \sum_{k = 0}^\infty a_k < \infty$.
	Then $f_k := H_{a_0} * \cdots * H_{a_k}$ is of class $C^{k-1}$ and has support in $[0,a]$.
	The sequence $(f_k)$ converges to a $C^\infty$-function $f$ with support in $[0,a]$
	such that $\int f \, dx = 1$ and
	\[
		|f^{(k)}(x)| \le \frac{1}2 \int |f^{(k+1)}(x)| \, dx \le \frac{2^k}{a_0 a_1 \cdots a_k},
	\]
	for all $k \in \mathbb{N}$ and $x \in \mathbb{R}$.
\end{lemma}

\begin{proof}
  Cf.\ \cite[Theorem 1.3.5]{Hoermander83I}.
	It is easy to see that $f_1$ is the piecewise linear function which vanishes outside $[0,a_0+a_1]$
	has slope $(a_0 a_1)^{-1}$ in $[0,a_1]$, is constant in $[a_1,a_0]$ and decreases linearly to $0$ in $[a_0,a_0+a_1]$.
	In particular, $f_1$ is continuous.

	It follows that $f_k$ is $C^{k-1}$ with support in $[0,a_0 + a_1 + \cdots + a_k]$.
	Writing $(\ta_a f)(x) = f(x-a)$ for the translation operator, we have
	\[
		f_k'(x) = H_{a_0}' * H_{a_1} * \cdots * H_{a_k} (x) = \frac{1}{a_0} (1 - \ta_{a_0}) H_{a_1} * \cdots * H_{a_k} (x)
	\]
	and iterating we find
	\[
		f_k^{(j)} = \prod_{i = 0}^{j-1} \frac{1}{a_i} (1 - \ta_{a_i}) H_{a_j} * \cdots * H_{a_k}, \quad \text{ for } j \le k-1.
	\]
	Since $\int H_{a_i} \, dx = 1$ for all $i$, it follows that
	\[
		|f_k^{(j)}| \le \frac{2^j}{a_0 a_1 \cdots a_j}  \quad \text{ and } \quad \int |f_k^{(j)}| \, dx
		\le \frac{2^j}{a_0 a_1 \cdots a_{j-1}},
	\]
	where we use $| u * v | \le \sup |u| \int |v| \, dx$ and $\int u * v \,dx = \int u  \,dx \int v \,dx$.
	We have
	\begin{align*}
		|f_{k+\ell}(x)- f_\ell(x)| &= |f_\ell * H_{a_{\ell +1}} * \cdots * H_{a_{\ell +k}}(x) - f_\ell(x)|
		\\
		&= |\int( f_\ell(x-y) - f_\ell(x)) H_{a_{\ell +1}} * \cdots * H_{a_{\ell +k}}(y) \,dy |
		\\
		&\le (a_{\ell +1} + \cdots + a_{\ell +k})  \sup |f_\ell'|
		\\
		&\le 2 \frac{a_{\ell +1} + \cdots + a_{\ell +k}}{a_0 a_1},
	\end{align*}
	since $H_{a_{\ell +1}} * \cdots * H_{a_{\ell +k}}$ has support in $[0,a_{\ell +1} + \cdots + a_{\ell +k}]$.
	It follows that $f_k$ has a uniform limit $f$. Similarly, one sees that all derivatives of $f_k$
	have uniform limits.
	Hence $f$ is $C^\infty$ and the lemma follows easily.
\end{proof}

\subsection{The Denjoy--Carleman theorem}
We are ready to characterize quasianalyticity in the Roumieu case.
In order to also include the Beurling case we first treat the following lemma.
This lemma will be very useful for reducing the Beurling to the Roumieu case
at several occasions.

\begin{lemma}[{\cite[Lemma 16]{ChaumatChollet94}}] \label{lem:CC16}
	Let $(\al_k)_{k\ge 1}$, $(\be_k)_{k\ge 1}$, and $(\ga_k)_{k\ge 1}$ be sequences with the following properties:
  \begin{enumerate}
    \item $\al_k \ge 0$ and $\sum_k \al_k<\infty$.
    \item $\be_k >0$ and $\be_k \to 0$.
    \item $\ga_k >0$ and $\ga_k \searrow 0$.
  \end{enumerate}
	Then there exists a positive sequence $(\la_k)_{k\ge 1}$ such that $\la_k \nearrow \infty$ and
  \begin{enumerate}
    \item $\sum_{j \ge k} \la_j \al_j \le 8 \la_k \sum_{j \ge k} \al_j$  for all  $k \ge 1$,
    \item $\la_k \be_k \to 0$,
    \item $\la_k \ga_k$ is decreasing.
  \end{enumerate}
\end{lemma}

\begin{proof}
	We find a strictly increasing sequence $(j_p)_{p\ge 1}$ of integers with $j_1 =1$ and
	\begin{align*}
	 	\sum_{j \ge j_{p+1}} \al_j \le \frac{1}{4} \sum_{j \ge j_{p}} \al_j, \quad
	 	\be_j \le \frac{1}{4^p} ~ \text{ if } j \ge j_{p+1}, \quad
	 	\ga_{j_{p+1}} \le \frac{1}{4} \ga_{j_p}
 	\end{align*}
 	for all $p\ge 1$.
 	Then $2^{p+1} \ga_{j_{p+1}} < 2^p \ga_{j_{p}}< 2^{p+1} \ga_{j_{p}}$. Let $k_p$ be the largest
 	integer $k$ satisfying $j_p \le k < j_{p+1}$ and $2^p \ga_{j_{p}}\le 2^{p+1} \ga_{k}$.
  Define $(\la_k)_{k\ge 1}$ by setting
 	\[
 		\la_k :=
 		\begin{cases}
 			2^p \frac{\ga_{j_p}}{\ga_{k}} &\text{ if } j_p \le k \le k_{p},
 			\\
 			2^{p+1} & \text{ if } k_p < k \le j_{p+1}.
 		\end{cases}
 	\]
 	It is clear that $\la_k \nearrow \infty$ and that $\la_k\ga_k$ is decreasing.
 	If $j_{p+1} \le k \le j_{p+2}$, then $2^{p+1} \le \la_k \le 2^{p+2}$ and hence
 	$\la_k \be_k \le  2^{-p+2}$, whence $\la_k \be_k \to 0$.
  Finally,
 	\begin{align*}
 		\sum_{j \ge j_{p+1}} \la_j \al_j
 		=  \sum_{n = 0}^\infty \sum_{j= j_{p+1+n}}^{j_{p+ 2 + n}-1} \la_j \al_j
 	  &\le  \sum_{n = 0}^\infty  2^{p+2 +n}\sum_{j= j_{p+1+n}}^{j_{p+2+n}-1}  \al_j
 		\\
    &\le  \sum_{n = 0}^\infty  2^{p+2+n} \frac{1}{4^{n}} \sum_{j\ge j_{p+1}}  \al_j
 	  = 2^{p+3}  \sum_{j\ge j_{p+1}}  \al_j.
 	\end{align*}
  For $k \ge 1$ take $p$ such that $j_p \le k < j_{p+1}$.
  Then
 	\begin{align*}
 		\sum_{j \ge k} \la_j \al_j
 		&= \sum_{j= k}^{j_{p+1}-1} \la_j \al_j + \sum_{j \ge j_{p+1}} \la_j \al_j
 		\\
    &\le 2^{p+1} \sum_{j= k}^{j_{p+1}-1} \al_j + 2^{p+3}  \sum_{j\ge j_{p+1}}  \al_j
 		\le 8 \cdot 2^{p} \sum_{j\ge  k} \al_j
 		\le 8 \la_k \sum_{j\ge  k} \al_j.
 	\end{align*}
 	The proof is complete.
\end{proof}

\begin{corollary} \label{lem:smaller}
	Let $(\al_k)$ be a decreasing positive sequence with $\sum_k a_k < \infty$.
	There exists a decreasing positive sequence $(\be_k)$ such that $\sum_k \be_k < \infty$
	and $\be_k/\al_k \nearrow \infty$.
\end{corollary}

\begin{proof}
	Apply \Cref{lem:CC16} to $\al_k = \be_k = \ga_k$. Then $\be_k:= \la_k \al_k$ has the required properties.
\end{proof}

\begin{theorem}[Denjoy--Carleman theorem] \label{thm:DC}
	Let $M=(M_k)$ be a weight sequence. Then the following conditions are equivalent:
	\begin{enumerate}
			\item $\cE^{\{M\}}$ is quasianalytic.
			\item $\cE^{(M)}$ is quasianalytic.
			\item $\sum_{k} \frac{1}{\mu_k} = \infty$.
	\end{enumerate}
\end{theorem}

\begin{proof}
  The equivalence of (1) and (3) is an easy consequence of \Cref{lem:sizenearflat} and
  \Cref{lem:specialbump} (with $a_k=\mu_k^{-1}$).

If $\cE^{(M)}$ contains non-trivial functions of compact support then so does $\cE^{\{M\}}$,
by the trivial inclusion $\cE^{(M)} \subseteq \cE^{\{M\}}$, which shows (3) $\Rightarrow$ (2)
since we already have the equivalence of (1) and (3).

In order to show that
(2) implies (3) suppose that $\sum_{k} \frac{1}{\mu_k} < \infty$.
By \Cref{lem:smaller}, we find an increasing sequence $(\nu_k)$ such that $\sum_k \frac{1}{\nu_k}  < \infty$
and $\frac{\mu_k}{\nu_k} \nearrow \infty$. So the sequence $\frac{M_k}{N_k}$ is log-convex and thus
$(\frac{M_k}{N_k})^{1/k} \to \infty$ since $\frac{\mu_k}{\nu_k} \to \infty$ (cf.\ \Cref{lem:atinfinity}(4)).
Then $N=(N_k)$ is a weight sequence satisfying $N \lhd M$ and thus $\cE^{\{N\}} \subseteq \cE^{(M)}$.
Since $\sum_k \frac{1}{\nu_k}  < \infty$, we find that $\cE^{\{N\}}$ is non-quasianalytic, and hence so is $\cE^{(M)}$.
\end{proof}

The Denjoy--Carleman theorem completely characterizes injectivity of the Borel map \eqref{eq:Borel}.
It shows that the following definition makes sense.

\begin{definition}
	A weight sequence $M=(M_k)$ is called \emph{quasianalytic} if
	the equivalent conditions of \Cref{thm:DC} hold, and it is called
	\emph{non-quasianalytic} otherwise.
\end{definition}

\begin{remark} \label{rem:reduction}
	For later reference note that the proof of the implication (2) $\Rightarrow$ (3) in \Cref{thm:DC} shows the following:
	For any non-quasianalytic weight sequence $M$ there exists a non-quasianalytic weight sequence $N$ with $N \lhd M$.
\end{remark}

\begin{remark} \label{lem:Carleman}
  For a weight sequence $M=(M_k)$, the divergence of the series $\sum_k \frac{1}{\mu_k}$ is
  equivalent to the divergence of $\sum_k \frac{1}{M_k^{1/k}}$.
  This follows from $M_k^{1/k} \le \mu_k$ for all $k$ (cf.\ \Cref{lem:atinfinity}(1))
  and \emph{Carleman's inequality}:
  	 If $(a_k)$ is a positive sequence, then
  	 \[
  	 	\sum_{n=1}^\infty (a_1 a_2 \cdots a_n)^{1/n} \le e \sum_{k=1}^\infty a_k.
  	 \]
  Cf.\ \cite[Lemma 1.3.9]{Hoermander83I}.

  If $M$ is \emph{not} necessarily a weight sequence, then one can work with the \emph{log-convex minorant}
  \[
  \ul M_k := \inf\{M_k, M_j^{\frac{\ell-k}{\ell-j}} M_\ell^{\frac{k-j}{\ell-j}} : j <k < \ell\}
  \] of $M$
  or the \emph{increasing minorant}
  \[
  u_k :=\inf_{k \ge j} M_k^{1/k}
  \]
  of $M_k^{1/k}$. In fact, the following are equivalent (see \cite[Theorem 1.3.8]{Hoermander83I}):
  \begin{enumerate}
    \item $\cE^{[M]}$ is quasianalytic.
    \item $\sum_k \frac{1}{u_k} = \infty$.
    \item $\sum_k \frac{1}{\ul M_k^{1/k}} = \infty$.
    \item $\sum_k \frac{1}{\ul \mu_k} = \infty$.
  \end{enumerate}
\end{remark}

\begin{example} \label{ex:Denjoy}
  (1) The Gevrey sequences $G^s$, $s\ge 1$, are non-quasianalytic if and only if $s>1$.

  (2) The \emph{Denjoy sequences}
    \[
      Q^{n,\de}_k = k^k \prod_{j=1}^{n-1} \big(\log^{[j]}(k + \exp^{[j]}(1))\big)^{k}
      \big(\log^{[n]}(k + \exp^{[n]}(1))\big)^{\de k},
    \]
    where $n \in \N_{\ge 1}$, $0< \de \le 1$,
    and $f^{[n]}$ is the $n$-fold composite of $f$, are quasianalytic strongly log-convex weight sequences
    of moderate growth.
\end{example}

\subsection{Non-quasianalytic cutoff functions}

Before we study surjectivity of the Borel map \eqref{eq:Borel} in the next section
we show that we may use \Cref{lem:specialbump} to construct non-quasianalytic
cutoff functions in all dimensions.

\begin{proposition}[{\cite[Theorem 1.4.2]{Hoermander83I}}] \label{prop:cutoff}
	Let $U \subseteq \mathbb{R}^n$ be open and $K \subseteq U$ compact.
	Let $\| \cdot \|$ be any norm on $\mathbb{R}^n$.
	If $d = \inf\{ \| x-y\| : x \in K,~ y \in \mathbb{R}^n\setminus U\}$
	and $(d_j)$ is a positive decreasing sequence with
	$\sum_{j=1}^\infty d_j < d$, then
	there exists $\ph \in C^\infty_c(U)$ with $0 \le \ph \le 1$ such that $\ph =1$ on a neighborhood
	of $K$ and
	\[
		|\ph^{(k)}(x)(v_1,\ldots,v_k)|
		\le C(n)^k \frac{\|v_1\| \|v_2\| \cdots \|v_k\|}{d_1 d_2 \cdots d_k}, \quad k \ge 1, ~ x \in U.
	\]
\end{proposition}

\begin{proof}
	Let us first consider the norm $\|x\| = \max_i |x_i|$.
	Let $f$ be the function from \Cref{lem:specialbump} (with $a_j = d_{j+1}$)
	and set $h(t) := f(t + \frac{1}{2} \sum_{j=1}^\infty d_j)$.
	Then $t \in \on{supp} h$ implies $|t| \le \frac{1}{2} \sum_{j=1}^\infty d_j$,
	we have $\int h \, dt = 1$, and
	\[
	\int |h^{(j)}(t)| \,dt \le \frac{2^j}{d_1 d_2 \cdots d_j}.
	\]
	Let $\ch(x) := h(x_1) h(x_2) \cdots h(x_n)$ which satisfies $\int \ch \, dx =1$
	and set $\ph := \ch * \mathbf 1_{K_{d/2}}$,
	where $K_{d/2} := \{y\in \R^n : \|x-y\| \le d/2 \text{ for some } x \in K\}$.
	Then $\ph  \in C^\infty_c(U)$ is as required, since
	\[
		|\ph^{(\al)}| \le  \int |\ch^{(\al)}(x)| \,dx \le \frac{2^{|\al|}}{d_1 d_2 \cdots d_{|\al|}}
	\]
	as $(d_k)$ is decreasing.
	The statement for arbitrary norms follows from the fact that any two norms on $\mathbb R^n$ are
	equivalent with constants only depending on $n$.
\end{proof}

\begin{corollary} \label{cor:cutoff}
	Let $M$ be a non-quasianalytic weight sequence.
	Let $U \subseteq \mathbb{R}^n$ be open and $K \subseteq U$ compact.
	Then there exists a function $\ph \in \mathcal{E}^{[M]}(U)$ with support contained in $U$ such that
	$\ph = 1$ in a neighborhood of $K$.
\end{corollary}

\begin{proof}
	Let $d := \on{dist}(K,\mathbb R^n \setminus U)$\footnote{The distance
	$\on{dist}(A,B) := \inf\{|a-b| : a \in A, b \in B\}$ of two subsets of $\R^n$ is defined with respect
	to the Euclidean norm.}
	and $d_j := \frac{d}{s\mu_j}$, where $s := \sum_{j\ge 1} \frac{1}{\mu_j} < \infty$.
	Then the proposition provides a function $\ph \in \mathcal{E}^{\{M\}}(U)$ with the desired properties.
	Using \Cref{rem:reduction},
	we may achieve that $\ph \in \mathcal{E}^{(M)}(U)$.
\end{proof}

That non-quasianalytic functions are ``abundant'' is witnessed by the following remark.

\begin{remark}
  Let $M$ be a non-quasianalytic weight sequence.
  For each closed subset $A \subseteq \R^n$ there exists $f \in \mathcal{B}^{[M]}(\mathbb{R}^n)$
  such that $f$ is flat on $A$, i.e., $j^\infty_A f =0$, and
  $f|_{\R^n \setminus A}$ is positive and real analytic. Cf. \cite{Langenbruch:1999aa,Langenbruch03}.
\end{remark}

\section{Surjectivity of the Borel map}

Let $M=(M_k)$ be a weight sequence. We shall see that a necessary and sufficient condition for
the surjectivity of the Borel map $j^\infty_{\{0\}} : \cE^{[M]}(\R) \to \La^{[M]}$ is
\begin{equation} \label{eq:snq}
				\sup_{k} \frac{\mu_k}{k} \sum_{j \ge k} \frac{1}{\mu_j} < \infty.
\end{equation}
In the Beurling case the condition even guarantees that the Borel map is split surjective.
This section is primarily based on \cite{Petzsche88}.

\begin{definition}
	A weight sequence $M$ satisfying \eqref{eq:snq} is said to be \emph{strongly non-quasianalytic}.\footnote{Many authors use Komatsu's (M.$\ul{\hspace{2mm}}$) notation \cite{Komatsu73}:
  log-convexity (M.1),
  moderate growth (M.2),
  strong non-quasianalyticity (M.3),
  derivation closedness (M.2)',
  and
  non-quasianalyticity (M.3)'.}
\end{definition}

Obviously a strongly non-quasianalytic weight sequence $M$ is non-quasianalytic, and thus
$\cE^{[M]}$ admits cutoff functions. We will see that \eqref{eq:snq} allows for
special cutoff functions that are in a certain sense optimal; see \Cref{sec:cutoff}.

\begin{example}
	The Gevrey sequences $G^s$ are strongly non-quasianalytic provided that $s>1$.
	Indeed, $\ga^s_k = G^s_k/G^s_{k-1} = k^s$ and
	\[
		\sum_{j \ge k} \frac{1}{j^s} \le \int_{k-1}^\infty \frac{1}{t^s}\, dt
		= \frac{1}{s-1} \frac{1}{(k-1)^{s-1}} = \frac{1}{s-1} \Big(\frac{k}{k-1}\Big)^{s-1} \frac{k}{k^s}
		\le  \frac{2^{s-1}}{s-1} \frac{k}{k^s}.
	\]
\end{example}

\subsection{Necessity}
Let us assume throughout this section that $M$ is non-quasianalytic.
The quasianalytic case which requires different methods will be discussed in \Cref{sec:quasianalytic}.

\begin{theorem} \label{thm:necessityB}
	Let $M$ be a non-quasianalytic weight sequence.
  If the Borel map $j^\infty_{\{0\}} : \cE^{(M)}(\R) \to \La^{(M)}$ or
  $j^\infty_{\{0\}} : \cE^{\{M\}}(\R) \to \La^{\{M\}}$ is surjective, then $M$ is strongly non-quasianalytic.
\end{theorem}

\begin{proof}
  We give details in the Beurling case and indicate the required changes for the Roumieu case.

	We may suppose that $j^\infty_{\{0\}} : \cB^{(M)}((-1,1)) \to \La^{(M)}$ is surjective,
	by multiplication with a suitable cutoff function.
	Then the open mapping theorem for Fr\'echet spaces
	(cf.\ \cite[24.30]{MeiseVogt97}) implies that
	there exist constants $C,\rh>0$ and functions $f_k \in \cB^{(M)}((-1,1))$ such that $f_k^{(j)}(0) = \de_{jk}$ and
	\begin{equation} \label{eq:nec0}
		\|f_k\|^M_{[-1,1],1} \le \frac{C}{\rh^k M_k}, \quad k \in \N.
	\end{equation}
	Let $t_k := \inf\{t \in [0,1] : f_k^{(k)}(t) < \frac{1}{2}\}$. Then $f_k^{(k)}(t)\ge \frac{1}{2}$ for $t \in [0,t_k]$ and hence
	\(
		f_{2k}^{(k)}(t) \ge \frac{t^k}{2 k!}
	\)
  for $t \in [0,t_{2k}]$.
	In particular, for $t=t_{2k}$ we find, by \eqref{eq:nec0},
	\(
		\frac{t_{2k}^k}{2 k!} \le f_{2k}^{(k)}(t_{2k}) \le \frac{C M_k}{\rh^{2k} M_{2k}}
	\)
	which implies (using that $\mu_k$ is increasing)
	\begin{equation} \label{eq:nec1}
		t_{2k} \le \frac{(2C)^{1/k}}{\rh^2} \frac{k}{\mu_k}.
	\end{equation}
	We claim that there exists $0<h< \min\{\tfrac{1}{4}, \tfrac{\rh}{4}\}$ such that
	\begin{equation} \label{eq:nec2}
		t_{k} \ge h \sum_{j\ge 2k} \frac{1}{\mu_j} \quad \text{ for all sufficiently large } k.
	\end{equation}
	Then \eqref{eq:nec1} and \eqref{eq:nec2} imply \eqref{eq:snq},
	\begin{align*}
		\sum_{j\ge k} \frac{1}{\mu_j} = \sum_{j= k}^{4k-1} \frac{1}{\mu_j} + \sum_{j\ge 4k} \frac{1}{\mu_j}
		\le \frac{3k}{\mu_k} + \frac{t_{2k}}{h}	\lesssim \frac{k}{\mu_k}.
	\end{align*}
	It remains to show \eqref{eq:nec2}. To this end we apply \Cref{lem:sizenearflat} to the functions
	\[
		g_k(t) :=
		\begin{cases}
			\big(\frac{C}{\rh^k} + 1\big)^{-1} (f_k^{(k)}(t) -1) & \text{ if } t\ge 0,
			\\
			0 & \text{ if } t<0.
		\end{cases}
	\]
  Indeed, each $g_k$ is $C^\infty$, satisfies
	$|g_k| \le 1$, and for $j\ge 1$ we have
	\[
		|g_k^{(j)}| = \Big(\frac{C}{\rh^k} + 1\Big)^{-1}  |f_k^{(k+j)}| \le \Big(\frac{C}{\rh^k} + 1\Big)^{-1} \frac{C}{\rh^k}  \frac{M_{k+j}}{M_k}\le \mu_{k+1}\cdots \mu_{k+j}=:M^{+k}_j,
	\]
  by \eqref{eq:nec0}.
	Then $M^{+k}$ (with $M^{+k}_{0}:=1$) is a weight sequence and \Cref{lem:sizenearflat} gives
	\[
		|g_k(t)| \le \Big(\frac{2t}{\sum_{j\ge 2k} \frac{1}{\mu_j}}\Big)^k, \quad 0\le t \le \frac{1}{4}\sum_{j\ge 2k} \frac{1}{\mu_j}.
	\]
	Suppose for contradiction that \eqref{eq:nec2} is violated for infinitely many $k$.
	For those $k$ we obtain
	\[
		\frac{1}{2} \frac{\rh^k}{C+\rh^k} = |g_k(t_k)| \le \Big(\frac{2t_k}{\sum_{j\ge 2k} \frac{1}{\mu_j}}\Big)^k \le (2h)^k,
	\]
	which is however only possible for finitely many $k$. Thus \eqref{eq:nec2} is proved.

  Now suppose that $j^\infty_{\{0\}} : \cB^{\{M\}}((-1,1)) \to \La^{\{M\}}$ is surjective.
  By the Grothendieck factorization theorem
  (cf.\ \cite[24.33]{MeiseVogt97}),
  there exist $C,\rh>0$ with
  \[
    \{(a_k) : |(a_k)|^M_1 \le 1\} \subseteq j^\infty_{\{0\}} \{ f : \|f\|^M_{[-1,1],\rh} \le C\}.
  \]
  Thus there exist $f_k \in \cB^{\{M\}}((-1,1))$ such that $f_k^{(j)}(0) = \de_{jk}$ and
  \[
    \|f_k\|^M_{[-1,1],\rh} \le \frac{C}{M_k}, \quad k \in \N.
  \]
  It is easy to modify the above proof to conclude again \eqref{eq:snq}.
\end{proof}

\subsection{Sufficiency}

\begin{theorem} \label{thm:Borelsufficient}\label{cor:extopsingleton}
	Let $M$ be a strongly non-quasianalytic weight sequence. Then:
  \begin{enumerate}
    \item $j^\infty_{\{0\}} : \cE^{\{M\}}(\R) \to \La^{\{M\}}$ is surjective.
    \item $j^\infty_{\{0\}} : \cE^{(M)}(\R) \to \La^{(M)}$ is \emph{split surjective}, i.e., there exists a continuous linear right-inverse
    $E : \La^{(M)} \to \cE^{(M)}(\R)$ with $j^\infty_{\{0\}} \o E = \on{id}_{\La^{(M)}}$.
  \end{enumerate}
\end{theorem}

\begin{remark} \label{rem:extopBorelRoumieu}
	By \cite{Petzsche88} (see also \cite[p.223]{SchmetsValdivia00}),
	the existence of a continuous linear right-inverse for
		$j^\infty_{\{0\}} : \mathcal{E}^{\{M\}}(\R) \to \Lambda^{\{M\}}$ is equivalent to \eqref{eq:snq} and additionally
		\begin{equation} \label{eq:PetzscheRI}
      \forall \ep>0 ~\exists a \in \mathbb{N}_{>1} :
			\limsup_{k \to \infty} \Big(\frac{M_{ak}}{M_k}\Big)^{\frac{1}{k(a-1)}} \frac{1}{\mu_{ak}} \le \ep.
    \end{equation}

	This condition is not easy to satisfy; for instance Gevrey sequences fail it.
	Indeed
	\[
		\Big(\frac{G^s_{ak}}{G^s_k}\Big)^{\frac{1}{k(a-1)}} \frac{1}{\ga^s_{ak}}
		\ge \big((k(a-1))!^s\big)^{\frac{1}{k(a-1)}} \frac{1}{(ak)^s}
		\ge \Big(\frac{a-1}{ea}\Big)^s.
	\]
	See also \Cref{sec:extopRoumieu}.
\end{remark}

Let us concentrate on the more interesting Beurling case (2) in \Cref{thm:Borelsufficient}.
We have to develop some tools before we can give the proof in \Cref{sec:proofBorelBeurlingsufficient}.

\subsection{The descendant of a non-quasianalytic weight sequence}

A strongly non-quasianalytic weight sequence is equivalent to a strongly log-convex sequence, as will be shown in the next lemma.
More generally, we may associate with any non-quasianalytic
weight sequence $N$ a weight sequence $S$ with many good properties.
The following construction was used in \cite{Rainer:2019ac}; it is inspired by \cite[Proposition 1.1]{Petzsche88}.

\begin{definition}
  Let $\nu=(\nu_k)$ be an increasing positive sequence with $\nu_0 =1$ and  $\sum_k \frac{1}{\nu_k} < \infty$.
  The \emph{descendant} of $\nu$ is the positive sequence $\si = \si(\nu)$ defined by setting $\si_0 := 1$,
  \begin{equation} \label{tau}
    \ta_k := \frac{k}{\nu_k} + \sum_{j\ge k} \frac 1 {\nu_j}, \quad k \ge 1,
  \end{equation}
and
  \begin{equation} \label{sigma}
    \si_k := \frac{\ta_1 k}{\ta_k}, \quad k \ge 1.
  \end{equation}
We shall also say that $S_k=\si_0 \si_1 \cdots \si_k$ is the descendant of $N_k = \nu_0 \nu_1 \cdots \nu_k$.
\end{definition}

\begin{lemma} \label{lem:log-convex}
  Let $\si$ be the descendant of $\nu$. Then:
  \begin{enumerate}
    \item $\sup_k \frac{\si_k}{\nu_k} < \infty$.
    \item $\sup_{k} \frac{\si_k}k \sum_{j\ge k}  \frac{1}{\nu_j} < \infty$.
    \item $1 \le  \frac{\si_k}k =: \si^* \nearrow \infty$.
    \item $\si$ is maximal among the sequences satisfying (1) and (2):
    If $\mu$ is an increasing positive sequence with $\sup_k \frac{\mu_k}{\nu_k} < \infty$
    and $\sup_k \frac{\mu_k}k \sum_{j\ge k}  \frac{1}{\nu_j} < \infty$, then $\sup_k \frac{\mu_k}{\si_k} < \infty$.
  \end{enumerate}
\end{lemma}

\begin{proof}
  (1), (2), and (3) are immediate.
  The assumptions on $\mu$ in (4) imply
  \(
    \ta_k \lesssim \frac{k}{\nu_k} +  \frac{k}{\mu_k} \lesssim \frac{k}{\mu_k}
  \)
  and hence $\mu_k \lesssim \si_k$.
\end{proof}

 Note that, by (3), the descendant $S$ of $N$ is a strongly log-convex weight sequence.

\begin{corollary} \label{cor:log-convex}
	The descendant $\si$ of $\nu$ satisfies
	\begin{equation} \label{eq:qequivalent}
		\exists C>0 ~\forall k : C^{-1} \le \frac{\si_k}{\nu_k} \le C
	\end{equation}
	if and only if $N$ is strongly non-quasianalytic.
	In that case the weight sequences $S$ and $N$ are equivalent.
\end{corollary}

\begin{proof}
	This follows from (1), (2), and (4) in \Cref{lem:log-convex}.
\end{proof}

\begin{lemma} \label{lem:Sepsilon}
	Let $N=(N_k)$ be a strongly non-quasianalytic strongly log-convex weight sequence and $S$ its descendant.
  Then,
  for all small $\ep >0$, $\tilde S := (S_k/k!^\ep)$ is a strongly non-quasianalytic weight sequence.
\end{lemma}

\begin{proof}
	Note that $\nu_k^* := \frac{\nu_k}k$ is increasing. So, for $\si_k^*  := \frac{\si_k}{k}$ and $k\ge 1$,
	\begin{align*}
    \frac{\si_{2k}^*}{\si_{k}^*} = \frac{\ta_k}{\ta_{2k}}
	 	=\frac{\frac{1}{\nu_k^*} + \sum_{j\ge k} \frac 1 {\nu_j}}{\frac{1}{\nu_{2k}^*} + \sum_{j\ge 2k} \frac 1 {\nu_j}}
	 	\ge
	 	\frac{\frac{1}{\nu_{2k}^*} + \sum_{j\ge k} \frac 1 {\nu_j}}{\frac{1}{\nu_{2k}^*} + \sum_{j\ge 2k} \frac 1 {\nu_j}}
	 	= 1+  \frac{\sum_{k \le j<  2k} \frac 1 {\nu_j}}{\frac{1}{\nu_{2k}^*} + \sum_{j\ge 2k} \frac 1 {\nu_j}}.
	\end{align*}
	Since $\sum_{k \le j<  2k} \frac 1 {\nu_j} \ge \sum_{k \le j<  2k} \frac 1 {\nu_{2k}} = \frac k {\nu_{2k}}$
	and $\frac{1}{\nu_{2k}^*} + \sum_{j\ge 2k} \frac 1 {\nu_j} \lesssim \frac{1}{\nu_{2k}^*}$, we find that
	\[
		q:=\inf_{k\ge 1} \frac{\si_{2k}^*}{\si_{k}^*} > 1.
	\]
	Suppose that $\ep\in (0,1)$ is such that $q>2^{\ep}$.
	Then $\tilde \si_k^* := \frac{\tilde \si_k}{k} =  \frac{\si_k^*}{k^{\ep}}$ satisfies
	\[
		\tilde q:= \inf_{k\ge 1} \frac{\tilde \si_{2k}^*}{\tilde \si_{k}^*} = \inf_{k\ge 1} \frac{\si_{2k}^*}{\si_{k}^*} \cdot \frac{1}{2^\ep}> 1.
	\]
  The fact that $\si^*_k = \frac{\si_k}{k}$ is increasing implies that $\tilde \si_k$ is increasing. Thus,
	\begin{align*}
		\sum_{j \ge k} \frac{1}{\tilde \si_j}
		= \sum_{\ell =0}^\infty \sum_{j=2^\ell k}^{2^{\ell+1} k-1} \frac{1}{\tilde \si_j}
		\le  \sum_{\ell =0}^\infty  \frac{2^\ell k}{\tilde \si_{2^\ell k}}
    = \sum_{\ell =0}^\infty  \frac{1}{\tilde \si_{2^\ell k}^*}
		\le  \frac{1}{\tilde \si_k^*} \sum_{\ell =0}^\infty \tilde q^{-\ell}.
	\end{align*}
	It follows that $\tilde S$ is a strongly non-quasianalytic weight sequence.
\end{proof}

\begin{corollary} \label{cor:setupcutoff}
  Let $M=(M_k)$ be a strongly non-quasianalytic weight sequence.
  We may assume that $\tilde \mu := (\mu_k/k^\ep)$ is increasing to infinity and satisfies
	\eqref{eq:snq} if $\ep>0$ is small enough. Hence
	the sequence
	\begin{equation} \label{eq:modifiedseq}
		\underbrace{\mu_k,\ldots,\mu_k}_{k \text{ times}}, \mu_{k+1} \big(\tfrac{k}{k+1}\big)^\ep,\mu_{k+2} \big(\tfrac{k}{k+2}\big)^\ep, \ldots
	\end{equation}
	is increasing and for the sum of the reciprocals we have
	\begin{equation} \label{eq:cutoffk}
		\frac{k}{\mu_k} + \sum_{j \ge k+1} \frac{1}{\mu_j} \Big(\frac{j}{k}\Big)^\ep
		\le A \frac{k}{\mu_k}, \quad  k \ge 1,
	\end{equation}
	for some constant $A>0$. \qed
\end{corollary}

\subsection{Proof of \Cref{thm:Borelsufficient}} \label{sec:proofBorelBeurlingsufficient}

  We give a detailed proof of (2).

\begin{claim*}
  For each $k \in \N$, there exists $\ps_k \in \cD^{(M)}(\R)$ such that $\ps_k^{(j)}(0) = \de_{jk}$
	and
  \begin{equation} \label{eq:claimsur}
    \|\ps_k\|^M_{\R,\rh} \le \frac{C_\rh H_\rh^k}{M_k} \quad \text{ for all small } \rh>0.
  \end{equation}
\end{claim*}

	Then $E : \La^{(M)} \to \cE^{(M)}(\R)$ defined by
	\(
		E: a = (a_k) \mapsto \sum_{k\ge 0} a_k \ps_k
	\)
  is the required right inverse of $j^\infty_{\{0\}}$.
	Indeed,
	for every $\si>0$ we have $|a|^M_\si < \infty$. Let $\rh>0$ (small) be given. Then
	\begin{align*}
		\Big| \sum_{k\ge 0} a_k \ps_k^{(j)} \Big| &\le \sum_{k\ge 0} |a_k| |\ps_k^{(j)}|
		\\
		&\le  \sum_{k\ge 0} |a|^M_\si\, \si^k M_k \cdot  \|\ps_k\|^M_{\R,\rh}\, \rh^j M_j
		\\
		&\le  \sum_{k\ge 0} |a|^M_\si \si^k M_k \cdot \frac{C_\rh H_\rh^k}{M_k} \rh^j M_j
		= |a|^M_\si  \cdot C_\rh \rh^j M_j \cdot \sum_{k\ge 0}  (\si H_\rh)^k.
	\end{align*}
	We may choose $\si=\si(\rh)>0$
	such that \(\sum_{k\ge 0}  (\si H_\rh)^k\) converges.
	It follows that $f := E(a)=\sum_{k\ge 0} a_k \ps_k$ defines a $C^\infty$-function $f$ on $\R$ such that
	for all $\rh>0$ there exist $\si>0$ and $C>0$ such that
	\begin{equation} \label{eq:extop}
		\|f\|^M_{\R,\rh} \le C  |a|^M_\si,
	\end{equation}
	in particular, $f \in \cE^{(M)}(\R)$. So the linear map $E$ is continuous and $\Bmap E(a) = a$ is clear from $\ps_k^{(j)}(0) = \de_{jk}$.

  Let us prove the claim.
	We apply \Cref{prop:cutoff} to the setup in \Cref{cor:setupcutoff} (for $K=\{0\}$ and $\sum d_j$ the left-hand side of \eqref{eq:cutoffk}).
  So for each $k \ge 1$ there is a $C^\infty$-function $\vh_k$ with $0 \le \vh_k \le 1$, $\on{supp} \vh_k \subseteq (-A \frac{k}{\mu_k},A \frac{k}{\mu_k})$,
	$\vh_k(0)=1$, $\vh_k^{(j)}(0) = 0$ for all $j \ge 1$, and
	\begin{align*}
	 	|\vh_k^{(j)}| \le
	 	\begin{cases}
	 		2^j \mu_k^j & \text{ if } 0 \le j \le k,
	 		\\
	 		2^j \mu_k^k \frac{M_j}{M_k} \Big( \frac{k^{j-k} k!}{j!}\Big)^\ep & \text{ if } j>k.
	 	\end{cases}
	\end{align*}
	Then $\ps_k(t):= \vh_k(t) \frac{t^k}{k!}$ clearly satisfies $\ps_k^{(j)}(0) = \de_{jk}$.
	It remains to show \eqref{eq:claimsur}.
	If $0 \le j \le k$, then
	\begin{align} \label{eq:est1}
		|\ps_k^{(j)}| &\le \sum_{i =0}^j \binom{j}{i} 2^i \mu_k^i \Big(A \frac{k}{\mu_k}\Big)^{k-j+i} \frac{1}{(k-j+i)!}
		 \notag \\
		 &\le \frac{M_j}{M_k}  \frac{\mu_{j+1} \cdots \mu_k}{\mu_k^{k-j}} (Ae)^k \sum_{i =0}^j \binom{j}{i} 2^i  A^{-j+i}
		 \le  \frac{M_j}{M_k} (Ae)^k (2 + A^{-1})^j.
	\end{align}
	If $j\ge 2 k$, then
	\begin{align} \label{eq:est2}
		|\ps_k^{(j)}| &\le \sum_{i =j-k}^j \binom{j}{i} 2^i \mu_k^k \frac{M_i}{M_k} \Big(\frac{k^{i-k} k!}{i!}\Big)^\ep \Big(A \frac{k}{\mu_k}\Big)^{k-j+i} \frac{1}{(k-j+i)!}
		 \notag \\
		 &\le \frac{M_j}{M_k}  \frac{\mu_k^{j-i}\mu_{k+1} \cdots \mu_i}{\mu_{k+1} \cdots \mu_j} (Ae)^k \Big(\frac{k^{j-2k} k!}{(j-k)!}\Big)^\ep \sum_{i =j-k}^j \binom{j}{i} 2^i A^{-j+i}
		 \notag \\
		 &\le \frac{M_j}{M_k} (Ae)^k  (2 + A^{-1})^j \Big(\frac{k^{j-2k} }{(j-2k)!}\Big)^\ep.
	\end{align}
	If $k < j < 2k$, then (since $\frac{k^{i-k} k!}{i!} = \frac{k^i}{i!} \frac{k!}{k^k} \le e^k$)
	\begin{align} \label{eq:est3}
		|\ps_k^{(j)}| &\le \sum_{i =0}^k \binom{j}{i} 2^i \mu_k^i \Big(A \frac{k}{\mu_k}\Big)^{k-j+i} \frac{1}{(k-j+i)!}
		\notag \\
		&\quad +
		\sum_{i =k+1}^j \binom{j}{i} 2^i \mu_k^k \frac{M_i}{M_k} \Big(\frac{k^{i-k} k!}{i!}\Big)^\ep \Big(A \frac{k}{\mu_k}\Big)^{k-j+i} \frac{1}{(k-j+i)!}
		 \notag \\
		 &\le \frac{M_j}{M_k} (Ae)^k  (2 + A^{-1})^j +
		\frac{M_j}{M_k} (Ae^2)^k  (2 + A^{-1})^j
		 \notag \\
		 &\le \frac{M_j}{M_k} 2(Ae^2)^k  (2 + A^{-1})^j.
	\end{align}
	Let $0<\rh < 2 + A^{-1}$ be fixed. For $j < 2k$, we get from \eqref{eq:est1} and \eqref{eq:est3} that
	\[
		\frac{|\ps_k^{(j)}|}{\rh^j M_j} \le \frac{1}{M_k} 2(Ae^2)^k  \Big(\frac{2 + A^{-1}}{\rh}\Big)^j
    \le \frac{1}{M_k} 2(Ae^2)^k  \Big(\frac{2 + A^{-1}}{\rh}\Big)^{2k}
	\]
	and, for $j \ge 2k$,  \eqref{eq:est2} gives
	\begin{align*}
		\frac{|\ps_k^{(j)}|}{\rh^j M_j}
		&\le \frac{1}{M_k} (Ae)^k  \Big(\frac{2 + A^{-1}}{\rh}\Big)^j \Big(\frac{k^{j-2k} }{(j-2k)!}\Big)^\ep
		\\
		&= \frac{1}{M_k} (Ae)^k \Big(\frac{2 + A^{-1}}{\rh}\Big)^{2k}  \Big(\frac{2 + A^{-1}}{\rh}\Big)^{j-2k} \Big(\frac{k^{j-2k} }{(j-2k)!}\Big)^\ep
		\\
		&= \frac{1}{M_k} (Ae)^k \Big(\frac{2 + A^{-1}}{\rh}\Big)^{2k}  \exp\Big(k \ep \Big(\frac{2 + A^{-1}}{\rh}\Big)^{1/\ep} \Big).
	\end{align*}
	This completes the proof of the claim and hence of (2) in \Cref{thm:Borelsufficient}.

  Let us briefly discuss the Roumieu case, that is (1) in \Cref{thm:Borelsufficient}.
  Given $a=(a_k) \in \La^{\{M\}}$, which means that $|a|^M_\rh<\infty$ for some $\rh>0$, one
  looks for functions $\ps_k$, $k \in \N$, such that $f:= \sum_{k \ge 0} a_k \ps$ is of class $\cE^{\{M\}}$ and $\Bmap f = a$.
  It suffices to repeat the above construction, where \eqref{eq:modifiedseq} is replaced by
  \begin{equation*}
		\underbrace{H\mu_k,\ldots,H\mu_k}_{k \text{ times}}, H\mu_{k+1},H\mu_{k+2}, \ldots
	\end{equation*}
  If $H>0$ is chosen large enough, one finds that there is a constant $B\ge 1$ such that
  \[
    \|f\|^M_{\R,B H} \le C(\rh) \, |a|^M_\rh,
  \]
  but $H$ and thus also $\ps_k$ and $f$ depend on $\rh$; cf.\ \Cref{rem:extopBorelRoumieu}.

\section{The Borel map in the quasianalytic setting} \label{sec:quasianalytic}

Let us investigate surjectivity of the Borel map in the quasianalytic case.
Since there are no cutoff functions, we will restrict $\Bmap$
to the ring of germs of $\cE^{[M]}$-functions at $0$ which we denote by $\cE^{[M]}_0$.
(The dimension of the ambient space will be clear from the context.)
A result of Carleman \cite{Carleman23}
states that the Borel map $\Bmap : \cE^{[M]}_0 \to \La^{[M]}$ is never surjective if $M$ is a quasianalytic weight sequence
with $m_k^{1/k} \to \infty$. Recall that $m_k^{1/k} \to \infty$ means that the real analytic class
is strictly contained in $\cE^{[M]}$; on the ring of real analytic germs the Borel map is
an isomorphism $\cE^{\{(k!)\}}_0 \cong \La^{\{(k!)\}}$.

Many different proofs of Carleman's result are known.
We present a much stronger result
which shows that certain elements of $\La^{[M]}$
cannot be extended to germs of class $\cE^{\{N\}}$, where $N$ is \emph{any} quasianalytic weight sequence.
It is based on a theorem of Bang \cite{Bang53} on the zeros of quasianalytic functions and their derivatives.
See also \cite{NazarovSodinVolberg04} and \cite{RainerSchindl15}.

\subsection{Zeros of quasianalytic functions and their derivatives}

\begin{theorem} \label{thm:Bang}
	Let $M$ be a quasianalytic weight sequence.
	Suppose that $f \in C^\infty([0,1])$ is not identically zero and satisfies
	\begin{equation} \label{eq:Bd1}
 		\|f^{(k)}\|_{[0,1]} \le M_k \quad  \text{ for all } k \in \N.
	\end{equation}
	Suppose that for $0\le j \le m$ there exists $x_j \in [0,1]$ such that $f^{(j)}(x_j) = 0$.
	Then
	\begin{equation} \label{eq:assertion}
		\sum_{j=1}^m |x_{j-1} -x_j| \ge \frac{1}{e} \sum_{\ka < k \le m} \frac{1}{\mu_k},
	\end{equation}
	where
	\begin{equation} \label{eq:kappa}
		\ka = - \log \sup_{j \ge 0} \frac{|f^{(j)}(x_0)|}{e^j M_j}.
	\end{equation}
\end{theorem}

\begin{proof}
	For $n \in \N$ and $t \in [0,1]$ consider
  \(
    B_{f,n}(t) :=  \sup_{j \ge n} \frac{|f^{(j)}(t)|}{e^j M_j}.
  \)
Then:
  \begin{enumerate}
    \item $B_{f,n}(t) \le e^{-n}$ for all $t \in [0,1]$.
    \item $B_{f,n} \ge B_{f,n+1}$, and $f^{(n)}(t_0) = 0$ implies $B_{f,n}(t_0) = B_{f,n+1}(t_0)$.
    \item For all $k > n$ and all $t, s \in [0,1]$,
    \[
      B_{f,n}(s) < \max \{B_{f,n}(t) , e^{-k} \}\, e^{e |t-s| \mu_k}.
    \]
  \end{enumerate}
  (1) and (2) are obvious. To see (3)
  let $k > n$, $n \le j < k$, and $t, s \in [0,1]$. Then, by Taylor's formula, for some $\xi$ between $t$ and $s$,
  \begin{align*}
    \frac{|f^{(j)}(s)|}{e^j M_j}
    &\le \sum_{i = 0}^{k-j-1} \frac{|f^{(j+i)}(t)|\, |t-s|^i}{e^j M_j\, i!}
    + \frac{ |f^{(k)}(\xi)|\, |t-s|^{k-j}}{e^j M_j\, (k-j)!}  \\
    &= \sum_{i = 0}^{k-j-1} \frac{M_{j+i}}{M_j}\frac{|f^{(j+i)}(t)| }{e^{j+i} M_{j+i}} \frac{(e|t-s|)^i}{i!}
    + e^{-k} \frac{M_{k}}{M_j} \frac{ |f^{(k)}(\xi)|}{M_k} \frac{ (e|t-s|)^{k-j}}{(k-j)!} \\
    &\le B_{f,N}(t) \sum_{i = 0}^{k-j-1} \mu_k^i \frac{(e|t-s|)^i}{i!}
    + e^{-k} \mu_k^{k-j}  \frac{ (e|t-s|)^{k-j}}{(k-j)!} \\
    &< \max\{ B_{f,n}(t) , e^{-k} \}\, e^{e |t-s| \mu_k},
  \end{align*}
  where we used that $\mu_k$ is increasing.
  If $j \ge k$, then trivially
  \begin{align*}
    \frac{|f^{(j)}(s)|}{e^j M_j} \le e^{-j} < \max\{ B_{f,n}(t), e^{-k} \} \, e^{e |t-s| \mu_k}.
  \end{align*}
  This implies (3).

    Let $f$ and $x_j$ be as in the theorem.
  Set $\ta_k := \sum_{j=1}^{k} |x_{j-1}-x_{j}|$, $k \ge 1$, $\ta_0:= 0$, and define for $t \in [\ta_{n-1},\ta_n]$,
  \[
    \tilde B_{f,n} (t) :=
    \begin{cases}
       B_{f,n}(x_{n-1} +\ta_{n-1} -t) & \text{if } x_n < x_{n-1}, \\
       B_{f,n}(x_{n-1} -\ta_{n-1} +t) & \text{if } x_n \ge x_{n-1}.
    \end{cases}
  \]
  The function $\tilde B_{f,n}$ is continuous, by (3), and
  $\tilde B_{f,n}(\ta_n) = B_{f,n}(x_n) = B_{f,n+1}(x_n) =\tilde B_{f,n+1}(\ta_n)$, by (2).
  Thus,
  \[
    \tilde B_{f}(t) := \tilde B_{f,n}(t) \quad \text{ if } t \in [\ta_{n-1},\ta_n],~ n \ge 1.
  \]
  defines a continuous function on $[0,\ta_m]$.
  By (1) and (2), we have $\tilde B_{f}(t) \le e^{-n}$ for all $t \ge \ta_{n-1}$,
  in particular, $\tilde B_{f}(t) \le e^{-m}$ for all $t \in  [\ta_{m-1}, \ta_m]$.
  Since $f$ does not vanish identically,
  \(
  	\max_{t \in [0,\ta_m]} \tilde B_{f}(t) \ge \tilde B_{f}(0) = B_{f,1}(x_0) = B_{f,0}(x_0) >0.
  \)
  Hence, the range of $\tilde B_f$ contains all numbers $e^{-k}$ for $\ka < k \le m$, where
  \(
  	 e^{-\ka} = B_{f,0}(x_0)
  \)
  which is equivalent to \eqref{eq:kappa}.
  So we may choose a strictly increasing sequence $t_k$, $\ka < k \le m$, such that $\tilde B_f(t_k) = e^{-k}$
  and $\tilde B_f(t) > e^{-k}$ for all
  $t \in (t_{k-1},t_k)$ (recursively, take for $t_k$ the smallest $t \in \tilde B_f^{-1}(e^{-k})$ with $t>t_{k-1}$).
  By (3) (applied to each interval in the subdivision of $(t_{k-1},t_k)$ induced by the points $\ta_n$ between $t_{k-1}$ and $t_k$)
  we may conclude that
  \[
    \tilde B_{f}(t_{k-1}) \le \tilde B_{f}(t_k)\, e^{e (t_k-t_{k-1}) \mu_k}.
  \]
  Since $\tilde B_f(t_k) = e^{-k}$,
  that means
  \[
     t_k-t_{k-1} \ge \frac1{e \mu_k}.
  \]
  Summing over $k$ we find
  \[
    t_m \ge \frac1{e} \sum_{\kappa < k\le m} \frac1{\mu_k}.
  \]
  By the choice of the sequence $t_k$, we have $\ta_k \ge t_k$ which implies \eqref{eq:assertion}.
\end{proof}

\begin{corollary} \label{corBang}
  Let $M$ be a quasianalytic weight sequence. Let $f \in C^\infty([0,1])$ satisfy \eqref{eq:Bd1}.
  If $f^{(j)}(0) > 0$ for all $j \in \N$, then $f^{(j)}(x) > 0$ for all $x \in [0,1]$ and $j \in \N$.
\end{corollary}

\begin{proof}
  Suppose that some
  derivative $f^{(j)}$ has a zero $x_{j} \in (0,1]$. By Rolle's theorem, we find a strictly decreasing sequence
  $x_j > x_{j+1} > \cdots > 0$, where $f^{(k)}(x_k)=0$ for all $k \ge j$. This contradicts \Cref{thm:Bang}.
\end{proof}

\begin{corollary}
	Let $M$ be a quasianalytic weight sequence. Let $f \in C^\infty([0,1])$ satisfy \eqref{eq:Bd1}.
	The total number of zeros (counted with multiplicities) of $f$ is bounded by
	\[
		\sup\Big\{m \in \N : \sum_{\ka < k \le m} \frac{1}{\mu_k} \le e\Big\}+1.
	\]
\end{corollary}

\begin{proof}
	Let $z_1 \le z_2 \le \cdots \le z_m$ be the zeros of $f$ in $[0,1]$.
	By Rolle's theorem, there exists a sequence of points $z_1 = x_0 \le x_1 \le \cdots \le x_{m-1}$ in $[0,1]$ such that $f^{(j)}(x_j) = 0$ for all $j$.
	By \eqref{eq:assertion},
	\[
		1 \ge \sum_{j=1}^{m-1} (x_j - x_{j-1}) \ge \frac{1}{e} \sum_{\ka < k \le m-1} \frac{1}{\mu_k}
	\]
	and the statement follows.
\end{proof}

\subsection{The quasianalytic Borel map is never onto}

\begin{theorem} \label{seqRoumieu}
  Let $M$ be a quasianalytic weight sequence such that $m_k^{1/k}\to \infty$.
  Then there exist elements in $\La^{(M)}$ that are not contained in $j^\infty_{\{0\}}  \cB^{\{N\}}((-r,r))$ for
  any quasianalytic weight sequence $N$ and any $r>0$.
\end{theorem}

\begin{proof}
  Let $a=(a_j) \in \La^{\{M\}}$ be positive, i.e., $a_j >0$ for all $j$.
  Let $N$ be any quasianalytic weight sequence and $r>0$.
  We claim that
  if there exists $f \in \cB^{\{N\}}((-r,r))$ such that $j^\infty_{\{0\}} f = a$ then $f$ is real analytic.
  After rescaling we may assume that $\|f^{(j)}\|_{[0,1]} \le N_j$ for all $j \in \N$.
  \Cref{corBang} implies
  that $f^{(j)}(x) >0$ for all $x \in [0,1]$ and all $j \in \N$.
  By Bernstein's theorem (e.g.\ \cite[p.~146]{Widder41}), $f$ extends to an analytic function
  on the unit disk in $\C$.
  Now it suffices to choose $a=(a_j)$ such that it does not define a real analytic germ, which is possible by the assumption
  $m_k^{1/k}\to \infty$.

  To see that there exist such $a$ even in $\La^{(M)}$ set $L_k := k!\, \sqrt{m_k}$
  and let $\ul L$ be the log-convex minorant of $L$. Thus $\ul L \lhd M$ and $\underline \ell_k^{1/k} \to \infty$
  (cf.\ \Cref{fn1}), so that $\ul L$ is a quasianalytic weight sequence for which the argument in the
  previous paragraph applies.\footnote{It is not hard to see that
  $\La^{(M)} = \bigcup \{\La^{\{L\}} : L \text{ positive sequence}, ~L \lhd M, ~\ell_k^{1/k} \to \infty\}$.}
\end{proof}

\subsection{The impossibility of quasianalytic extension}

\begin{theorem}
	Let $M$ be a quasianalytic weight sequence with $m_k^{1/k} \to \infty$.
	Then $\cB^{(M)}((-1,1))$ contains functions $f$ which have no quasianalytic extension to a larger interval, i.e.,
	if $\tilde f$ is an extension of $f$ to a neighborhood of $[-1,1]$ and $N$ is any quasianalytic weight sequence, then $\tilde f \not\in \cE^{\{N\}}$.
\end{theorem}

\begin{proof}
  We use an argument of \cite{NazarovSodinVolberg04} to see that $\cB^{\{M\}}((-1,1))$ contains
  functions with the asserted properties. That such functions can even be found in $\cB^{(M)}((-1,1))$
  follows from the reasoning at the end of the proof of \Cref{seqRoumieu}.
	Let $c_j$ be a positive sequence
	with
	\begin{gather}
		c_j^{1/j} \to 1, \label{eq:cond1} \\
		\sum_{j=1}^\infty j^n c_j \le M_n \quad \text{ for all } n \in \N.	\label{eq:cond2}
	\end{gather}
	The even function
	\(
		f(x) = \sum_{k=0}^\infty c_{2k} x^{2k}
	\)
	is real analytic on $(-1,1)$, by \eqref{eq:cond1},
	and belongs to $\cB^{\{M\}}((-1,1))$, by \eqref{eq:cond2}.
	Suppose that $f$ has a quasianalytic extension $\tilde f$ to a larger interval.
	Then $\tilde f^{(j)} (0) =  f^{(j)} (0) \ge 0$ for all $j \in \N$.
	Hence $\tilde f$ is analytic on a disk centered at $0$ with radius larger than $1$, by \Cref{corBang} und Bernstein's theorem.
	But this contradicts \eqref{eq:cond1}.
\end{proof}

\subsection{Digression: intricacies of the quasianalytic setting}

We want to point out some interesting intricacies of the quasianalytic setting. This section deviates from
the central theme of the article it is thus kept rather short.

  (1) For each smooth germ $f$ there exist quasianalytic weight sequences $M^1,M^2$ such that
  $f =f_1 + f_2$ with $f_i \in \cE^{\{M^i\}}$, $i=1,2$. See \cite{Mandelbrojt40} and
  \cite{RolinSpeisseggerWilkie03} (who concluded from this that there is no largest o-minimal expansion of the real field).

  (2) Weierstrass division and preparation, generally, fail in quasianalytic classes.
  See \cite{Acquistapace:2014wf,ElkhadiriSfouli08,Parusinski:2014tg}.
  If $M$ is derivation-closed and quasianalytic, then division in $\cE^{\{M\}}$ by a Weierstrass polynomial $\vh$ holds
  if and only if all roots of $\vh$ are real; see \cite{Childress76} and \cite{CC04}.
  For an overview of related results see \cite{Thilliez08}.

  (3) Assume that $M$ is a quasianalytic strongly log-convex and derivation-closed weight sequence.
  It is not known if the local ring $\cE^{[M]}_0$ is Noetherian.
  But $\cE^{[M]}$ admits resolution of singularities \cite{BM97,BM04} and also \cite{RolinSpeisseggerWilkie03}
  and, consequently, tools such as
  topological Noetherianity, curve selection, and {\L}ojasiewicz inequalities are available.

  (4) Ultradifferentiable quasianalyticity cannot be tested in lower dimensions (in contrast to real analyticity
  \cite{Siciak70,Bochnak70,Bochnak:2020tz}).
  Let $M$ be a quasianalytic strongly log-convex weight sequence with $m_k^{1/k}\to \infty$.
  For any $n\ge 2$ and any positive sequence $N$
  there exists a $C^\infty$-function $f \in \cE^{[M]}(\R^n \setminus \{0\}) \setminus \cE^{\{N\}}(\R^n)$
  such that $f \o p \in \cE^{[M]}$ for all $\cE^{[M]}$-mappings $p : \R^m \supseteq U \to \R^n$ with $m<n$.
  See \cite{Jaffe16} and \cite{Rainer:2019aa}.

  (5) Quasianalytic Roumieu classes are intersections of non-quasianalytic ones:
  Let $M$ be a quasianalytic strongly log-convex weight sequence. Then
  \[
    \cE^{\{M\}} = \bigcap_{N \in \cN(M)} \cE^{\{N\}},
  \]
  where $\cN(M)$ is the collection of all  non-quasianalytic weight sequences $N \ge M$.
  For suitable $M$ (e.g.\ Denjoy sequences $Q^{n,1}$, see \Cref{ex:Denjoy}), we may even restrict the intersection to all
  \emph{strongly log-convex} weight sequences in $\cN(M)$. But notice that we cannot describe all $\cE^{\{M\}}$
  in this way: the intersection of \emph{all} non-quasianalytic $\cE^{\{N\}}$, where $N$ is strongly log-convex,
  is the quasianalytic class $\cE^{\{Q^{1,1}\}} \supsetneq C^\om$.
  See \cite{Boman63} and \cite{KMRq} as well as \cite{Nenning:2021wd} for an application.

  (6) Let $M$ be a quasianalytic, strongly log-convex,  and derivation-closed weight sequence.
  Consider the quasianalytic equation
  \[
    \vh(x,y) = y^d + a_1(x) y^{d-1} + \cdots a_d(x) \in \cE^{[M]}_0[y], \quad x=(x_1,\ldots,x_n).
  \]
  Due to \cite{Thilliez10}, if $n = 1$, then a smooth solution $y=h(x)$ is of class $\cE^{[M]}$.
  It is not known, if this is true for $n>1$.
  There is the following partial solution \cite{Belotto-da-Silva:2017aa}:
  Let $\vh(x,y)$ be a function of class $\cE^{\{M\}}$ defined near $(a,b) \in \R^n \times \R$.
  Then there exist $p \in \N$ and a quasianalytic class $\cQ \subseteq \cE^{\{M^{(p)}\}}$ (where $M^{(p)}_k := M_{pk}$)
  such that if $\vh(x,y)=0$ admits a formal power series solution $y = H(x)$ at $a$ then
  there is a solution $y = h(x) \in \cQ$ near $a$ and $T_a h = H$. Note that, in general,
  $M^{(p)}$ is no longer quasianalytic and $\cQ$ is not a Denjoy--Carleman class.

\section{Borel's lemma with controlled loss of regularity}

Let $M$ be a weight sequence satisfying $m_k^{1/k} \to \infty$.
We have seen that the Borel map $j^\infty_{\{0\}} : \cE^{[M]}(\R) \to \La^{[M]}$
is injective if and only if $M$ is quasianalytic and
surjective if and only if $M$ is strongly non-quasianalytic.
If $M$ is not strongly non-quasianalytic, then several new natural questions arise.

\begin{question} \label{Q:summation}
	Assume that $M$ is not strongly non-quasianalytic.
	\begin{enumerate}
		\item What can be said about the image $j^\infty_{\{0\}} \cE^{[M]}(\R)$? Is there an intrinsic description?
		\item In the case that $M$ is quasianalytic, is there a constructive method for finding the unique function $f$
		with $j^\infty_{\{0\}}f=a$ for a given $a \in j^\infty_{\{0\}} \cE^{[M]}_0$?
		\item When do we have $\La^{[L]} \subseteq j^\infty_{\{0\}} \cE^{[M]}(\R)$ for another weight sequence $L$?
	\end{enumerate}
\end{question}

Note that, as see in \Cref{seqRoumieu},
positive sequences that grow fast enough cannot belong to $j^\infty_{\{0\}} \cE^{[M]}_0$ if
$M$ is quasianalytic.

\subsection{When do we have $\La^{[L]} \subseteq j^\infty_{\{0\}} \cE^{[M]}(\R)$?}

This question was completely answered by Schmets and Valdivia \cite{Schmets:2003aa}
if $M$ is a non-quasianalytic weight sequence.

Let $M\le N$ be weight sequences, $N$ non-quasianalytic.
For $p \in \N_{\ge 1}$ we consider the sequence $\la_p = \la_p(M,N)$ defined by
\[
	\la_{p,k} := \sup_{0 \le j <k} \Big(\frac{M_k}{p^k N_j} \Big)^{\frac{1}{k-j}},\quad  k \ge 1.
\]
Note that $\la_p \le \mu$ for all $p\ge 1$, indeed, since $M \le N$,
\[
	\Big(\frac{M_k}{p^k N_j} \Big)^{\frac{1}{k-j}}
	\le \Big(\frac{M_k}{M_j} \Big)^{\frac{1}{k-j}}
	= (\mu_{j+1} \cdots \mu_k)^{\frac{1}{k-j}} \le \mu_k.
\]
If $M$ has moderate growth, then $\mu_k \lesssim M_k^{1/k}$ and so also a converse estimate holds:
\begin{align*}
	\mu_k \lesssim M_k^{1/k} =  p \Big(\frac{M_k}{p^k N_0}\Big)^{1/k} \le  p \la_{p,k}.
\end{align*}

\begin{theorem}[\cite{Schmets:2003aa}] \label{thm:Borelloss}
	Let $M\le N$ be weight sequences, $N$ non-quasianalytic.
	The following conditions are equivalent:
	\begin{enumerate}
		\item $\La^{(M)} \subseteq j^\infty_{\{0\}} \cE^{(N)}(\R)$.
		\item $\La^{\{M\}} \subseteq j^\infty_{\{0\}} \cE^{\{N\}}(\R)$.
		\item There is $p \ge 1$ such that
		\begin{equation} \label{eq:SVcond}
			\sup_{k \ge 1} \frac{\la_{p,k}}{k} \sum_{j \ge k} \frac{1}{\nu_j} < \infty.
		\end{equation}
	\end{enumerate}
\end{theorem}

By the remarks before the theorem, if $M$ has moderate growth, then \eqref{eq:SVcond} (for any $p \ge 1$) is equivalent to
\begin{equation}\label{eq:mixedga1}
	\sup_{k \ge 1} \frac{\mu_k}{k} \sum_{j \ge k} \frac{1}{\nu_j} < \infty.
\end{equation}
In general, the latter condition is stronger than \eqref{eq:SVcond}.
By \Cref{lem:log-convex}, the descendant $\si$ of $\nu$ is the \emph{largest} sequence among all sequences $\mu$
satisfying $\mu \lesssim \nu$ and \eqref{eq:mixedga1}, in the sense that it generates the largest function space.

\subsection{Moment-type summation}

Some interesting results concerning the questions \ref{Q:summation}(1) and (2)
based on a moment-type summation method
are due to
Beurling, Carleson \cite{Carleson:1961wa} and recently Kiro \cite{Kiro:2020vc}.

Informally, it works as follows. Let $I \subseteq \R$ be an open interval containing $0$.
Given $f \in \cE^{(M)}(I)$, the \emph{singular transform} of $f$ is the entire function
\[
	S_M f (z) := \sum_{n \ge 0} \frac{f^{(n)}(0)}{n!\, m_{n+1}} z^n.
\]
It depends only on $j^\infty_{\{0\}} f$
and hence studying the image $S_M \cE^{(M)}(I)$ is equivalent to studying the image of the Borel map
$j^\infty_{\{0\}}\cE^{(M)}(I)$.
This study is based on the observation that elements in $S_M \cE^{(M)}(I)$ have a certain growth behavior in
horizontal strips.
(In the Roumieu case the series may have a finite radius of convergence and analytic continuation to a horizontal strip
is required.)

Let $K$ be the kernel that solves the moment problem
\[
	\int_0^\infty t^n K(t) \, dt = m_{n+1}, \quad n \in \mathbb N.
\]
Under suitable conditions, the \emph{regular transform}
\[
	R_M F(x) := \int_0^\infty F(xt) K(t)\, dt
\]
then yields the desired reconstruction $R_M S_M f =f$.

To make this approach work, one needs to control the asymptotic behavior of $K$ and of
\[
	E(z) := \sum_{n \ge 0} \frac{z^n}{m_{n+1}}
\]
which manifests itself in regularity properties of the weights $M$. The function $E$ regulates the growth
behavior of elements in $S_M \cE^{(M)}(I)$ in horizontal strips.

Kiro \cite{Kiro:2020vc} executes this program for a suitable class of weights; the case $M_k = k! \log(k +e)^k$
was treated by Beurling. Also a duality between suitable non-quasianalytic and quasianalytic classes, which have the
same image under the respective singular transform, is explored.

\section{Optimal cutoff functions} \label{sec:cutoff}

The most important tool for the extension problem for general closed subsets $A \subseteq \R^n$ are \emph{good} cutoff functions.
We start with an observation that shows how good ultradifferentiable cutoff functions can be.

\subsection{A lower bound for cutoff functions}
Cf.\ \cite{Bruna80}.
Let $M=(M_k)$ be a non-quasianalytic weight sequence.
Let $r,\ep,\rh>0$.
Suppose that $\vh \in \cB^{M}_\rh(\R)$ satisfies
\begin{equation} \label{eq:Bruna1}
	\vh|_{[-r,r]} = 1 \quad \text{ and } \quad
	\on{supp} \vh \subseteq [-(1+\ep)r,(1+\ep)r].
\end{equation}
It is to be expected that $\|\vh\|^M_{\R,\rh}$ tends to infinity,
if $r$, $\ep$, or $\rh$ approach $0$.
Let us check and quantify this guess.
For $t \in (r,(1+\ep)r)$, we have by Taylor's formula
\[
	|\vh(t)| \le \frac{\vh^{(k)}(\ta)}{k!} (r \ep)^k, \quad k \in \N,
\]
for some $\ta \in (t,(1+\ep)r)$. Since $\vh(t)\to 1$ as $t \to r$, we find
$\|\vh^{(k)}\|_\R \ge k!/(r\ep)^k$ for all $k$. Thus,
\[
	\|\vh\|^M_{\R,\rh} \ge \sup_{k \in \N} \frac{k!}{(\rh r \ep)^k M_k} = \frac{1}{\inf_{k \in \N} (\rh r \ep)^k m_k}.
\]
It turns out that $h_m(t) := \inf_{k \in \N} m_k t^k$ is a very useful auxiliary function; we will discuss it in \Cref{sec:associatedfunctions}.
So any cutoff function $\vh \in \cB^{M}_\rh(\R)$ with \eqref{eq:Bruna1}
must satisfy
\begin{equation} \label{eq:goodcutoff}
	\|\vh\|^M_{\R,\rh} \ge  \frac{1}{h_m(\rh r \ep)}.
\end{equation}

\subsection{Associated functions} \label{sec:associatedfunctions}

Suppose that $m=(m_k)$ is log-convex, $1 = m_0 \le m_1$ and $m_k^{1/k} \to \infty$
(i.e., $m$ is a weight sequence or, equivalently, $M_k = k!\, m_k$ is a
strongly log-convex weight sequence).
We associate the function $h_m : [0,\infty) \to [0,\infty)$, where  $h_m(0):=0$ and
\begin{equation} \label{h}
  h_m(t) := \inf_{k \in \N} m_k t^k, \quad t > 0.
\end{equation}
Then $h_m$ is increasing, continuous, and positive for $t>0$. For $t \ge 1/m_1$ we have $h_m(t) = 1$. From $h_m$
we may recover the sequence $m$ by $m_k = \sup_{t>0} h_m(t)/ t^{k}$.
We also associate the counting function $\Ga_m : (0,\infty) \to \N_{\ge 1}$ by setting
\begin{equation} \label{counting2}
  \Ga_m(t) := \min\{k \ge 1 : h_m(t) =  m_k t^k\} = \min\Big\{k \ge 1 : \frac{m_{k+1}}{m_k} \ge \frac{1}{t} \Big\};
\end{equation}
for this identity we need that $\frac{m_{k+1}}{m_k}$ is increasing, i.e., $m$ is log-convex. Then:
\begin{gather}
	\text{$k \mapsto m_k t^k$ is decreasing for $k \le \Ga_m(t)$,} \label{GaProp1}
	\\
	\text{$h_m(t) = m_{\Ga_m(t)} t^{\Ga_m(t)} \le m_k t^k$ for all $k$.} \label{GaProp2}
\end{gather}
It follows that
\begin{equation} \label{GaProp3}
  \Ga_m(t) = k \quad \text{ and }\quad h_m(t) = m_k t^k, \qquad \text{ for } t \in [\tfrac{m_k}{m_{k+1}},\tfrac{m_{k-1}}{m_{k}}).
\end{equation}

\begin{example}
	For $m_k = k!^s$, where $s > 0$, the function $h_m$ behaves like $t \mapsto \exp(-1/t^{1/s})$:
  \begin{align*}
     \sup_{k} \frac{t^k}{k!^s} = \Big( \sup_{k}  \frac{t^{k/s}}{k!} \Big)^s \le \Big( \sum_{k}  \frac{t^{k/s}}{k!} \Big)^s = e^{s t^{1/s}}
  \end{align*}
  and, conversely,
  \begin{align*}
     e^{st^{1/s}} = \Big(\sum_{k} \frac{1}{2^k}  \frac{(2 t^{1/s})^k}{k!} \Big)^s \le 2^s \sup_{k} \frac{(2^s t)^k}{k!^s}.
  \end{align*}
	Similarly, for $m_k = \log(k+e)^{sk}$,  $h_m$ behaves like $t \mapsto \exp(-\exp(1/t^{1/s}))$.
\end{example}

We shall need some properties of $h_m$ and $\Ga_m$ which are guaranteed if
$m$ has moderate growth:

\begin{lemma} \label{claim1}
 Suppose that $m$ is a weight sequence of moderate growth, i.e.,
 there is a constant $C>1$ such that
 \begin{equation} \label{eq:mgv1}
 	\frac{m_{2k}}{m_{2k-1}} \le C \frac{m_{k}}{m_{k-1}},  \quad k\ge 1.
 \end{equation}
 Then\footnote{Actually \eqref{eq:mgv1} and \eqref{hmg} are equivalent, see \cite[Remark 2.5]{RainerSchindl16a}.}
  \begin{gather}
       h_m(t) \le h_m(Ct)^2, \quad t>0, \label{hmg}
  \end{gather}
  and, for a possibly other constant $C>1$,
  \begin{gather}
      	2  \Ga_m(Ct) \le \Ga_m(t),  \quad t>0 \text{ small enough}.   \label{eq32}
  \end{gather}
\end{lemma}

\begin{proof}
	Consider the function
\(
  \Si_m(t) := |\{k \ge 1 : \tfrac{m_k}{m_{k-1}} \le t\}|.
\)
 Now \eqref{eq:mgv1} implies
 \begin{equation} \label{eq:Sigmam}
   2 \Si_{m}(t) \le \Si_m (Ct), \quad t>0.
 \end{equation}
It is well-known (cf.\ \cite{Mandelbrojt52} and \cite{Komatsu73}) that $\om_m$ defined by
$\om_m(t) = \log(1/h_m(1/t))$
satisfies\footnote{Note that the integral
$\int_1^\infty \frac{\om_M(t)}{t^2}\,dt$ converges if and only if $\int_1^\infty \frac{\Si_M(t)}{t^2}\,dt$ converges, and that is the case if and only if $M$ is
non-quasianalytic; \cite[Section 4]{Komatsu73}.}
\begin{equation} \label{omM}
  \om_m(t) = \int_{0}^t\frac{\Si_m(u)}{u} \,du.
\end{equation}
So \eqref{eq:Sigmam} implies $2\omega_{m}(t) \le \omega_m(Ct)$ for all $t>0$, which is clearly equivalent to \eqref{hmg}.

Let us check \eqref{eq32}. Since $m$ has moderate growth and is log-convex, there are constants $C_1,C_2>0$ such that
\[
  	\frac{m_{2k+1}}{m_{2k}} \le C_1 \frac{m_{2k}}{m_{2k-1}} \le C_2 \frac{m_{k}}{m_{k-1}} \le C_2 \frac{m_{k+1}}{m_{k}},  \quad k\ge 1.
\]
Fix $t>0$ and let $\ell := \Ga_m(t)$. If $\ell=2k$,
then
\[
  \frac{1}{t} \le \frac{m_{\ell+1}}{m_{\ell}} \implies \frac{1}{C_2 t} \le    \frac{m_{k+1}}{m_{k}},
\]
that is, $\Ga_m(C_2t) \le k = \tfrac{1}{2} \Ga_m(t)$.
If $\ell = 2k+1$, then we similarly see that
$\Ga_m(\frac{C_2}{C_1}t) \le k =  \tfrac{1}{2} (\Ga_m(t) -1) \le \tfrac{1}{2} \Ga_m(t)$.
Since $\Ga_m$ is decreasing, \eqref{eq32} follows.
\end{proof}

Moderate growth of the weight sequence is an essential technical tool in this section. So we make the following definition.

\begin{definition}
	A weight sequence $M=(M_k)$ is called \emph{strongly regular} if $M$ is strongly non-quasianalytic, has moderate growth,
	and $m$ is log-convex.
	Note that the last condition is actually for free:
  a strongly non-quasianalytic weight sequence $M$ is equivalent to a strongly log-convex weight sequence (that is still
  strongly non-quasianalytic),
  by \Cref{cor:log-convex}.
\end{definition}

\subsection{Optimal cutoff functions of Roumieu type} \label{sec:optcutoffRWS}

It turns out that there exist \emph{optimal} $\cE^{\{M\}}$ cutoff functions, i.e., $\vh \in \cE^{\{M\}}(\R)$ satisfying \eqref{eq:Bruna1} and realizing \eqref{eq:goodcutoff} so that
\begin{equation*}
    \|\vh\|^M_{\R,\rh} \sim  \frac{1}{h_m(\rh r \ep)},
 \end{equation*}
if and only if $M$ is strongly non-quasianalytic.
The existence of $\vh$ will follow from \Cref{prop:cutoff}. The following lemma provides the necessary assumptions.
Recall that $\mu_k^* = \frac{\mu_k}{k} = \frac{m_k}{m_{k-1}}$.

\begin{lemma} \label{lem:assforcutoff}
  Let $M$ be a strongly regular weight sequence.
  There is a constant $A\ge 1$ such that for each integer $p\ge 1$ there is a sequence $(\alpha^p_k)_{k\in\N}$ satisfying
\begin{align}
  \sum_{k\ge 0}\frac{\alpha^p_k}{\alpha^p_{k+1}}\le 1, \quad \alpha_0^p=1,
  \quad 0<\alpha^p_k\le   \frac{(\frac{A}{\mu_{p+1}^*})^k M_k}{h_{m}(\frac{1}{e\mu_{p}^*})}.
\label{alpha3}
\end{align}
\end{lemma}

\begin{proof}
  We will show that, for a suitable constant $A \ge 1$, the family of sequences
    \begin{equation*}
      \al^p_k :=
      \begin{cases}
          \big(\frac{A}{\mu_{p+1}^*}\big)^k M_k & \text{ if } k >p, \\
          (2p)^k & \text{ if } k \le p,
      \end{cases}
    \end{equation*}
    has the desired properties.
    Since $M$ has moderate growth, $\mu_{k+1} \lesssim  M_k^{1/k}$ and thus,
    for some constant $C \ge 1$,
    \begin{align*}
      \alpha^p_p= 2^pp^p\le  \Big( \frac{2C}{\mu_{p+1}^*}\Big)^p M_p.
    \end{align*}
    So if $A$ is chosen large enough, then
    \begin{align*}
        \sum_{k\ge 0} \frac{\alpha_k^p}{\alpha^p_{k+1}}
        &\le\sum_{k<p}\frac{1}{2p}+ \frac{\mu_{p+1}^*}{A}\sum_{k\ge p}\frac{1}{\mu_{k+1}}
        \le 1,
    \end{align*}
    since $M$ is strongly non-quasianalytic.
    The bound for $\al^p_k$ (in \eqref{alpha3}) is obvious for $k>p$,
    since $h_m\le 1$.
    If $k\le p$, then, for large enough $A$,
    \begin{align*}
    \frac{\al^p_k}{(\frac{A}{\mu_{p+1}^*})^k\,  M_k}
    &=\frac{2^k p^k}{(\frac{A}{\mu_{p+1}^*})^k\,  M_k}
    \le \frac{\mu_{p+1}^k}{(\frac{A}{2})^{k} M_k}
    \le \frac{\mu_{p}^k}{M_k}
    \le \frac{\mu_{p}^p}{M_p} = \frac{\mu_{p}^p}{p! m_p} \le  \frac{(e \mu_p^*)^p}{m_p}.
    \end{align*}
    This implies the statement.
\end{proof}

\begin{theorem}[{\cite[Theorem 2.2]{Bruna80}}] \label{thm:optimalbump}
 	Let $M=(M_k)$ be a non-quasianalytic strongly log-convex weight sequence of moderate growth.
	Then the following conditions are equivalent:
	\begin{enumerate}
		\item $M$ is strongly non-quasianalytic.
		\item There is a constant $C>0$ such that for all $\rh, r, \ep>0$ there exists $\vh \in \cE^{\{M\}}(\R)$ satisfying \eqref{eq:Bruna1} and
		\begin{equation} \label{eq:cutoffRM}
		 		\|\vh\|^M_{\R,\rh} \le  \frac{1}{h_m(C \rh r \ep)}.
		 \end{equation}
	\end{enumerate}
\end{theorem}

\begin{proof}
  (1) $\Rightarrow$ (2)
  We first consider the case $r=\ep=1$.
  Let $A$ be the constant from \Cref{lem:assforcutoff}.
  Fix $0<\eta\le 2 A$.
  Since $\mu_k^* \nearrow \infty$,
there is an integer $p \ge 1$ such that
\begin{equation}\label{testfunctionconstructionequ1}
	\frac{2A}{\mu_{p+1}^*} < \et \le  \frac{2A}{\mu_p^*}.
\end{equation}
By \Cref{lem:assforcutoff}, we may
apply \Cref{prop:cutoff} (to $d_{k+1}=\frac{\alpha^p_k}{\alpha^p_{k+1}}$ and $K =[-1,1]$) and
get $\vh =\vh_{\et} \in C^\infty(\R)$ with
$0 \le \vh \le 1$, $\vh|_{[-1,1]} = 1$, $\on{supp}\vh \subseteq [-2,2]$, and
\begin{align*}
  |\vh^{(k)}(t)|
    \le 2^{k}\alpha_k^p
  	\le \frac{\eta^k M_k}{h_m(\frac{\et}{6A})}, \quad k \ge 1,
\end{align*}
by \eqref{alpha3} and \eqref{testfunctionconstructionequ1}.
For $\eta>2A$ we put $\vh_{\eta}:=\vh_{2A}$; then since $h_m \le1$,
\begin{align*}
|\vh^{(k)}(t)| &\le \frac{(2A)^k M_k}{h_m(\frac{2A}{6A})}
  \le \frac{1}{h_m(\frac{1}{3})}  \frac{\eta^k M_k}{h_m(\frac{\et}{6A})}.
\end{align*}
If $\de:= 1/h_m(\frac{1}{3})$ then for every $\eta>0$,
\begin{equation*}
|\vh^{(k)}(t)| \le \frac{(\de \eta)^k M_k}{h_m(\frac{\et}{6A})}.
\end{equation*}
This implies (2), for $r=\ep=1$.
For the general case it suffices to compose $\vh_{\rh r \ep}$
with an
odd smooth function $\th : \R \to \R$ satisfying $\th([-r,r]) = [-1,1]$ and
$\th(x) = \frac{x + r\ep -r}{r\ep}$ for $x \ge r$.

(2) $\Rightarrow$ (1)
	Let $\vh$ be as in (2) for given $r>0$ and $\rh = \ep =1$.
	Then the function $\ps(x):= \vh(x-2r) \cdot h_m(Cr)$ satisfies $\|\ps^{(k)}\|_\R \le M_k$ for all $k$ and vanishes on $(-\infty,0]$.
	\Cref{lem:sizenearflat} implies
	\[
		h_m(Cr) = |\ps(r)| \le \Big( \frac{2r}{\sum_{k \ge \ell} \frac{1}{\mu_k}} \Big)^\ell, \quad \ell \ge 1,
	\]
	provided that $r< \frac{1}{4}\sum_{k \ge \ell} \frac{1}{\mu_k}$.
  For those $\ell$ that satisfy $\frac{1}{C \mu_\ell^*} < \frac{1}{4}\sum_{k \ge \ell} \frac{1}{\mu_k}$,
  we obtain
  \[
			h_m(\tfrac{1}{\mu^*_\ell}) \le \Big( \frac{\tfrac{2}{C}\tfrac{1}{\mu^*_\ell}}{\sum_{k \ge \ell} \frac{1}{\mu_k}} \Big)^\ell.
	\]
  On the other hand,
  \[
		h_m(\tfrac{1}{\mu^*_\ell}) \stackrel{\eqref{GaProp3}}{=} m_{\ell-1} (\tfrac{1}{\mu^*_\ell})^{\ell-1} \ge H^{\ell-1}
	\]
  for some constant $H>0$, since $m$ has moderate growth. So we may conclude that $\sum_{k \ge \ell} \frac{1}{\mu_k} \lesssim \frac{1}{\mu^*_\ell}$
  for all $\ell$.
\end{proof}

\subsection{Optimal cutoff functions of Beurling type}

There exist also optimal cutoff functions of Beurling type. They will play a crucial role in \Cref{sec:extentionop}.

\begin{theorem}
  \label{thm:cutoffBM}
  Let $M$ be a strongly regular weight sequence.
  Then there exist functions $(\vh_r)_{r>0}$ in $\cE^{(M)}(\R)$ with the following properties:
  \begin{enumerate}
    \item $0 \le \vh_r \le 1$ for all $r>0$.
    \item $\vh_r|_{[-r,r]} = 1$ and $\supp \vh_r \subseteq [- \frac{9}8r,\frac{9}8r]$ for all $r>0$.
    \item For each $\rh>0$ there exist constants $A_\rh>0$ and $b_\rh>0$, such that
    \begin{equation} \label{eq:cutoffBM}
      \|\vh_r\|^M_{\R,\rh} \le  \frac{A_\rh}{h_m(b_\rh r)}, \quad r>0.
     \end{equation}
  \end{enumerate}
\end{theorem}

The essential difference to \Cref{thm:optimalbump} is that the cutoff functions may be taken independent of $\rh$.
The dependence of the argument of $h_m$ on $\rh$ in \eqref{eq:cutoffRM} is however simpler than in \eqref{eq:cutoffBM};
the linear dependence in \eqref{eq:cutoffRM} will be used in the proof of \Cref{thm:WRoumieu}.

\begin{proof}
We may assume the setup of \Cref{cor:setupcutoff}, in particular,
there is a constant $A>0$ such that \eqref{eq:cutoffk} holds.
Let $0<r \le 8A$ be fixed and take an integer $k=k(r) \ge 1$ such that
\[
  \frac{8A}{\mu_{k}^*} < r \le \frac{8A}{\mu_{k-1}^*}.
\]
By \Cref{prop:cutoff} (applied to $K=[-r,r]$ and $\sum d_j$ the left-hand side of \eqref{eq:cutoffk}),
there exists $\vh=\vh_r \in C^\infty(\R)$
with
$0 \le \vh \le 1$, $\vh|_{[-r,r]} = 1$, $\on{supp} \vh \subseteq [-\frac{9}8r,\frac{9}8r]$, and
\begin{align*}
  |\vh_r^{(j)}| \le
  \begin{cases}
    2^j \mu_k^j & \text{ if } 0 \le j \le k,
    \\
    2^j \mu_k^k \frac{M_j}{M_k} \Big( \frac{k^{j-k} k!}{j!}\Big)^\ep & \text{ if } j>k.
  \end{cases}
\end{align*}
Since $m$ has moderate growth, we have $\mu^*_k \le C \mu^*_{k-1}$ for some $C \ge 1$. Thus
\begin{equation*}
  \frac{\mu_k^k}{M_k} = \frac{\mu_k^k}{k!\,m_k} \le \frac{(e\mu_k^*)^k}{m_k} \le \frac{(Ce\mu_{k-1}^*)^k}{m_k}
  \le \Big(\frac{8ACe}{r}\Big)^k \frac{1}{m_k}.
\end{equation*}
Let $\rh \in (0,1)$ be given.
If $j \le k$, then $\mu_k^j \le \mu_k^k \frac{M_j}{M_k}$, and hence
\begin{align*}
   |\vh_r^{(j)}| \le 2^j M_j \Big(\frac{8ACe}{r}\Big)^k \frac{1}{m_k}
   \le \rh^j M_j \Big(\frac{16ACe}{r\rh}\Big)^k \frac{1}{m_k}.
\end{align*}
If $j> k$, then $\frac{k^{j-k} k!}{j!} \le \frac{k^{j-k}}{(j-k)!}$, and so
\begin{align*}
   |\vh_r^{(j)}| &\le 2^j M_j \Big(\frac{8ACe}{r}\Big)^k \frac{1}{m_k}  \Big( \frac{k^{j-k} }{(j-k)!}\Big)^\ep
   \\
   &= \rh^j M_j \Big(\frac{16ACe}{r \rh}\Big)^k \frac{1}{m_k}  \Big( \frac{(k(\frac{2}{\rh})^{1/\ep})^{j-k} }{(j-k)!}\Big)^\ep
   \\
   &\le \rh^j M_j \Big(\frac{16ACe}{r \rh}\Big)^k \frac{1}{m_k} e^{\ep k (\frac{2}{\rh})^{1/\ep}}
   \\
   &= \rh^j M_j  \frac{1}{(b_\rh r)^k m_k}, \quad \text{ where }\quad b_\rh := \frac{\rh}{16ACe^{1+ \ep (\frac{2}{\rh})^{1/\ep}}}.
\end{align*}
It follows that, for all $0 < r \le 8A$ and all $\rh >0$,
\begin{align*}
   \|\vh_r\|^M_{\R,\rh} \le  \frac{1}{h_m(b_\rh r)};
\end{align*}
for $\rh\ge 1$ the estimates are much simpler.
For $r> 8A$ we apply \Cref{prop:cutoff} to $K=[-r,r]$ and \eqref{eq:cutoffk} with $k=1$.
This gives $\vh_r$ having the same bounds as $\vh_{8A}$. Then
\[
  \|\vh_r\|^M_{\R,\rh} \le  \frac{1}{h_m(b_\rh 8A)} \le  \frac{1}{h_m(b_\rh 8A)}  \frac{1}{h_m(b_\rh r)} =: \frac{A_\rh}{h_m(b_\rh r)},
\]
since $h_m \le 1$.
\end{proof}

\subsection{Whitney cubes}

We denote by $d_A(z) =d(z,A) = \inf\{|z-y| : y \in A\}$ the Euclidean distance of $z \in \R^n$ to some set $A \subseteq \R^n$.

\begin{proposition}[{\cite[pp.167--170]{Stein70}, \cite{Bruna80}}] \label{prop:Wcubes}
	For each closed non-empty set $A \subseteq \R^n$ there exists a countable collection of closed cubes $(Q_j)_{j \ge 1}$ with sides parallel to the coordinate
	axes and the following properties.
	\begin{enumerate*}
		\item $\R^n \setminus A = \bigcup_{j\ge 1} Q_j$.
		\item The interiors of the cubes $Q_j$ are pairwise disjoint.
		\item For each $j \ge 1$ we have $$\on{diam} Q_j \le \on{dist} (Q_j,A) \le 4 \on{diam} Q_j.$$
		\item For each $j\ge 1$ let $Q_j^*$ be image of $Q_j$ under the dilation by the factor $\tfrac{9}{8}$ with center the center of $Q_j$.
		There exist constants $0< b_0 \le 1 \le B_0$ such that for all $j \ge 1$ and all $z \in Q_j^*$,
		\[
			b_0 \on{diam} Q_j \le d_A(z) \le B_0 \on{diam} Q_j.
		\]
		\item For each $j\ge 1$ the number of $i \ge 1$ with $Q_i^* \cap Q_j^* \ne \emptyset$ is bounded by $12^{2n}$.
		\item There exist constants $0< b_1 \le 1 \le B_1$ such that for all $i,j\ge 0$ with $Q_i^* \cap Q_j^* \ne \emptyset$ we have
		\[
			b_1 \on{diam} Q_j \le \on{diam} Q_i \le B_1 \on{diam} Q_j.
		\]
	\end{enumerate*}
	The constants $b_0,B_0,b_1,B_1$ are independent of $A$.
\end{proposition}

A collection $(Q_j)_{j\ge 1}$ with these properties is called a \emph{family of Whitney cubes for $A$.}

\subsection{Optimal partitions of unity}

Having optimal cutoff functions, it is a standard procedure to construct
optimal partitions of unity subordinate to a given family of Whitney cubes.
We begin with the Roumieu case.

\begin{proposition}\label{Proposition6matrix}
Let $M$ be a strongly regular weight sequence.
Let $A \subseteq \R^n$ be a non-empty closed set and $(Q_j)_{j \ge 1}$ a family of Whitney cubes for $A$.
Then there exists
$C_1\ge 1$ such that for all $\ep>0$ there is a family of $C^\infty$-functions
$(\vh_{j,\ep})_{j \ge 1}$ satisfying
\begin{enumerate}
\item $0\le\vh_{j,\ep}\le 1$ for all $j \ge 1$,
\item $\on{supp}\vh_{j,\ep}\subseteq Q^*_j$ for all $j\ge 1$,
\item $\sum_{j\ge 1}\vh_{j,\ep}(x)=1$ for all $x \in \R^n\setminus A$,
\item for all $j\ge 1$, $\be \in\N^n$, and $x\in\R^n\setminus A$,
\begin{equation*}
  |\vh^{(\be)}_{j,\ep}(x)| \le \frac{\ep^{|\be|} M_{|\be|}}{h_m(C_1\ep\, d_A(x))}.
\end{equation*}
\end{enumerate}
\end{proposition}

\begin{proof}
  By \Cref{thm:optimalbump}, there is a constant $C >1$ such that for each $\rh>0$ and each $r>0$
  there exist functions $\ch_{\rh,r}$ such that
  \begin{itemize}
  \item $0\le\ch_{\rh,r}(x)\le 1$ for all $x\in\R^n$,
  \item $\ch_{\rh,r}(x)=1$ for all $x\in [-r,r]^n$,
  \item $\ch_{\rh,r}(x)=0$ for all $x\not\in [-\frac{9}{8}r,\frac{9}{8}r]^n$,
  \item for all $\be \in\N^n$ and all $x\in\R^n$,
  \begin{equation*}
    |\ch^{(\be)}_{\rh,r}(x)|\le \frac{\rh^{|\be|} M_{|\be|}}{h_m(C\rh r)}.
  \end{equation*}
  \end{itemize}
  Let $x_j$ be the center and $2 s_j$ the sidelength of $Q_j$. Set
   \[
      \ps_{j,\ep}(x) := \ch_{\frac{\ep}{12^{2n}},s_j}(x-x_j).
   \]
  Then $\ps_{j,\ep}$ is $1$ on $Q_j$ and $0$ outside $Q_j^*$.
  Moreover,
  \begin{align*}
   |\ps^{(\be)}_{j,\ep}(x)| \le \frac{(\frac{\ep}{12^{2n}})^{|\be|} M_{|\be|}}{h_{m}(  \frac{C\ep s_j}{12^{2n}} )}
   \le \frac{(\frac{\ep}{12^{2n}})^{|\be|} M_{|\be|}}{h_{m}(C_0 \ep\, d_A(x))}, \quad \text{ for } C_0 := \frac{C}{2 \sqrt{n} 12^{2n} B_0},
  \end{align*}
  since $d_A(x) \le B_0 \on{diam} Q_j = 2 \sqrt{n} B_0 s_j$ if $x \in Q^*_j$, by \Cref{prop:Wcubes}(4).
  Now define $(\vh_{j,\ep})_{j \ge 1}$ by
 \begin{equation} \label{eq:defpartition}
   \vh_{1,\ep}:=\psi_{1,\ep}, \quad \vh_{j,\ep}:=\ps_{j,\ep} \prod_{k=1}^{j-1}(1-\psi_{k,\ep}), ~ j\ge 2.
 \end{equation}
 It is clear that $0\le\vh_{j,\ep}\le 1$ and $\on{supp}\vh_{j,\ep}\subseteq Q^*_j$ for all $j \ge 1$.
 To see (3) observe that, for all $j \ge 1$,
 \[
   \vh_{1,\ep} + \cdots + \vh_{j,\ep} = 1 - (1-\ps_{1,\ep}) \cdots (1-\ps_{j,\ep}).
 \]
 If $x \in \R^n \setminus A$, then $x \in Q_{i_x}$ for some $i_x\ge 1$ and thus $\ps_{i_x,\ep}(x) =1$.
 Since the family $(\supp \vh_{j,\ep})_{j\ge 1}$ is locally finite, it follows that $\sum_{j\ge 1}\vh_{j,\ep}(x)=1$ for all $x \in \R^n\setminus A$.
 Finally, let us check (4).
 In the product defining $\vh_{j,\ep}$ at most $12^{2n}$ factors are different from $1$, by \Cref{prop:Wcubes}(5).
 Thus
 \[
   |\vh^{(\be)}_{j,\ep}(x)|
   \le \frac{\ep^{|\be|} M_{|\be|}}{h_{m}(C_0 \ep\, d_A(x))^{12^{2n}}}.
 \]
   By \Cref{claim1}, there is a constant $D>1$ such that $h_m(t) \le h_m(Dt)^{12^{2n}}$ for all $t>0$.
 Thus, we obtain (4) with $C_1 = C_0/D$.
\end{proof}

Let us come to the Beurling case.

\begin{proposition}\label{Proposition6matrixB}
  Let $M$ be a strongly regular weight sequence.
Let $A \subseteq \R^n$ be a non-empty closed set and $(Q_j)_{j \ge 1}$ a family of Whitney cubes for $A$.
Then there is a family of $C^\infty$-functions
$(\vh_{j})_{j \ge 1}$ satisfying
\begin{enumerate}
\item $0\le\vh_{j}\le 1$ for all $j \ge 1$,
\item $\on{supp}\vh_{j}\subseteq Q^*_j$ for all $j\ge 1$,
\item $\sum_{j\ge 1}\vh_{j}(x)=1$ for all $x \in \R^n\setminus A$,
\item for each $\rh>0$ there exist $C_1,\ta_1>0$ such that
for all $j\ge 1$, $\be \in\N^n$, and $x\in\R^n\setminus A$,
\begin{equation*}
  |\vh^{(\be)}_{j}(x)| \le C_1 \frac{\rh^{|\be|} M_{|\be|}}{h_m(\ta_1\, d_A(x))}.
\end{equation*}
\end{enumerate}
\end{proposition}

\begin{proof}
   By \Cref{thm:cutoffBM}, there exist smooth functions $(\ch_r)_{r>0}$ such that
   \begin{itemize}
   \item $0\le\ch_{r}(x)\le 1$ for all $x\in\R^n$,
   \item $\ch_{r}(x)=1$ for all $x\in [-r,r]^n$,
   \item $\ch_{r}(x)=0$ for all $x\not\in [-\frac{9}{8}r,\frac{9}{8}r]^n$,
   \item for each $\rh>0$ there exist $A_\rh\ge 1$ and $b_\rh>0$ such that for all $\be \in\N^n$ and $x\in\R^n$,
   \begin{equation*}
     |\ch^{(\be)}_{r}(x)|\le A_\rh \frac{\rh^{|\be|} M_{|\be|}}{h_m(b_\rh r)}.
   \end{equation*}
 \end{itemize}
   If $x_j$ denotes the center of the cube $Q_j$ and $2 s_j$ its sidelength, we set
   \[
      \ps_{j}(x) := \ch_{s_j}(x-x_j).
   \]
   Then
   $\ps_{j}|_{Q_j} =1$, $\on{supp} \ps_j \subseteq Q_j^*$, and
   \begin{align*}
    |\ps^{(\be)}_{j}(x)| \le  A_\rh \frac{\rh^{|\be|} M_{|\be|}}{h_{m}( b_\rh s_j)}
    \le A_\rh \frac{\rh^{|\be|} M_{|\be|}}{h_{m}(b_\rh' d_A(x))}.
   \end{align*}
   In analogy to \eqref{eq:defpartition} we define the family $(\vh_j)_{j\ge 1}$ which clearly has the properties (1)--(3).
   The fact that at most $12^{2n}$ factors in the product defining $\vh_j$ are different from $1$ together with the estimates for $\ps_j$
   and \eqref{hmg}
   easily imply (4).
\end{proof}

\section{Extension of Whitney ultrajets}

If we restrict an $\cE^{[M]}$-function defined on $\R^n$ to a closed subset of $\R^n$, then Taylor's theorem
implies necessary conditions which lead to the notion of Whitney ultrajets.

\subsection{Whitney ultrajets}

Let $M=(M_k)$ be a weight sequence.
Let $K \subseteq \R^n$ be a compact set.
A \emph{Whitney ultrajet of class $\cE^{\{M\}}$ on $K$} is a Whitney jet
$F= (F^\al)_{\al} \in \cE(K)$ such that
there exist $C>0$ and $\rh \ge 1$
with
\begin{gather}
  |F^\al(x)| \le C \rh^{|\al|}   M_{|\al|}, \quad \al \in \N^n,~ x \in K,
   \label{jets1}
  \\
  |(R^p_x F)^\al(y)| \le C \rh^{p+1} \, |\al|!\, m_{p+1}\,  |x-y|^{p+1-|\al|}, \quad p \in \N,\, |\al| \le p,~ x,y \in K.
  \label{jets2}
\end{gather}
A \emph{Whitney ultrajet of class $\cE^{(M)}$ on $K$} is a Whitney jet
$F= (F^\al)_{\al} \in \cE(K)$ with the property that for each $\rh>0$ there exists $C>0$ such that
\eqref{jets1} and \eqref{jets2} hold.

The set of Whitney ultrajets of class $\cE^{[M]}$ on $K$ is denoted by $\cE^{[M]}(K)$.
For a closed set $A \subseteq \R^n$ let
\[
  \cE^{[M]}(A) := \{F \in \cE(A) : F|_K \in \cE^{[M]}(K) \text{ for all compact } K \subseteq A\}
\]
be the set of \emph{Whitney ultrajets of class $\cE^{[M]}$ on $A$}.

We endow $\cE^{(M)}(A)$ with the projective limit topology with respect to the system of seminorms
$p_{K,1/\ell}(F) := \|F\|^M_{K,1/\ell} + |F|^M_{K,1/\ell}$,  $\ell \in \N_{\ge 1}$ and $K \subseteq A$ compact,
where
\begin{equation} \label{eq:jetseminorm1}
 	\|F\|^M_{K,\rh} := \sup_{\al \in \N^n} \frac{\|F^\al\|_K}{\rh^{|\al|} M_{|\al|}}
 \end{equation}
 and\footnote{Note that $|\al|!\, m_{p+1} \le \frac{M_{p+1}}{(p+1 -|\al|)!} \le 2^{p+1} |\al|!\, m_{p+1}$.}
 \begin{equation} \label{eq:jetseminorm2}
 	| F |^M_{K,\rh}  := \sup_{\substack{x,y \in K\\ x \ne y}} \sup_{p \in \N } \sup_{|\al| \le p}  |(R^p_x F)^\al (y)|
 	\frac{(p+1-|\al|)!}{|x-y|^{p+1-|\al|} \rh^{p+1} M_{p+1}}.
 \end{equation}
Similarly, the Roumieu space
\[
\cE^{\{M\}}(A) = \{F \in \cE(A) : \A K \subseteq_{cp} A \E m\in \N_{\ge 1} : p_{K,m}(F)<\infty \}
\]
is endowed with its natural locally convex topology.
Sometimes we shall need the Banach space $\cE^M_\rh(K) := \{F \in \cE(K): p_{K,\rh}(F)<\infty \}$.

\begin{remark}
	Let $k$ be a positive integer.
	A compact subset $K \subseteq \R^n$ is called \emph{Whitney $k$-regular} if there is a constant $C >0$
	such that for all $x,y \in K$ there exists a rectifiable path $\ga$ in $K$ joining $x$ and $y$ such that
	\[
		\on{length}(\ga) \le C |x-y|^{1/k}.
	\]
	It is not difficult to see that \eqref{jets1} implies \eqref{jets2} provided that $K$ is Whitney $1$-regular, e.g.,
	if $K$ is convex.
	This is no longer true if $K$ is only Whitney $k$-regular for some $k>1$; cf.\ \cite[1.9]{Plesniak:1994ti}.
\end{remark}

The goal of this section is to prove that Whitney ultrajets admit extensions preserving the class if the weight sequence $M$
is strongly regular. The proof of the Roumieu case is based on the existence of optimal cutoff functions and the extension problem for
the singleton (the Borel map). The Beurling case can be extracted from the Roumieu case by a reduction argument based on
\Cref{lem:CC16}.

In \Cref{sec:extentionop} we will prove that extension operators always exist in the Beurling case.
As a by-product it gives an alternative direct proof that extensions exist.

\subsection{Extension of Whitney ultrajets of Roumieu type}

Let $A \subseteq \R^n$ be a non-empty closed set and $F \in \cE(A)$ a Whitney jet on $A$.
A function $f \in C^\infty(\R^n)$ is said to be a \emph{local extension of $F$ in x} if
\[
  j^\infty_{\{x\}} f = F(x).
\]
The next lemma provides local extensions in a uniform way.

\begin{lemma}[{\cite{Bruna80}, \cite[Lemma 3.8]{BBMT91}}] \label{lem:BorelSilva}
	Let $M$ be a strongly regular weight sequence.
	Let $K\subseteq \R^n$ be a compact non-empty set and $F \in \cE^{\{M\}}(K)$.
	Then there exists $\rh>0$ such that
	for all $x \in K$ there exists $f_x \in \cB^{\{M\}}(\R^n)$ with $j^\infty_{\{x\}} f_x = F(x)$ and
	\begin{align} \label{eq:Wprep1}
			\sup_{x \in K} \|f_x\|^M_{\R^n,\rh} < \infty.
	\end{align}
\end{lemma}

\begin{proof}
	By assumption, $\{F(x) : x \in K\}$ is a bounded subset of $\La^{\{M\}}$.
	We know that $j^\infty_{\{0\}} : \cB^{\{M\}}(\R^n) \to \La^{\{M\}}$ is surjective, by \Cref{thm:Borelsufficient}.
	Since both spaces are Silva spaces, each bounded set in $\La^{\{M\}}$ is the image of a bounded set in $\cB^{\{M\}}(\R^n)$.
	Thus, there is a bounded set $B$ in $\cB^{\{M\}}(\R^n)$ such that for all $x \in K$
	there exists $f_x \in B$ with $j^\infty_{\{x\}} f_x = F(x)$.
  This also implies \eqref{eq:Wprep1}.
\end{proof}

\begin{lemma} \label{lem:prepW2}
	In the setting of the previous lemma,
	there exist $D,\si>0$
	such that for all
	$x,y \in K$, $z \in \R^n$
	and $\al \in \N^n$ we have
	\begin{equation} \label{eq:prepW2}
  	|(f_x - f_y)^{(\al)}(z)| \le D \si^{|\al|} M_{|\al|}  h_m(\si(|z-x|+|z-y|)).
  \end{equation}
\end{lemma}

\begin{proof}
	Let $x,y \in K$ be fixed. Take $p\in \N$ and $|\al| \le p$.
	We have
	\begin{align*}
		T_x^pF(z) - T_y^pF(z)
		&=	\sum_{\be \le p} \frac{(z-x)^\be}{\be!} (T_x^pF - T_y^pF)^{(\be)}(x)
		\\
		&=	\sum_{\be \le p} \frac{(z-x)^\be}{\be!} (R_y^pF - R_x^pF)^{\be}(x)
		= \sum_{\be \le p} \frac{(z-x)^\be}{\be!} (R^p_y F)^\be(x).
	\end{align*}
	Using \eqref{jets2}, we find that there exist $A,\ta>0$ such that
	\begin{align*}
		|(T_x^pF - T_y^pF)^{(\al)}(z)|
		&\le A \ta^{p+1} M_{p+1} \frac{(|z-x|+|x-y|)^{p+1-|\al|}}{(p+1-|\al|)!}
		\\
		&\le A (2\ta)^{p+1} M_{p+1} \frac{(|z-x|+|z-y|)^{p+1-|\al|}}{(p+1-|\al|)!}.
	\end{align*}
	By Taylor's formula and $j^\infty_{\{x\}} f_x = F(x)$,
	\begin{align*}
		f_x(z) - T_x^p F(z) = \sum_{|\be| = p+1} \frac{(z-x)^\be}{\be!}  f_x^{(\be)}(x+t(z-x)), \quad  t \in [0,1],
	\end{align*}
	and thus
	\begin{align} \label{eq:fxTxF}
		|(f_x - T_x^pF)^{(\al)} (z)| \le C \rh^{p+1} M_{p+1} \frac{|z-x|^{p+1-|\al|}}{(p+1-|\al|)!}, \quad |\al| \le p.
	\end{align}
	It follows that, for suitable constants $D,\si_1>0$,
	\begin{align*}
		|(f_x - f_y)^{(\al)}(z)| \le D \si_1^{p+1} M_{p+1} \frac{(|z-x|+|z-y|)^{p+1-|\al|}}{(p+1-|\al|)!},
	\end{align*}
	for all $x,y \in K$, $z \in \R^n$, $p \in \N$, and
	$|\al| \le p$. Using that $M$ has moderate growth, we infer
	\begin{align*}
		|(f_x - f_y)^{(\al)}(z)| \le D \si^{|\al|} M_{|\al|} \cdot (\si(|z-x|+|z-y|))^{p+1-|\al|} m_{p+1-|\al|}
	\end{align*}
  for all $p \ge |\al|$. This implies \eqref{eq:prepW2}.
\end{proof}

\begin{theorem} \label{thm:WRoumieu}
	Let $M$ be a strongly regular weight sequence.
	For each non-empty compact set $K \subseteq \R^n$ the mapping $j^\infty_K : \cE^{\{M\}}(\R^n) \to \cE^{\{M\}}(K)$ is surjective.
\end{theorem}

\begin{proof}
	Let $F \in \cE^{\{M\}}(K)$ satisfy \eqref{jets1} and \eqref{jets2}.
	By \Cref{lem:BorelSilva}, there are $f_x \in \cB^{\{M\}}(\R^n)$, $x \in K$, satisfying $j^\infty_{\{x\}}f_x = F(x)$ for all $x \in  K$
  and \eqref{eq:Wprep1}.
	Let $(Q_j)_{j\ge 1}$ be a family of Whitney cubes for $K$.
	Let $\ep>0$ (to be specified later).
	By \Cref{Proposition6matrix}, there is
	a family of $C^\infty$-functions $(\vh_{j,\ep})_{j\ge 1}$ satisfying
	$0 \le \vh_{j,\ep} \le 1$, $\on{supp} \vh_{j,\ep} \subseteq Q_j^*$, $\sum_j \vh_{j,\ep} =1$ on $\R^n \setminus K$,
	and
	\begin{equation} \label{estvh}
  	|\vh^{(\be)}_{j,\ep}(x)|
 	 \le \frac{\ep^{|\be|} M_{|\be|}}{h_{m}(C_1\ep\, d_K(x))}, \quad \be \in \N^n,~ x \in \R^n\setminus K,
	\end{equation}
	where $C_1>0$ is independent of $\ep$.
	Let $x_j$ denote the center of $Q_j$ and let
	$\hat x_j \in K$ be such that $|x_j - \hat x_j|= d_K(x_j)$. We define
\begin{equation} \label{eq:extensionformula}
  f(z) :=
  \begin{cases}
    \sum_{j \ge 1} \vh_{j,\ep}(z) \, f_{\hat x_j}(z),  & \text{ if } z \in \R^n \setminus K, \\
    F^0(z), & \text{ if } z \in  K.
  \end{cases}
\end{equation}
Clearly, $f$ is $C^\infty$ in $\R^n \setminus K$.

\begin{claim*}
There exist constants $D_1,\si_1>0$ such that, for $x \in K$,
\begin{align} \label{eq:Wclaim}
	|(f-f_x)^{(\al)} (z)| \le D_1 \si_1^{|\al|} M_{|\al|} h_m(\si_1 |z-x|), \quad z \in \R^n \setminus K.
\end{align}
\end{claim*}

We must estimate
\begin{align} \label{eq:toest}
	(f-f_x)^{(\al)} (z) = \sum_{\be \le \al} \binom{\al}{\be} \sum_{j \ge 1} \vh^{(\be)}_{j,\ep}(z) \, (f_{\hat x_j} -f_x)^{(\al-\be)}(z).
\end{align}
Let us first treat the term with $\be =0$.
If $z \in \on{supp} \vh_{j,\ep} \subseteq Q_j^*$, then
\begin{align} \label{consequenceWcubes}
  |z-\hat x_j|\le |z-x_j| + |x_j-\hat x_j|
  \le \frac{9}{8} \on{diam} Q_j + d_K(x_j) \lesssim  d_K(z) \le   |z-x|,
\end{align}
 by \Cref{prop:Wcubes}(3)--(4).
Thus, by \Cref{lem:prepW2}, there exist $D,\si>0$ such that
\begin{align*}
	\sum_{j\ge 1} \vh_{j,\ep}(z) \, \big|(f_{\hat x_j} -f_x)^{(\al)}(z)\big| \le D \si^{|\al|} M_{|\al|} h_m(\si |z-x|).
\end{align*}
Now we consider the terms with $\be \ne 0$ in \eqref{eq:toest}. To this end let $\hat z$ be a point in $K$ with $|\hat z - z| = d_K(z)$.
Since $\be \ne 0$, $\sum_{j \ge 1} \vh_{j,\ep}^{(\be)}(z) = 0$ and thus
\begin{align*}
  	\sum_{j\ge 1} \vh^{(\be)}_{j,\ep}(z) \, (f_{\hat x_j} -f_x)^{(\al-\be)}(z) = \sum_{j \ge 1} \vh^{(\be)}_{j,\ep}(z) \, (f_{\hat x_j} -f_{\hat z})^{(\al-\be)}(z).
  \end{align*}
We infer from \Cref{lem:prepW2}, \eqref{estvh}, \eqref{consequenceWcubes}, and \Cref{prop:Wcubes}(5) that there exist constants $D,\si>0$ such that
\begin{align*}
  	\sum_{j\ge 1} |\vh^{(\be)}_{j,\ep}(z)| | (f_{\hat x_j} -f_{\hat z})^{(\al-\be)}(z)|
  	&\le
  	12^{2n} D\ep^{|\be|}\si^{|\al-\be|} M_{|\be|}  M_{|\al - \be|} \frac{h_m(\si d_K(z))}{h_{m}(C_1\ep\, d_K(z))}
  	\\
  	&\stackrel{\eqref{hmg}}{\le}
  	12^{2n} D\ep^{|\be|}\si^{|\al-\be|} M_{|\al|} h_m(C\si d_K(z))
  \end{align*}
  if we choose $\ep:= C\si/C_1$, where $C$ is the constant from \eqref{hmg}. The claim follows.

In order to show that $f$ is a $C^\infty$-function on $\R^n$ with $j^\infty_K f = F$, let, for $\al \in \N^n$,
\[
	f^\al(z) :=
	\begin{cases}
		f^{(\al)}(z) & \text{ if } z \in \R^n \setminus K,
		\\
		F^\al(z) & \text{ if } z \in K.
	\end{cases}
\]
We prove that $f = f^0$ is a $C^\infty$-function with $f^{(\al)} = f^\al$ for all $\al \in \N^n$.
First we show that all $f^\al$ are continuous.
This is clear near points $x \not\in K$. So let $x \in K$.
Now if $z \in K$, then
\[
	|f^\al(z) - f^\al(x)| = |F^\al(z) - F^\al(x)| = |(R^{|\al|}_x F)^\al (z)| \to 0 \quad \text{ as } z\to x,
\]
by \eqref{jets2}.
If $z \not\in K$, then
\begin{align*}
	|f^\al(z) - f^\al(x)| &= |f^{(\al)}(z) - f_x^{(\al)}(x)|
	\\
	&\le |f^{(\al)}(z) - f_x^{(\al)}(z)| + |f_x^{(\al)}(z) - f_x^{(\al)}(x)| \to 0  \quad \text{ as } z\to x,
\end{align*}
by \eqref{eq:Wclaim} and since $f_x^{(\al)}$ is continuous.
At this point we may refer to the lemma of Hestenes \cite[Lemma 1]{Hestenes41} or argue as follows:
If $\de_i \in \N^n$ denotes the multiindex with $(\de_i)_j = \de_{ij}$ (the Kronecker symbol), then for $x\in K$ and $z \in \R^n$,
\begin{align*}
	|f^\al(z) - f^\al(x) - \sum_{i=1}^n (z_i - x_i) f^{\al+ \de_i}(x)| = |f^\al(z) - (T_x^{|\al|+1} F)^{(\al)}(z)| = o(|z-x|),
\end{align*}
by \eqref{jets2}, \eqref{eq:fxTxF}, and \eqref{eq:Wclaim}. It means that $f^\al$ is $C^1$ and $\p_{x_i} f^{\al} = f^{\al + \de_i}$.
The assertion follows.

Finally, we claim that $f \in \cE^{\{M\}}(\R^n)$. By \eqref{jets1}, it suffices to consider $z \not\in K$.
Then the claim follows from
\begin{align*}
	|f^{(\al)}(z)| \le  |f_{\hat z}^{(\al)} (z)| +
	|(f-f_{\hat z})^{(\al)} (z)|,
\end{align*}
\eqref{eq:Wprep1}, and \eqref{eq:Wclaim}.
\end{proof}

\subsection{Extension of Whitney ultrajets of Beurling type}

The Beurling case may be reduced to the Roumieu case by means of the next lemma.

\begin{lemma}[{\cite[Proposition 17]{ChaumatChollet94}}] \label{lem:CC17}
	Let $M=(M_k)$ be a strongly regular weight sequence.
	For any non-negative sequence $L=(L_k)$ with $L \lhd M$
	there exists a strongly regular weight sequence $N =(N_k)$ with $L \preceq N \lhd M$.
\end{lemma}

\begin{proof}
	There is a decreasing sequence $(\ep_k)_{k\ge 1}$ tending to $0$ with $L_k \le \ep_1\cdots \ep_k M_k$
	for all $k \ge 1$; it suffices to take $\ep_k:= \sup_{j\ge k} (\frac{L_k}{M_k})^{1/k}$.
	Since $M$ is strongly non-quasianalytic, there is a constant $C>0$ such that
	\begin{equation} \label{eq:red1}
    \sum_{j\ge k} \frac{1}{\mu_k} \le C \frac{k}{\mu_k}, \quad k \in \N.
  \end{equation}
	Applying \Cref{lem:CC16} to $\al_k =\frac{1}{\mu_k}$, $\be_k = \max \{\ep_k,\frac{k}{\mu_k}\}$, and $\ga_k = \frac{k}{\mu_k}$,
	we find a sequence $(\th_k)_{k \ge 1}$ with $\th_k \nearrow \infty$
	such that
	\begin{align} \label{eq:red2}
		\sum_{j \ge k} \frac{\th_j}{\mu_j} \le 8 \th_k \sum_{j \ge k} \frac{1}{\mu_j}, \quad \text{ for all } k \ge 1,
	\end{align}
	$\th_k\ep_k \to 0$, and  $\frac{k\th_k}{\mu_k} \searrow 0$. Define
	$\nu_k := \frac{\mu_k}{\th_k}$, $k \ge 1$, and $\nu_0 := 0$.
  Then $N_k := \nu_0 \nu_1 \cdots \nu_k$ is a strongly log-convex weight sequence
  that is strongly non-quasianalytic, by \eqref{eq:red1} and \eqref{eq:red2},
  and has moderate growth,
  \[
    \frac{N_{j+k}}{N_jN_k} = \frac{M_{j+k}}{M_jM_k} \frac{\th_1\cdots\th_{j}  \th_1\cdots\th_{k}}{\th_1\cdots\th_{j+k}}
    \le \frac{M_{j+k}}{M_jM_k}.
  \]
  To see $L \preceq N \lhd M$ observe that
  $(\frac{L_k}{N_k})^{1/k} = (\frac{L_k}{M_k})^{1/k} (\th_1\cdots\th_{k})^{1/k} \le (\th_1\ep_1 \cdots\th_{k}\ep_k)^{1/k}$
  is bounded and $(\frac{M_k}{N_k})^{1/k} = (\th_1 \cdots \th_k)^{1/k} \to \infty$.
\end{proof}

\begin{theorem} \label{thm:WBeurling}
	Let $M$ be a strongly regular weight sequence.
	For each non-empty compact set $K \subseteq \R^n$ the mapping $j^\infty_K : \cE^{(M)}(\R^n) \to \cE^{(M)}(K)$ is surjective.
\end{theorem}

\begin{proof}
	Let $F = (F^\al)_\al \in \cE^{(M)}(K)$. Set $L_0:= \|F^0\|_K$ and
	\[
		L_k := \max \Big\{ \sup_{|\al|= k}\|F^\al\|_K, \sup_{x\ne y \in K,\, |\al|\le k-1}
		\frac{k!\, |(R^{k-1}_x F)^\al(y)|}{|\al|! \, |y-x|^{k-|\al|}}  \Big\}, \quad k \ge 1.
	\]
	Then $L \lhd M$ and, by \Cref{lem:CC17}, there is a strongly regular weight sequence
	$N$ with $L \preceq N \lhd M$. So $F \in \cE^{\{N\}}(K)$ and it has an extension $f \in \cE^{\{N\}}(\R^n)$,
	by \Cref{thm:WRoumieu}, which is also an element of $\cE^{(M)}(\R^n)$.
\end{proof}

\subsection{Extension from closed sets}

It is now easy to get extension theorems from arbitrary closed subsets of $\R^n$.

\begin{theorem} \label{thm:closed}
	Let $M$ be a strongly regular weight sequence.
	For each non-empty closed set $A \subseteq \R^n$ the mapping $j^\infty_A : \cE^{[M]}(\R^n) \to \cE^{[M]}(A)$ is surjective.
\end{theorem}

\begin{proof}
 	This follows easily from \Cref{thm:WRoumieu,thm:WBeurling}, since $\cE^{[M]}$ admits partitions of unity. Indeed,
	fix $F \in \cE^{[M]}(A)$.
  For $k \in \N_{\ge 1}$ set $U_k := \{x \in \R^n :  k-2 < |x| < k\}$; note that $U_k$, $k \ge 2$, are open shells and
  $U_1 = \{x \in \R^n : |x|<1\}$.
  There exist functions $\vh_k \in \cE^{[M]}(\R^n)$ such that $0 \le \vh_k \le 1$, $\on{supp} \vh_k \subseteq U_k$, and
  $\sum_{k=1}^\infty \vh_k =1$.
  For each $k \in \N_{\ge 1}$ the jet $F_k := F|_{\ol U_k}$ belongs to
  $\cE^{[M]}(A\cap \ol U_k)$ and thus has an extension
  $f_k \in \cE^{[M]}(\R^n)$, i.e.,
  $j^\infty_{A\cap \ol U_k}(f_k) = F_k$.
  Then
  $f := \sum_{k=1}^\infty \vh_k f_k$
  is a function in $\cE^{[M]}(\R^n)$ since on any compact set the sum is finite.
  Each $x \in A$ belongs to at most two consecutive sets $U_\ell$, $U_{\ell+1}$.
  Thus, for each $\al \in \N^n$,
  \begin{align*}
    f^{(\al)}(x)
    &= \sum_{\be \le \al} \binom{\al}{\be}\vh^{(\be)}_\ell(x) f^{(\al-\be)}_\ell(x)
    + \sum_{\be \le \al} \binom{\al}{\be}\vh^{(\be)}_{\ell+1}(x) f^{(\al-\be)}_{\ell+1}(x)
    \\&= \sum_{\be \le \al} \binom{\al}{\be}\p^\be(\vh_\ell(x) + \vh_{\ell+1}(x))
     F^{\al-\be}(x)
    =  F^\al(x),
  \end{align*}
  since all summands with $|\be|>0$ vanish.
\end{proof}

\subsection{Analytic extensions}

The ultradifferentiable extension in \Cref{thm:closed} can be made real analytic
on the complement of $A$. This follows from a result of Schmets and Validivia \cite{Schmets:1999aa}.
A different proof due to Langenbruch \cite[Theorem 13]{Langenbruch03} is based on a general
approximation theorem of Whitney type (cf.\ \cite[Lemma 6]{Whitney34a}) which is interesting in its own right.
In fact it is a special case of a quite general approximation theorem based
on a surjectivity criterion for continuous linear maps between Fr\'echet spaces.

\begin{theorem}[{\cite[Theorem 7]{Langenbruch03}}] \label{thm:Whitneyapproximation}
	Let $M=(M_k)$ be a non-quasianalytic weight sequence.
	Let $\Om \subseteq \R^n$ be open and $\et: \Om \to (0,\infty)$ continuous.
	For any $f \in \cE^{(M)}(\Om)$ there exists $g \in \cH(\Om^*)$ such that
	\[
		| f^{(\al)}(x) - g^{(\al)}(x) | \le \et(x)^{|\al|+1} M_{|\al|} \quad \text{ for all } x \in \Om, ~ \al \in \N^n.
	\]
	Here $\Om^* := \{ z \in \C^n : \Re (z) \in \Om, \, |\Im(z)| < \on{dist}(\Re (z) , \p \Om) \}$.
	We may assume that $g|_\Om$ is real-valued.
\end{theorem}

\begin{theorem}[{\cite[Theorem 13]{Langenbruch03}}] \label{thm:analyticmodification}
	Let $M=(M_k)$ be a non-quasianalytic weight sequence and $\Om \subseteq \R^n$ open.
	There exists a continuous linear mapping $T : \cB^{[M]}(\R^n) \to \cB^{[M]}(\R^n)$ such that
	$T(f)$ is an analytic $\Om$-modification of $f$, i.e.,
  $T(f)|_\Om$ is real analytic and $j^\infty_A T(f) = j^\infty_A  f$, where $A:= \R^n \setminus \Om$.
\end{theorem}

\section{Continuous linear extension operators} \label{sec:extentionop}

In this section we address the question whether there exist continuous linear extension operators.
We concentrate on the Beurling case for which the question has an affirmative answer.
In \Cref{sec:extopRoumieu} we comment briefly on the Roumieu case which has a negative answer most of the time.

Let $M=(M_k)$ be a strongly regular weight sequence.
Let $A \subseteq \R^n$ be a non-empty closed set.
By \Cref{thm:closed}, the following sequence is exact:
\begin{equation}\label{shortexactsequenceA}
	\xymatrix{
	0 \ar[r] & \ker j^\infty_A ~ \ar@{^{(}->}[r] & \cE^{(M)}(\R^n) \ar[r]^{j^\infty_A} & \cE^{(M)}(A) \ar[r] & 0
	}
\end{equation}
We will show that the sequence \eqref{shortexactsequenceA} splits: there exists
a continuous linear right-inverse $E_A : \cE^{(M)}(A) \to \cE^{(M)}(\R^n)$ of $j^\infty_A$,
that is $j^\infty_A \o E_A = \on{id}_{\cE^{(M)}(A)}$. We say that $E_A$ is an \emph{extension operator}.

We give an elementary constructive proof based on the fact that the Borel map is split surjective
and on the existence of optimal cutoff functions.
In \Cref{sec:topinv} we shall briefly comment on the alternative approach based on
the splitting theorem of Vogt and Wagner \cite{Vogt:1980th}.

\subsection{The existence of extension operators is a local property}

We need a variant of \Cref{lem:prepW2}.

\begin{lemma} \label{lem:prepW3}
	Let $M$ be a strongly regular weight sequence.
	For each $\rh>0$ there exist $C,\si >0$ such that the following holds.
	If $K \subseteq \R^n$ is compact and $H \subseteq \R^n$ is compact convex with $K \subseteq \on{int} H$,
	$x_1,x_2 \in K$, and
	$f_1,f_2 \in C^\infty(\R^n)$ and $F \in \cE(K)$ satisfy $j^\infty_{\{x_1\}} f_1 = F(x_1)$, $j^\infty_{\{x_2\}} f_2 = F(x_2)$,
	and
	\[
		\|f_1\|^M_{H,\si} < \infty, \quad \|f_2\|^M_{H,\si} < \infty, \quad |F|^M_{K,\si} < \infty,
	\]
	then for all $z \in H \setminus K$ and $\al \in \N^n$,
	\begin{equation*}
		|(f_1 -f_2)^{(\al)}(z)| \le  (\|f_1\|^M_{H,\si} + \|f_2\|^M_{H,\si} + |F|^M_{K,\si}) \cdot C \rh^{|\al|} M_{|\al|} h_m(\rh (|z-x_1| + |z- x_2|)).
	\end{equation*}
\end{lemma}

\begin{proof}
	As in the proof of \Cref{lem:prepW2} we find, for $|\al| \le p$,
	\[
		|(T^p_{x_1}F - T^p_{x_2}F)^{(\al)}(z)| \le |F|^M_{K,\si} (2\si)^{p+1} M_{p+1} \frac{(|z-x_1| +|z-x_2|)^{p+1 -|\al|}}{(p+1 -|\al|)!}.
	\]
	Since $j^\infty_{\{x_i\}} f_i  = F(x_i)$,
	\begin{align*}
		|(f_i- T^p_{x_i} F)^{(\al)}(z)| \le \|f_i\|^M_{H,\si} \si^{p+1} M_{p+1} \frac{(|z-x_i|)^{p+1 -|\al|}}{(p+1 -|\al|)!}, \quad i=1,2,
	\end{align*}
	as in \eqref{eq:fxTxF}.
	Since $M$ has moderate growth, $M_{p+1} \le D^{p+1} M_{|\al|}M_{p+1 - |\al|}$ for some $D>0$. Then the triangle inequality gives
	for all $p \ge |\al|$,
	\begin{align*}
		|(f_1 -f_2)^{(\al)}(z)| &\le (2\si D)^{|\al|} M_{|\al|} m_{p+1-|\al|} (2\si D (|z-x_1| + |z- x_2|))^{p+1 - |\al|}
		\\
		&\qquad \times	 (\|f_1\|^M_{H,\si} + \|f_2\|^M_{H,\si} + |F|^M_{K,\si})
	\end{align*}
	and the assertion follows easily.
\end{proof}

Let $M$ be a strongly regular weight sequence.
Let $A \subseteq \R^n$ be a closed non-empty set.
A continuous linear mapping $E : \cE^{(M)}(A) \to \cE^{(M)}(\R^n)$ is said to be a \emph{local extension operator for $A$ in $x$}
if
\begin{equation*}
	j^\infty_{\{x\}} E(F) =  F(x), \quad \text{ for all } F \in \cE^{(M)}(A).
\end{equation*}
Clearly, $E$ is an extension operator for $A$ if it is a local extension operator for $A$ in \emph{every} $x \in A$.

\begin{proposition}[{\cite[Lemma 3.6]{Franken:1993tn}}] \label{prop:localproperty}
	Let $M$ be a strongly regular weight sequence.
	Let $A \subseteq \R^n$ be a closed non-empty set.
	Suppose that for each $x \in \p A$ there exists a local extension operator $E_x$ for $A$ in $x$
	and that $\{E_x : x \in \p A\}$ is locally equicontinuous in $L(\cE^{(M)}(A), \cE^{(M)}(\R^n))$.
	Then there exists an extension operator $E : \cE^{(M)}(A) \to \cE^{(M)}(\R^n)$.
\end{proposition}

\begin{proof}
	Let $(Q_j)_{j\in \N}$ be a family of Whitney cubes for $A$ (cf.\ \Cref{prop:Wcubes}) and let $(\vh_j)_{j\in \N}$ be the
	partition of unity provided by \Cref{Proposition6matrixB}.
  For each $j$ choose a point $x_j \in \p A$ such that $d_{Q_j}(x_j)=\on{dist}(A,Q_j)$.
  For $F \in \cE^{(M)}(A)$ define
	\[
		E(F)(z) :=
		\begin{cases}
			\sum_{j\ge1} \vh_j(z) E_{x_j}(F)(z) & \text{ if } z \in \R^n \setminus A,
			\\
			F^0(z) & \text{ if } z \in A.
		\end{cases}
	\]
  Clearly, $E(F)$ is of class $C^\infty$ on $\R^n \setminus A$.
  Recall that $p_{K,\si}(F) := \|F\|^M_{K,\si} + |F|^M_{K,\si}<\infty$ for all compact $K \subseteq A$ and $\si>0$.

  \begin{claim*}
		For each compact convex $H \subseteq \R^n$
	and for each $\rh>0$
	there exist $C,\si,\ta>0$ and a compact subset $K \subseteq A$
	such that
	for all $x \in \p A \cap H$, $z \in H \setminus A$ with $|x-z|\le \frac{b_0}{7}$ (with $b_0$ from \Cref{prop:Wcubes}),
	$F \in \cE^{(M)}(A)$, and $\al \in \N^n$,
	\begin{equation} \label{eq:estclaim}
		|(E(F) - E_x(F))^{(\al)}(z)| \le C\, p_{K,\si}(F)\, \rh^{|\al|} M_{|\al|} h_m(\ta|x-z|).
	\end{equation}
	\end{claim*}

  Let us prove the claim.
  By \Cref{claim1}, there is $C>1$ such that
  \begin{equation} \label{hmgR}
    h_m(t) \le h_m(Ct)^2, \quad t>0.
  \end{equation}
  Let a compact convex set $H \subseteq \R^n$ and $\rh>0$ be fixed.
  By \Cref{Proposition6matrixB}, there exist positive $C_1 =C_1(\rh)$ and $\ta_1 = \ta_1(\rh)$ such that, for all $j$,
  \begin{equation} \label{eq:estcutoffB}
    |\vh_j^{(\al)}(z)| \le C_1 \frac{\rh^{|\al|} M_{|\al|}}{h_m(\ta_1 d_A(z))},\quad \al\in \N^n,\, z \in \R^n.
  \end{equation}
  Set
  \begin{equation} \label{eq:rhotwo}
  \rh_2 := \min\{\rh, \tfrac{b_0\ta_1}{7C} \}.
  \end{equation}
  We invoke \Cref{lem:prepW3} for $\rh_2$: there exist
  $C_2,\si_2>0$ such that the assertion of the lemma
  holds for $\rh_2$, $C_2$, $\si_2$ (instead of $\rh$, $C$, $\si$ in the lemma).

  For $\ep>0$ let $H_\ep := \{z \in \R^n : d_H(z) \le \ep\}$ be the closed $\ep$-neighborhood of $H$.
  By assumption, the set $\{E_x : x \in \p A \cap H_2\}$ is equicontinuous in $L(\cE^{(M)}(A), \cE^{(M)}(\R^n))$.
  It follows that there exist constants $\si_3 \ge \si_2$, $C_3\ge 1$ and a compact set $K \subseteq A$ with $A \cap H_2 \subseteq K$
  such that
  \begin{equation*}
    \|E_x(F)\|^M_{H_2,\si_2} \le C_3 \, p_{K,\si_3}(F), \quad F \in \cE^{(M)}(A), ~ x \in \p A \cap H_2.
  \end{equation*}
  Then, by the assertion of \Cref{lem:prepW3}, for all $x,y \in \p A \cap H_1$, $z \in H_1 \setminus A$ and $\al \in \N^n$,
  \begin{align} \label{eq:equic}
    \MoveEqLeft
     |(E_x(F) - E_y(F))^{(\al)}(z)| \le C_4\, p_{K,\si_3}(F)\, \rh_2^{|\al|} M_{|\al|}  h_m(\rh_2(|z-x|+|z-y|)),
  \end{align}
  for some $C_4 \ge 1$.

  We are ready to estimate
  \[
    (E(F) - E_x(F))^{(\al)}(z) = \sum_{\be \le \al} \binom{\al}{\be} \sum_{j \ge 1} \vh_j^{(\be)}(z) (E_{x_j}(F) - E_x(F))^{(\al-\be)}(z),
  \]
  where $x \in \p A \cap H$ and $z \in H \setminus A$ with $|x-z|\le \frac{b_0}{7}$.
  By \Cref{prop:Wcubes}, for each $j \ge 1$ and $z \in Q_j^*$,
  \begin{align*}
    |x_j - z| &\le d_{Q_j}(x_j) + \on{diam} Q_j^*
    = \on{dist}(A,Q_j) + \tfrac{9}{8} \on{diam} Q_j \le 6 \on{diam} Q_j \le \tfrac{6}{b_0} d_A(z).
  \end{align*}
	Consequently (since $b_0 \le 1$),
	\[
		|x-x_j| \le |x-z| + |z-x_j| \le |x-z| + \tfrac{6}{b_0} d_A(z)  \le \tfrac{7}{b_0} |x-z| \le 1
	\]
	and thus $x_j \in H_1$.
  So, by \eqref{eq:equic}, for the term with $\be =0$ we have
	\begin{align*}
		\MoveEqLeft
		\sum_{j \ge 1} |\vh_j(z)| |(E_{x_j}(F) - E_x(F))^{(\al)}(z)|
			\\
			&\le
			C_4 \, p_{K,\si_3}(F)\, \rh_2^{|\al|} M_{|\al|}  h_m(\rh_2(|z-x|+|z-x_j|))
			\\
			&\le
			C_4 \, p_{K,\si_3}(F)\, \rh_2^{|\al|} M_{|\al|}  h_m(\tfrac{7\rh_2}{b_0} |z-x|)
			\\
      &\stackrel{\eqref{eq:rhotwo}}{\le}
			C_4 \, p_{K,\si_3}(F)\, \rh^{|\al|} M_{|\al|}  h_m(\ta_1 |z-x|).
	\end{align*}
  To estimate the terms with $\be \ne 0$, let $y \in \p A$ be such that $|y-z| = d_A(z)$.
  Then $\sum_{j \ge 1} \vh_j^{(\be)}(z) = 0$ implies
  \begin{align*}
    \sum_{j\ge 1} \vh^{(\be)}_{j}(z) \, (E_{x_j}(F) -E_x(F))^{(\al-\be)}(z) = \sum_{j\ge 1} \vh^{(\be)}_{j}(z) \, (E_{x_j}(F) -E_y(F))^{(\al-\be)}(z).
  \end{align*}
  We have $|y -z| = d_A(z) \le |z-x| \le \frac{b_0}7$ and thus $y \in H_1$.
  Moreover, if $z \in Q_j^*$, then
  \[
     |z-x_j| + |z-y| \le \tfrac{6}{b_0} d_A(z) + d_A(z)  \le \tfrac{7}{b_0} d_A(z) \le \tfrac{7}{b_0} |x-z| \le 1.
  \]
  Thus, by \eqref{eq:estcutoffB} and \eqref{eq:equic},
  \begin{align*}
    \MoveEqLeft
    \sum_{j \ge 1} |\vh_j^{(\be)}(z)| |(E_{x_j}(F) - E_y(F))^{(\al-\be)}(z)|
    \\
    &\le 12^{2n} C_1 C_4 \, p_{K,\si_3}(F)\, \rh^{|\be|} M_{|\be|} \rh_2^{|\al-\be|} M_{|\al-\be|} \frac{ h_m(\rh_2(|z-x_j|+|z-y|))}{h_m(\ta_1 d_A(z))}
    \\
    &\stackrel{\eqref{eq:rhotwo}}{\le} 12^{2n} C_1 C_4 \, p_{K,\si_3}(F)\, \rh^{|\al|} M_{|\al|}  \frac{ h_m(\frac{\ta_1}{C}d_A(z))}{h_m(\ta_1 d_A(z))}
    \\
    &\stackrel{\eqref{hmgR}}{\le} 12^{2n} C_1 C_4 \, p_{K,\si_3}(F)\, \rh^{|\al|} M_{|\al|}  h_m(\ta_1 d_A(z)).
  \end{align*}
  Then the claim follows easily.

  We are ready to complete the proof of the proposition.
  That $E(F)$ is a $C^\infty$-function on $\R^n$ with $j^\infty_A E(F) = F$ follows in analogy to the proof of \Cref{thm:WRoumieu}.

  It remains to show that $E(F) \in \cE^{(M)}(\R^n)$ and that the linear map $E : \cE^{(M)}(A) \to \cE^{(M)}(\R^n)$ is continuous.
  Let $H \subseteq \R^n$ be a compact convex set and $\rh>0$.
  Let $H_{b_0/7}$ be the closed $\tfrac{b_0}7$-neighborhood of $H$.
  By the claim, there exist $C,\si,\ta>0$ and a compact set $K \subseteq A$ such that for all $x \in \p A \cap H_{b_0/7}$, $z \in H_{b_0/7} \setminus A$
  with $|x-z| \le \tfrac{b_0}7$, $F \in \cE^{(M)}(A)$ and $\al \in \N^n$ we have \eqref{eq:estclaim}, in particular,
  \begin{equation} \label{eq:estclaimR}
		|(E(F) - E_x(F))^{(\al)}(z)| \le C\, p_{K,\si}(F)\, \rh^{|\al|} M_{|\al|}.
	\end{equation}
  If $z \in H \setminus A$ satisfies $d_A(z) \le \tfrac{b_0}7$, then there is $x \in \p A \cap H_{b_0/7}$ with $d_A(z) = |x-z|$.
  For such $z$ and $\al \in \N^n$,
  \begin{align*}
    |E(F)^{(\al)}(z)| &\le |E_x(F)^{(\al)}(z)| + |E(F)^{(\al)}(z) - E_x(F)^{(\al)}(z)|
    \\
    &\le C \, p_{K,\si}(F)\,  \rh^{|\al|} M_{|\al|},
  \end{align*}
  by \eqref{eq:estclaimR} and by the assumption on local equicontinuity of $E_x$ (increasing $K$, $C$, and $\si$ if necessary).
  Together with \eqref{jets1} this shows the assertion.
  The proof is complete.
\end{proof}

\subsection{Extension operators always exist in the Beurling case}

\begin{theorem} \label{thm:WSsplit}
	Let $M$ be a strongly regular weight sequence.
	Let $A \subseteq \R^n$ be a non-empty closed set.
	There exists an extension operator $E : \cE^{(M)}(A) \to \cE^{(M)}(\R^n)$.
\end{theorem}

\begin{proof}
	This follows from \Cref{cor:extopsingleton} and \Cref{prop:localproperty}.
\end{proof}

\subsection{A different approach based on a splitting theorem for Fr\'echet spaces} \label{sec:topinv}

Let $M=(M_k)$ be a strongly regular sequence.
Let $K$ and $L$ be a non-empty compact subsets of $\R^n$ such that $K \subseteq \on{int} L$.
By \Cref{thm:WBeurling}, the sequence
\begin{equation}\label{shortexactsequence}
	\xymatrix{
	0 \ar[r] & \cD^{(M)}(L,K)~ \ar@{^{(}->}[r] & \cD^{(M)}(L) \ar[r]^{j^\infty_K} & \cE^{(M)}(K) \ar[r] & 0
	}
\end{equation}
is an exact sequence of Fr\'echet spaces,
where we endow
\[
	\cD^{(M)}(L,K) := \{f \in \cE^{(M)}(\R^n) : \on{supp} f \subseteq L,\, j^\infty_K f = 0\}
\]
and $\cD^{(M)}(L)$
with the subspace topology of $\cE^{(M)}(\R^n)$.
By the splitting theorem of Vogt and Wagner \cite{Vogt:1980th} (see also \cite[30.1]{MeiseVogt97}),
the sequence splits if $\cE^{(M)}(K)$ has property (DN) and
 $\cD^{(M)}(L,K)$
has property ($\Om$).

A Fr\'echet space $E$ with fundamental system of seminorms $\sP = (p_j)_{j \in \N}$
has \emph{property} (DN) if
\[
  \E m \in \N \A k \in \N \E \ell \in \N \E C>0 : p_k^2 \le C p_m p_\ell.
\]
Then $p_m$ is a norm and every norm with this property is called a \emph{dominating norm}.
It has \emph{property} ($\Om$) if
\[
  \A m \in \N \E k \in \N \A \ell \in \N \E C>0 \E \th \in (0,1) : p_k^* \le C  (p_m^*)^{1-\th} (p_\ell^*)^\th,
\]
where $p^*_k(y) := \sup\{|y(x)| : p_k(x)\le 1\}$ is the dual norm of $p_k$.
The properties (DN) and ($\Om$) are linear topological invariants.
The property (DN) is inherited by all closed subspaces.\footnote{Hence, the splitting of
\eqref{shortexactsequence} implies that $\cE^{(M)}(K)$ has the property (DN).}
A nuclear Fr\'echet space $E$ has property (DN) if and only if $E$ is isomorphic to a closed subspace of $s$ (the space of rapidly decreasing sequences).
The property ($\Om$) is inherited by all quotient spaces.
A nuclear Fr\'echet space $E$ has property ($\Om$) if and only if $E$ is isomorphic to a quotient space of $s$.
Cf.\ \cite[29]{MeiseVogt97}.

That $\cE^{(M)}(K)$ has property (DN) is a triviality.
The fundamental system of seminorms
	\[
		p_m(F):= \|F\|^M_{K,1/m} + |F|^M_{K,1/m}, \quad m \in \N_{\ge 1},
	\]
cf.\ \eqref{eq:jetseminorm1} and \eqref{eq:jetseminorm2}, satisfies $p_k^2 \le p_1 p_{k^2}$ for all $k\ge 1$,
since
 \begin{align*}
 	\Big(\frac{k^{|\al|}}{M_{|\al|}} \Big)^2 = \frac{k^{2|\al|}}{M_{|\al|}} \frac{1^{|\al|}}{M_{|\al|}}.
 \end{align*}
That $\cD^{(M)}(L,K)$ has property ($\Om$) follows from a result by Franken \cite{Franken:1993tn} (building on work of Meise and Taylor \cite{Meise:1989un}).
Both papers are situated in the setting of Braun--Meise--Taylor classes of which the Denjoy--Carleman classes form a subclass if $M=(M_k)$ is a
strongly regular weight sequence; see \Cref{ex:strongWF}(2).

\subsection{Continuous linear extension operators in the Roumieu case} \label{sec:extopRoumieu}

As already pointed out in \Cref{rem:extopBorelRoumieu} in the case of the singleton, extension operators rarely exist in the Roumieu case.
For instance, the condition \eqref{eq:PetzscheRI} due to \cite{Petzsche88}
shows that $j^\infty_{\{0\}} : \cE^{\{M\}}(\R^n) \to \cE^{\{M\}}(\{0\}) = \La^{\{M\}}$ has no continuous linear right-inverse
if $M$ is any Gevrey sequence $M_k = G^s_k = k!^s$, $s>1$.

Actually, $j^\infty_{\{0\}} : \cE^{\{M\}}([-1,1]^n) \to \La^{\{M\}}$ is not split surjective
for any strongly regular weight sequence $M$ as seen by the following topological argument (cf.\ \cite[p.23]{MR964961}):
$\cE^{\{M\}}([-1,1]^n)$ is isomorphic to the dual of a power series space of finite type (cf.\ \cite{Langenbruch89}),
which cannot have $\La^{\{M\}}$ as a complemented subspace, since $\La^{\{M\}}$ is isomorphic to the
dual of a power series space of infinite type (cf.\ \cite[p.269]{Vogt82}).

For the Whitney problem, Langenbruch \cite{MR964961} proved that
for compact convex subsets $K,L$ of $\R^n$ with $\emptyset \ne \on{int} K \subseteq K \subseteq \on{int} L$
there is no extension operator $E : \cE^{\{M\}}(K) \to \cE^{\{M\}}(L)$ if
$M$ is a weight sequence satisfying
\begin{gather}
	\E a >0 \A b \in \N_{\ge a} :
	\ol \la(b):= \limsup_{k \to \infty} \frac{\mu_{bk}}{\mu_k} \le a  \liminf_{k \to \infty} \frac{\mu_{bk}}{\mu_k}
	=: a\, \ul \la(b), \label{eq:Lreg1}
	\\
	 \lim_{b \to \infty} \frac{\ul \la(b)}{b} = \infty. \label{eq:Lreg2}
\end{gather}
A weight sequence with these properties is strongly regular (cf.\ \cite[pp.355--356]{MR964961}
and \cite[Lemma 1.1]{Langenbruch89}). It is easy to check that the Gevrey sequences $G^s$, $s>1$, have the properties
\eqref{eq:Lreg1} and \eqref{eq:Lreg2} (indeed, in that case $\ol \la(b) = \ul \la(b) = b^s$).

\subsection{A different approach based on polynomial approximation} \label{sec:polyapprox}
Denote by $\sP_d$ the space of all polynomials on $\R^n$ of degree at most $d$, and let $\sP_{-1} := \{0\}$.
For compact $K \subseteq \R^n$ and
continuous $f : K \to \R$ we set
\[
	\on{dist}_K(f,\sP_d) := \inf\{\|f - p\|_K : p \in \sP_d\}.
\]
Let $M=(M_k)$ be a weight sequence.
Let $\cA^{\{M\}}(K)$ be the set of all $f \in C^0(K)$ such that there exists $\rh>0$ such that
\begin{equation} \label{eq:seminormapprox}
	\|f\|^{\cA^M}_{K,\rh}:= \sup_{k \in \N} \sup_{d \ge -1} \frac{(d+2)^k  \on{dist}_K(f,\sP_d)}{\rh^k M_k} < \infty
\end{equation}
and $\cA^{(M)}(K)$ the set of all $f \in C^0(K)$ such that \eqref{eq:seminormapprox} holds for all $\rh>0$.
The spaces $\cA^{[M]}(K)$ are endowed with their natural locally convex topologies.

The following theorem is an ultradifferentiable version of Jackson's theorem.

\begin{theorem}[{\cite[Theorem 2.7]{Plesniak:1994ti}}]
	Let $M$ be a derivation-closed weight sequence with $m_k^{1/k} \to \infty$.
	For any compact $K \subseteq \R^n$ we have $j^0_K \cE^{[M]}(\R^n) \subseteq \cA^{[M]}(K)$, where $j^0_K(f) = f|_K$.
	The inclusion is continuous.
\end{theorem}

Let $r \ge 1$.
A compact set $K \subseteq \R^n$ is said to have the \emph{Markov property} $(P_r)$ if
there exist $C_1,C_2>0$ such that
for every polynomial $p$\footnote{It is not hard to see that $(P_r)$ holds if and only if there is $C>0$ such that
$\|p^{(\al)}\|_K \le C \deg(p)^{r|\al|} \|p\|_K$ for all $|\al|\ge 1$ and all polynomials $p$;
cf.\ \cite[Theorem 3.3]{Plesniak:1990aa}.
Any compact subset of $\R^n$ that is uniformly polynomially cuspidal has the Markov property,
in particular, any fat subanalytic compact subset of $\R^n$.
}
\begin{equation} \label{eq:Markov}
	|p(z)| \le C_1 \|p\|_K \quad \text{ for } z \in \C^n \text{ with } \on{dist}(z,K) \le \frac{C_2}{\deg(p)^r}.
\end{equation}
If $K\subseteq \R^n$ satisfies $(P_r)$, then necessarily $r\ge 2$, by \cite{Baran:1995aa},
and $K$ is \emph{$C^\infty$-determining}, see \cite[Remark 3.5]{Plesniak:1990aa},
that means $j^0_K f =0$ implies $j^\infty_K f =0$ for any $C^\infty$-function $f$.

The following extension theorem involves an unavoidable loss of regularity: for instance
\cite{Franken:1995ta} shows that $j^0_{[-1,1]} : \cE^{(G^s)}(\R) \to \cA^{(G^s)}([-1,1])$ is not surjective.
On the other hand the weight sequence is not required to be strongly non-quasianalytic.

\begin{theorem}[{\cite{Plesniak:1994ti}}]
	Let $M$ be a non-quasianalytic weight sequence of moderate growth.
	Let $K \subseteq \R^n$ be a compact set satisfying $(P_r)$ with some $r\ge 2$.
	Then there exists a continuous linear operator
	\(L : \cA^{[M]}(K) \to \cE^{[M^{r+1}]}(\R^n)\)
	such that $j^0_K \o L = \on{id}$. Here $M^{r+1} := (M^{r+1}_k)$ is the $(r+1)$-st component-wise power of $M$.
\end{theorem}

\begin{proof}
	By \Cref{prop:cutoff}, there is a sequence $(\vh_k)_{k\ge 1}$ of functions $\vh_k \in C^\infty(\R^n)$
	such that $0 \le \vh_k \le 1$, $\on{supp} \vh_k \subseteq U_k := \{x \in \R^n : d_K(x)< C_2 k^{-r}\}$
  (where $C_2$ is the constant from \eqref{eq:Markov}),
	$\vh_k=1$ in a neighborhood of $K$, and
	\[
		|\vh_k^{(\al)}| \le C(n)^{|\al|} k^{r|\al|} M_{|\al|}, \quad \al \in \N^n.
	\]
	For $f \in \cA^{[M]}(K)$ let $L_d f$ be the Lagrange interpolation polynomial of $f$ with
	nodes in Fekete--Leja extremal points of $K$ of order $d$; cf.\ \cite[Section 2]{Plesniak:1990aa}.
	Then
	\[
		\|f- L_df\|_K \le 2 (d+1)^n \on{dist}_K(f,\sP_d).
	\]
	Set
	\[
		Lf(x) := \vh_1(x)L_1f(x) + \sum_{d=1}^\infty \vh_d(x) (L_{d+1}f(x)- L_df(x))
	\]
	and let us check that $L$ is well-defined and has the desired properties.
	By \eqref{eq:Markov},
	\begin{align*}
		|\p^\al &(\vh_d(x) (L_{d+1}f(x)- L_df(x)))|
		\\
		&\le
		\sum_{\be \le \al} \binom{\al}{\be} |\vh_d^{(\be)}(x)| |(L_{d+1}f- L_df)^{(\al-\be)}(x)|
		\\
		&\le
		C_1 \sum_{\be \le \al} \binom{\al}{\be} \|\vh_d^{(\be)}\|_{U_d} \|(L_{d+1}f- L_df)^{(\al-\be)}\|_K
		\\
		&\le
		C C_1 \sum_{\be \le \al} \binom{\al}{\be} C(n)^{|\be|} d^{r|\be|} M_{|\be|} \cdot  (d+2)^{n+r|\al-\be|} \on{dist}_K(f,\sP_d)
		\\
		&\le D \si^{|\al|} (d+2)^{n+r|\al|} M_{|\al|} \on{dist}_K(f,\sP_d),
	\end{align*}
  for suitable constants $D,\si>0$.
	Thus, by \eqref{eq:seminormapprox} and moderate growth of $M$,
		\begin{align*}
		\|\p^\al (\vh_d (L_{d+1}f- L_df))\|_{\R^n}
		&\le D \si^{|\al|} (d+2)^{n+r|\al|} M_{|\al|}\cdot  \|f\|^{\cA^M}_{K,\rh}  \frac{\rh^{n+2+r|\al|} M_{n+2+r|\al|}}{(d+2)^{n+2+r|\al|}}
		\\
		&\le D_1 \si_1^{|\al|}  M^{r+1}_{|\al|} \|f\|^{\cA^M}_{K,\rh} \cdot \frac{1}{(d+2)^{2}}.
	\end{align*}
	It follows that $Lf$ is a well-defined $C^\infty$-function satisfying
	\[
		\|(Lf)^{(\al)}\|^{M^{r+1}}_{\R^n, \si_1} \le D_2 \, \|f\|^{\cA^M}_{K,\rh},
	\]
	which implies the assertion; notice for the Beurling case that $\si_1 \to 0$ as $\rh \to 0$.
\end{proof}

\section{Extension with controlled loss of regularity}

If the weight sequence $M=(M_k)$ is not strongly non-quasianalytic, and thus extension preserving the class is not possible,
it is natural to ask if the loss of regularity occurring in the extension can at least be controlled.
This is indeed the case.
It requires a new extension technique which is inspired by Dyn'kin's theory of \emph{almost analytic functions};
cf.\ \cite{Dynkin80,MR1253229} and also \cite{FurdosNenningRainer}.
For the singleton see \Cref{thm:Borelloss}.

\subsection{The setup: admissible pairs of weight sequences} \label{sec:setup}
Let $M,N$ be strongly log-convex weight sequences, $M$ of moderate growth with $m_k^{1/k}\to \infty$, $N$ non-quasianalytic with
$M \preceq N$, and
\begin{equation} \label{eq:msnq}
   \sup_{k} \frac{\mu_k}{k} \sum_{j \ge k} \frac{1}{\nu_j} < \infty,
\end{equation}
where $\mu_k = \tfrac{M_k}{M_{k-1}}$ and $\nu_k = \tfrac{N_k}{N_{k-1}}$.
In that case we call $(M,N)$ an \emph{admissible pair}.
Note that $(M,M)$ is an admissible pair if and only if $M$ is strongly regular.
Recall that the conditions \eqref{eq:SVcond} and \eqref{eq:msnq} are equivalent
if $M$ has moderate growth; so our definition is compatible with the case treated in \Cref{thm:Borelloss}.

\begin{remark}
  Let $N$ be a strongly log-convex non-quasianalytic weight sequence of moderate growth
  and let $S$ be the descendant of $N$.
  Then $(S,N)$ is an admissible pair
  and, if $(M,N)$ is another admissible pair, then $M \preceq S$; see \Cref{lem:log-convex}.
  Note that, if $M$ has moderate growth, then $M \preceq N$ if and only if $\mu \lesssim \nu$,
  indeed,
  $\mu_k \lesssim M_k^{1/k} \lesssim N_k^{1/k} \le \nu_k$.
\end{remark}

\subsection{Suitable cutoff functions}

Let $(M,N)$ be an admissible pair of weight sequences.
In analogy to the construction of optimal cutoff functions in \Cref{sec:optcutoffRWS}
we obtain cutoff functions with bounds reflecting the weaker assumption \eqref{eq:msnq}.

\begin{proposition}[{\cite[Proposition 4]{ChaumatChollet94}}]
  Let $(M,N)$ be an admissible pair of weight sequences.
  There is a constant $C>0$ such that for all $\rh,r,\ep>0$ there exists $\vh \in \cE^{\{N\}}(\R)$
  with $\vh|_{[-r,r]} = 1$, $\on{supp} \vh \subseteq [-(1+\ep)r,(1+\ep)r]$, and
  \[
    \|\vh\|^N_{\R,\rh} \le \frac{1}{h_m(C \rh r \ep)}.
  \]
\end{proposition}

Then it is straightforward to construct a partition of unity subordinate to a given family of Whitney cubes.

\begin{proposition}[{\cite[Proposition 6]{ChaumatChollet94}}] \label{prop:partitionmixed}
Let $(M,N)$ be an admissible pair of weight sequences.
Let $K \subseteq \R^n$ be a non-empty compact set and $(Q_j)_{j \ge 1}$ a family of Whitney cubes for $K$.
Then there exists
$C_1\ge 1$ such that for all $\ep>0$ there is a family of $C^\infty$-functions
$(\vh_{j,\ep})_{j \ge 1}$ satisfying
\begin{enumerate}
\item $0\le\vh_{j,\ep}\le 1$ for all $j \ge 1$,
\item $\on{supp}\vh_{j,\ep}\subseteq Q^*_j$ for all $j\ge 1$,
\item $\sum_{j\ge 1}\vh_{j,\ep}(x)=1$ for all $x \in \R^n\setminus K$,
\item for all $j\ge 1$, $\be \in\N^n$, and $x\in\R^n\setminus K$,
\begin{equation*}
  |\vh^{(\be)}_{j,\ep}(x)| \le \frac{\ep^{|\be|} N_{|\be|}}{h_m(C_1\ep\, d_K(x))}.
\end{equation*}
\end{enumerate}
\end{proposition}

\subsection{Extension of Whitney ultrajets} \label{sec:mixedCC}

\begin{theorem}[\cite{ChaumatChollet94}] \label{thm:CCmixed}
  Let $M,N$ be strongly log-convex weight sequences, $M$ of moderate growth with $m_k^{1/k}\to \infty$, and $N$ non-quasianalytic with
  $M \preceq N$. Then $(M,N)$ is admissible if and only if
  $\cE^{\{M\}}(A) \subseteq j^\infty_A\cE^{\{N\}}(\R^n)$ for all closed $A \subseteq \R^n$.
\end{theorem}

We will sketch the proof of the ``only if'' part.
Assume that $(M,N)$ is admissible.
The new feature is the shape of the extension formula:
Let $K \subseteq \R^n$ be compact
and $F=(F^\al)_\al \in \cE^{\{M\}}(K)$ satisfy \eqref{jets1} and \eqref{jets2}.
Let $(Q_j)_{j\ge 1}$ be a family of Whitney cubes for $K$.
Let $\ep,L>0$ be fixed.
There is a family of $C^\infty$-functions $(\vh_{j,\ep})_{j\ge 1}$ satisfying
the conclusion of \Cref{prop:partitionmixed}.
Let $x_j$ be the center of $Q_j$ and let
$\hat x_j \in K$ be such that $|x_j - \hat x_j|= d_K(x_j)$. We define
\begin{equation} \label{eq:mixedextensionformula}
f(z) :=
\begin{cases}
  \sum_{j \ge 1} \vh_{j,\ep}(z) \, T^{2 \Ga_m(L d_K(x_j))}_{\hat x_j}F(z),  & \text{ if } z \in \R^n \setminus K, \\
  F^0(z), & \text{ if } z \in  K.
\end{cases}
\end{equation}
(So the local extensions $f_{\hat x_j}$ in \eqref{eq:extensionformula} are replaced by the Taylor polynomials $T^{2 \Ga_m(Ld_K(x_j))}_{\hat x_j}F$.)
The degree $p(x) := 2 \Ga_m(Ld_K(x))$ of the Taylor approximation tends to infinity as $x$ approaches the set $K$;
the function $\Ga_m$ was introduce in \Cref{sec:associatedfunctions}.

The proof that \eqref{eq:mixedextensionformula} is indeed the required extension
is similar to the proof of \Cref{thm:WRoumieu}.
First we need estimates for the partial derivatives of the Taylor polynomials $T_{\hat x}^{p(x)} F$.

\begin{lemma}
  There is a constant $A_0=A_0(M,N)>1$ such that
  for all $F \in \cE^{\{M\}}(K)$ satisfying \eqref{jets1} and \eqref{jets2} (with the constants $C,\rh$),
  all $L \ge A_0 \rh$, $x \in \R^n$, and $\al \in \N^n$,
  \begin{align}
  |(T_{\hat x}^{p(x)} F)^{(\al)}(x)| &\le C (2L)^{|\al|+1} M_{|\al|},  \label{prop91}
  \intertext{and, if $|\al| < p(x)$,}
  |(T_{\hat x}^{p(x)}F)^{(\al)}(x)-F^\al(\hat x)| &\le C (2L)^{|\al|+1} |\al|!\, m_{|\al|+1} d_K(x).
  \label{prop92}
\end{align}
\end{lemma}

\begin{proof}
For \eqref{prop91} we may restrict to the case $|\al| \le p(x)$. By \eqref{jets1},
\begin{align}
  |(T_{\hat x}^{p(x)}F)^{(\al)}(x)|
  &\le\sum_{\substack{\al \le \be\\ |\be|\le p(x)}}
  \frac{|x-\hat{x}|^{|\be|-|\al|}}{(\be-\al)!} C \rh^{|\be|} M_{|\be|}
  \notag \\
  &\le C |\al|! \sum_{\substack{\al \le \be\\ |\be|\le p(x)}}
  \frac{|\be|!\,(n d_K(x))^{|\be|-|\al|}}{|\al|!\, (|\be|-|\al|)!}  \rh^{|\be|} m_{|\be|}
  \notag \\
  &\le  \frac{C |\al|! }{(n d_K(x))^{|\al|}} \sum_{\substack{\al \le \be\\ |\be|\le p(x)}}
  (2n \rh  d_K(x))^{|\be|} m_{|\be|}
  \notag \\
  &\le \frac{C|\al|! }{(n d_K(x))^{|\al|}} \sum_{j = |\al|}^{p(x)}
  (2n^2 \rh  d_K(x))^{j} m_{j},  \label{calculation}
\end{align}
since the number of $\be \in \N^n$ with $|\be|=j$ is bounded by $n^j$.
By Lemma \ref{claim1},
\begin{equation} \label{eq:twoGa}
    \text{there exists $\la<1$ such that $2 \Ga_{m}(t) \le \Ga_m(\la t)$ for all $t>0$},
\end{equation}
and thus
\begin{align*}
  |(T_{\hat x}^{p(x)}F)^{(\al)}(x)|
  &\le \frac{C|\al|! }{(n d_K(x))^{|\al|}} \sum_{j = |\al|}^{\Ga_m(L\la d_K(x))}
  (2n^2 \rh  d_K(x))^{j} m_{j}.
\end{align*}
Since  $(L \la d_K(x))^{j} m_{j}\le (L \la d_K(x))^{|\al|} m_{|\al|}$ for $|\al|\le j \le \Ga_m(L \la d_K(x))$, by \eqref{GaProp1},
\begin{align*}
  |(T_{\hat x}^{p(x)}F)^{(\al)}(x)|
  &\le C M_{|\al|} \Big(\frac{L \la}{n}\Big)^{|\al|} \sum_{j = |\al|}^{\Ga_m(L\la d_K(x))}
  \Big(\frac{2n^2 \rh}{L \la } \Big)^{j}.
\end{align*}
We obtain \eqref{prop91} if $L$ is chosen such that $\frac{2n^2\rh}{L\la} \le \frac{1}{2}$;
then $A_0= \frac{4n^2}{\la}$.

For \eqref{prop92} it suffices to note that, if $|\al| < p(x)$,
\[
  (T_{\hat x}^{p(x)}F)^{(\al)}(x)-F^\al(\hat x)
  = \sum_{\substack{\al \le \be\\ |\al| < |\be|\le p(x)}}\frac{(x-\hat{x})^{\be-\al}}{(\be-\al)!} F^\be(\hat x),
\]
and to follow the same arguments.
\end{proof}

The core of the proof is the following lemma.

\begin{lemma}
  There exist constants $A_i = A_i(M,N)$, $i=1,\ldots,4$,
  such that the following holds.
  If $\ep=A_1L$ and $L>A_2\rh$,
  then for all $x \in \R^n \setminus K$ with $d_K(x)<1$ and all $\al \in \N^n$ we have
  \begin{equation} \label{eqclaim2}
    |(f - T^{p(x)}_{\hat x} F)^{(\al)}(x)| \le C (LA_3)^{|\al|+1} N_{|\al|} h_m(LA_4 \, d_K(x)).
  \end{equation}
  Here $C,\rh$ are the constants from \eqref{jets1} and \eqref{jets2}.
\end{lemma}

\begin{proof}
  By the Leibniz rule we have
  \[
    (f - T^{p(x)}_{\hat x} F)^{(\al)}(x) = \sum_{\be \le \al} \binom{\al}{\be} \sum_{j \ge 1}
    \vh_{j,\ep}^{(\al-\be)}(x) \big(T^{p(x_j)}_{\hat x_j} F-T^{p(x)}_{\hat x} F\big)^{(\be)} (x).
  \]
  Let us estimate
  \begin{align*}
    \MoveEqLeft
    |\big(T^{p(x_j)}_{\hat x_j} F-T^{p(x)}_{\hat x} F\big)^{(\be)} (x)|
    \\
    &\le
    |\big(T^{p(x_j)}_{\hat x_j} F-T^{p(x_j)}_{\hat x} F\big)^{(\be)} (x)| + |\big(T^{p(x_j)}_{\hat x} F-T^{p(x)}_{\hat x} F\big)^{(\be)} (x)|
    =: H_1 + H_2
  \end{align*}
  for $x \in Q_j^*$ (since $\on{supp} \vh_{j,\ep} \subseteq Q_j^*$).
  One checks easily (as in the proof of \Cref{lem:prepW2}) that, with $2q := p(x_j)$,
  \[
    H_1 \le C (2n^2 \rh)^{2q +1} |\be|!\, m_{2q+1} (|\hat x_j - x| + |\hat x_j - \hat x|)^{2q+1-|\be|}.
  \]
  By \Cref{prop:Wcubes}, there exist universal constants $0<b\le 1 \le B$ such that
  \begin{equation} \label{eq:bBd}
    b \,d_K(x) \le d_K(x_j) \le B \, d_K(x), \quad \text{ for } x \in Q^*_j.
  \end{equation}
  Moreover, $|\hat x_j - x| + |\hat x_j - \hat x| \lesssim d_K(x_j)$. So, using that $m$ has moderate growth,
  \[
    H_1 \le C (D_1 \rh)^{2q +1} |\be|!\, m_{q}^2 (D_2\, d_K(x_j))^{2q+1-|\be|},
  \]
  for constants $D_i\ge 1$ which depend only on $M$, $N$, and the dimension $n$.
  Since $h_m(L d_K(x_j)) = m_q (L d_K(x_j))^q \le m_{|\be|} (L d_K(x_j))^{|\be|}$, by \eqref{GaProp2},
  we get
  \[
    H_1 \le C D_1 D_2 \rh  \Big(\frac{D_1D_2 \rh}{L}\Big)^{2q}  d_K(x_j)  |\be|!\, m_{|\be|} L^{|\be|} h_m(L d_K(x_j)).
  \]
  Since $d_K(x_j) \lesssim d_K(x) < 1$ (by assumption), there is $D_3$ such that for $L > D_3 \rh$,
  \[
    H_1 \le C L^{|\be|+1} M_{|\be|}  h_m(L d_K(x_j)).
  \]

  Now consider $H_2$.
  By \eqref{eq:twoGa} and \eqref{eq:bBd},
  we have
  \[
   p(x_j)= 2 \Ga_m(Ld_K(x_j)) \le
    2 \Ga_m(Lb d_K(x)) \le \Ga_m(Lb \la d_K(x))
  \]
  and similarly $p(x) \le \Ga_m(Lb \la d_K(x))$. So the polynomial $T^{p(x_j)}_{\hat x} F-T^{p(x)}_{\hat x} F$
  has degree  at most $\Ga_m(Lb \la d_K(x))$.
  On the other hand, again by \eqref{eq:bBd}, the valuation of the polynomial $T^{p(x_j)}_{\hat x} F-T^{p(x)}_{\hat x} F$ is at least
  $2 \Ga_m(L B d_K(x)) =: 2 r$.
  Thus, the calculation \eqref{calculation} gives
  \begin{align*}
     H_2 \le \frac{C |\be|!}{(n d_K(x))^{|\be|}} \sum_{j = 2r}^{\Ga_m(Lb \la d_K(x))} (2n^2 \rh d_K(x))^j m_j.
  \end{align*}
  For such $j$ we have $m_j (Lb\la d_K(x))^j \le m_{2r} (Lb\la d_K(x))^{2r}$, by \eqref{GaProp1},
  and $h_m(LB d_K(x)) = m_r (LB d_K(x))^r \le  m_{|\be|} (LB d_K(x))^{|\be|}$, by \eqref{GaProp2}.
  Using moderate growth of $m$, we may thus conclude that there is $D_4 \ge 1$ such that, for $L>D_4 \rh$,
  \begin{align*}
     H_2 &\le C L^{|\be|+1} M_{|\be|}  h_m(L B d_K(x)).
  \end{align*}

  In summary, for $x \in Q_j^*$ and $d_K(x) < 1$,
  \begin{align*}
       |\big(T^{p(x_j)}_{\hat x_j} F-T^{p(x)}_{\hat x} F\big)^{(\be)} (x)|
       \le C (2 L)^{|\be|+1} M_{|\be|} h_m(L B d_K(x))
  \end{align*}
  and, consequently, in view of \Cref{prop:partitionmixed},
  \begin{align*}
    \MoveEqLeft
     |(f - T^{p(x)}_{\hat x} F)^{(\al)}(x)|
     \\
     &\le
     \sum_{\be \le \al} \frac{\al!}{\be!(\al-\be)!}
    \cdot 12^{2n} \cdot
     \frac{\ep^{|\al|-|\be|} N_{|\al|-|\be|}}{h_{m}(C_1\ep\, d_K(x))}
    \cdot C    (2 L)^{|\be|+1}  M_{|\be|} \, h_{m}(L B d_K(x))
    \\
    &\le  C 12^{2n}
    \sum_{j=0}^{|\al|} \frac{|\al|!\, n^{|\al|+j}}{j!(|\al|-j)!}
     \ep^{|\al|-j} (2L)^{j+1}  N_{|\al|-j}
          M_{j} \, \frac{h_{m}(LB d_K(x))}{h_{m}(C_1\ep\, d_K(x))}.
  \end{align*}
  Now, by \eqref{hmg},
  \[
    \frac{h_{m}(LB d_K(x))}{h_{m}(C_1\ep\, d_K(x))} \le h_{m}(LBC d_K(x))
  \]
  if we choose $\ep :=\frac{LBC}{C_1}$, where $C$ is the constant from \eqref{hmg}.
  Noting that $M \preceq N$, the lemma follows.
\end{proof}

Now we are ready to check that
\eqref{eq:mixedextensionformula} is the desired extension of $F$.
That $f$ defines a $C^\infty$-function on $\R^n$ with $j^\infty_K f=F$ can be seen as follows.
Let us fix a point $a \in K$ and $\al\in \N^n$.
Since $\Ga_{m}(t) \to \infty$ as $t \to 0$, we have $|\al| < p(x) = 2 \Ga_{m}(L d_K(x))$ if $x \in \R^n \setminus K$ is
sufficiently close to $a$. Thus, as $x \to a$,
\begin{align*}
  &|f^{(\al)}(x) - F^{\al}(a)|
  \\
  &\le
  |(f - T_{\hat x}^{p(x)} F)^{(\al)} (x)| +
    |(T_{\hat x}^{p(x)}F)^{(\al)}(x)-F^\al(\hat x)| + |F^\al(\hat x) - F^\al(a)|
  \\
  & = O(h_{m}(L A_4 d_K(x))) + O(d_K(x)) + O(|\hat x - a|),
\end{align*}
by \eqref{jets2}, \eqref{prop92}, and \eqref{eqclaim2}.
Hence $f^{(\al)}(x) \to  F^{\al}(a)$ as $x \to a$.

To get an extension in $\cE^{\{N\}}(\R^n)$ we first observe that,
for $x \in \R^n \setminus K$ with $d_K(x)<1$, $\al \in \N^n$, and a suitable constant $A=A(n,M,N)$,
 \begin{align*}
    |f^{(\al)}(x)|
    &\le
   |(T_{\hat x}^{p(x)} F)^{(\al)}(x)| + |f^{(\al)}(x) - (T_{\hat x}^{p(x) } F)^{(\al)} (x)|
   \le
    C (L A)^{|\al|+1} N_{|\al|}
 \end{align*}
 by \eqref{prop91} and \eqref{eqclaim2},
 because $h_{m} \le 1$. In view of \eqref{jets1},
 it suffices to multiply $f$ with a
 suitable cutoff function of class $\cE^{\{N\}}$ with support in $\{x : d_K(x) < 1\}$
 in order to obtain an extension in $\cE^{\{N\}}(\R^n)$.

\begin{remark} \label{rem:extopR}
  The proof shows that for each $\rh>0$ there is an extension operator
  $\cE^M_\rh(K) \to \cE^N_{A\rh}(\R^n)$, for a suitable constant $A$, but the operator depends
  on $\rh$ (through $L$ and $\ep$).
\end{remark}

There is a corresponding result in the Beurling case.

\begin{theorem}[\cite{ChaumatChollet94}] \label{thm:mixedBeurlingWS}
  Let $M,N$ be strongly log-convex weight sequences, $M$ of moderate growth with $m_k^{1/k}\to \infty$,
  and $N$ non-quasianalytic with
  $M \preceq N$. Then $(M,N)$ is admissible if and only if
  $\cE^{(M)}(A) \subseteq j^\infty_A\cE^{(N)}(\R^n)$ for all closed $A \subseteq \R^n$.
\end{theorem}

Let $(M,N)$ be a pair of admissible weight sequences and let $L$ be a non-negative sequence such that $L \lhd M$.
Then there exists an admissible pair of weight sequences $(M',N')$ such that $L \preceq M'$ and $N' \lhd N$.
Indeed, as in the proof of \Cref{lem:CC17} there is a decreasing sequence $(\ep_k)_{k\ge 1}$ tending to $0$
with $L_k \le \ep_1\cdots \ep_k M_k$ for all $k$.
Applying \Cref{lem:CC16} to $\al_k = 0$, $\be_k = \max \{\ep_k,\frac{k}{\mu_k}\}$, and $\ga_k = \frac{k}{\mu_k}$,
gives a sequence $(\th_k)_{k \ge 1}$ such that $\th_k \nearrow \infty$, $\frac{k\th_k}{\mu_k} \searrow 0$ and $\th_k \ep_k \to 0$.
Applying \Cref{lem:CC16} once more, now to $\al_k = \frac{1}{\nu_k}$, $\be_k = \max \{\frac{1}{\th_k},\frac{k}{\nu_k}\}$, and $\ga_k = \frac{k}{\nu_k}$,
yields a sequence $(\vt_k)_{k \ge 1}$
with $\vt_k \nearrow \infty$
such that
\begin{align*}
  \sum_{j \ge k} \frac{\vt_j}{\nu_j} \le 8 \vt_k \sum_{j \ge k} \frac{1}{\nu_j}, \quad \text{ for all } k \ge 1,
\end{align*}
$\frac{\vt_k}{\th_k} \to 0$, and  $\frac{k\vt_k}{\nu_k} \searrow 0$.
Define
$\mu'_k := \frac{\mu_k}{\th_k}$ and $\nu'_k := \frac{\nu_k}{\vt_k}$ for $k \ge 1$ and $\mu'_0 = \nu'_0 :=1$.
Then $\frac{\mu'_k}{k}\nearrow \infty$ and $\frac{\nu'_k}{k}\nearrow \infty$ and hence
$M'_k := \mu'_0 \mu'_1 \cdots \mu'_k$ and $N'_k := \nu'_0 \nu'_1 \cdots \nu'_k$ are strongly log-convex weight sequences
such that $(m'_k)^{1/k} \to \infty$ and $(n'_k)^{1/k} \to \infty$.
That $M'$ has moderate growth follows from
\[
  \frac{M'_{j+k}}{M'_jM'_k} = \frac{M_{j+k}}{M_jM_k} \frac{\th_1\cdots\th_{j}  \th_1\cdots\th_{k}}{\th_1\cdots\th_{j+k}}
  \le \frac{M_{j+k}}{M_jM_k}.
\]
Furthermore,
\begin{align*}
  \sum_{j \ge k} \frac{1}{\nu'} =  \sum_{j \ge k} \frac{\vt_j}{\nu_j} \le 8 \vt_k \sum_{j \ge k} \frac{1}{\nu_j}
  \lesssim  \frac{k\vt_k}{\mu_k}
  \lesssim  \frac{k\th_k}{\mu_k} = \frac{k}{\mu_k'}
\end{align*}
and hence $(M',N')$ is an admissible pair.
We have $L \preceq M'$ and $N' \lhd N$, since
$(\frac{L_k}{M'_k})^{1/k} = (\frac{L_k}{M_k})^{1/k} (\th_1\cdots\th_{k})^{1/k} \le (\th_1\ep_1 \cdots\th_{k}\ep_k)^{1/k}$ is bounded
and $(\frac{N_k}{N'_k})^{1/k} = (\vt_1 \cdots \vt_k)^{1/k} \to \infty$.

Then it is easy to conclude $\cE^{(M)}(A) \subseteq j^\infty_A\cE^{(N)}(\R^n)$ for all closed $A \subseteq \R^n$,
by a reduction argument to the Roumieu case (similar to the proof of \Cref{thm:WBeurling}).

\subsection{Extension operator}

Concerning the existence of extension operators in the mixed setting we have

\begin{theorem}[\cite{ChaumatChollet94}]
  Let $M,N$ be strongly log-convex weight sequences, $M$ of moderate growth with $m_k^{1/k}\to \infty$,
  and $N$ non-quasianalytic with $M \lhd N$.
  Then the following conditions are equivalent:
  \begin{enumerate}
     \item $\lim_{k\to \infty} \frac{\mu_k}{k} \sum_{j \ge k} \frac{1}{\nu_j}=0$.
     \item $\La^{\{M\}} \subseteq j^\infty_{\{0\}}\cE^{(N)}(\R^n)$.
     \item There exists an extension operator $\cE^{\{M\}}(A) \to \cE^{(N)}(\R^n)$ for each closed non-empty subset $A \subseteq \R^n$.
   \end{enumerate}
\end{theorem}

We briefly comment on the implication (1) $\Rightarrow$ (3).
The idea is to show that (1) implies the existence of
an admissible pair $(M',N')$ of weight sequences such that
$M \lhd M' \lhd N' \lhd N$. This can be achieved again with the help of \Cref{lem:CC16}; for the details we refer to
\cite[Proposition 20]{ChaumatChollet94}.
Then the extension operator is given by the composite
\[
  \xymatrix{
  \cE^{\{M\}}(K) ~ \ar@{^{(}->}[r] & \cE^{M'}_1(K) \ar[r] & \cE^{N'}_{A}(\R^n) ~ \ar@{^{(}->}[r] & \cE^{(N)}(\R^n)
  }
\]
where the middle arrow is the extension operator from \Cref{rem:extopR}.

\begin{remark}
  All extensions in this section can be chosen to be analytic in the complement of $A$
  in view of \Cref{thm:analyticmodification}.
\end{remark}

\part{More general ultradifferentiable classes}

\section{Further ultradifferentiable functions}

We introduce weight functions and Braun--Meise--Taylor classes.
This quickly leads to the notion of the associated weight matrix
which allows us to build the theory of Braun--Meise--Taylor classes (and of more general ultradifferentiable classes) to a great extent upon the theory of
Denjoy--Carleman classes.

\subsection{Weight functions} \label{sec:weightfunctions}

\begin{definition}
A \emph{weight function} is a continuous increasing function
$\om: [0,\infty) \to [0,\infty)$ with $\om(0)=0$ satisfying the following three properties:
\begin{enumerate}
	\item $\om(2t)=O(\om(t))$ as $t\to \infty$.
	\item $\log(t)=o(\omega(t))$ as $t\to \infty$.
	\item $\ph : t \mapsto \om(e^t)$ is convex on $[0,\infty)$.
\end{enumerate}
\end{definition}

Note that (2) is equivalent to $\lim_{t \to \infty} \frac{t}{\ph(t)} = 0$ and implies $\lim_{t \to \infty} \om(t) = \infty$.
For each weight function $\om$ there is a weight function $\tilde \om$ which coincides with $\om$ for all sufficiently large $t>0$ and such that $\tilde \om|_{[0,1]}=0$.
It is hence no loss of generality to assume that $\om|_{[0,1]}=0$ which we shall do
tacitly if convenient.

Since $\ph$ is convex, increasing, and $\ph(0)=0$, we may consider the \emph{Young conjugate}
$\ph^* : [0,\infty) \to [0,\infty)$
of $\ph$,
given by
\[
\ph^*(t):=\sup_{s \ge 0} (st-\ph(s)), \quad t \ge 0.
\]
The function $\ph^*$ is convex, increasing, and
it satisfies $\ph^*(0)=0$, $\ph^{**}=\ph$ (here we use $\om|_{[0,1]}=0$),
and $\lim_{t\to \infty}\frac{t}{\ph^*(t)}=0$.
Moreover,
the functions $t \mapsto \tfrac{\ph(t)}{t}$ and $t \mapsto \tfrac{\ph^*(t)}{t}$ are increasing on $(0,\infty)$.

\subsection{Braun--Meise--Taylor classes}

Let $\om$ be a weight function.
Let $U$ be an open subset of $\mathbb{R}^n$ and $\rho >0$.
For $f \in C^\infty(U)$ and compact $K \subseteq U$ we consider
the seminorm
\[
	\|f\|^\om_{K,\rho}
	:= \sup_{x \in K} \sup_{\alpha \in \mathbb{N}^n} |f^{(\alpha)}(x)| \exp(- \tfrac{1}{\rh} \ph^*(\rh |\al|)).
\]
(Occasionally, we will also use $\|f\|^\om_{U,\rho}$ for open sets $U$.)
We define the \emph{Braun--Meise--Taylor classes}
\begin{align*}
	\mathcal{E}^{(\om)}(U)
	&:= \big\{f \in C^\infty(U) : \forall K \subseteq_{cp} U ~\forall \rh > 0 : \|f\|^\om_{K,\rho} < \infty\big\},
	\\
	\mathcal{E}^{\{\om\}}(U)
	&:= \big\{f \in C^\infty(U) : \forall K \subseteq_{cp} U ~\exists \rh > 0 : \|f\|^\om_{K,\rho} < \infty\big\}.
\end{align*}
The class $\mathcal{E}^{(\om)}(U)$ is said to be of \emph{Beurling type}
and the class $\mathcal{E}^{\{\om\}}(U)$ of \emph{Roumieu type}.
In analogy to Denjoy--Carleman classes we use $\mathcal{E}^{[\om]}(U)$ as placeholder for either $\mathcal{E}^{(\om)}(U)$
or $\mathcal{E}^{\{\om\}}(U)$; cf.\ \Cref{convention}.

We endow the vector spaces $\mathcal{E}^{[\om]}(U)$ with their natural locally convex topologies:
$\mathcal{E}^{(\om)}(U)$ is supplied with its natural Fr\'echet space topology and
$\mathcal{E}^{\{\om\}}(U)$ with
the projective limit topology over $K$ of the
inductive limit topology over $\rh$; note that it suffices to take countable limits.
Then $\mathcal{E}^{(\om)}(U)$ is a nuclear Fr\'echet space and
$\mathcal{E}^{\{\om\}}(U)$ is nuclear, complete, and reflexive; see
\cite[Proposition 4.9]{BMT90}.

\begin{remark}
  The classes $\mathcal{E}^{[\om]}$ originate from work of Beurling \cite{Beurling61} (see also Bj\"orck \cite{Bjoerck66}).
  Beurling described his non-quasianalytic classes by decay properties of the Fourier transform of a compactly supported function:
  \begin{align*}
   	\mathcal{D}^{(\om)}(U)
   	&= \Big\{f \in C_c(U) : ~\forall \rh>0 : \int_{\R^n} | \widehat f(\xi)| e^{\rh \om(\xi)} \, dt <\infty \Big\},
   	\\
   	\mathcal{D}^{\{\om\}}(U)
   	&= \Big\{f \in C_c(U) : ~\exists \rh>0 : \int_{\R^n} | \widehat f(\xi)| e^{\rh \om(\xi)} \, dt <\infty \Big\}.
   \end{align*}
  Here, and below, $\om(\xi) := \om(|\xi|)$.
  Braun, Meise, and Taylor \cite{BMT90} showed that these classes are non-trivial (for each non-empty open set $U \subseteq \RR^n$) provided
  that $\om$ is a weight function (not necessarily satisfying \ref{sec:weightfunctions}(3))\footnote{Beurling and Bj\"orck had additionally assumed that $\om$ is subadditive.}
  such that
  \begin{equation} \label{eq:nqOmega}
  		\int_{1}^\infty \frac{\om(t)}{t^2} \, dt <\infty.
  \end{equation}
  We will give an independent proof of this fact in
  \Cref{thm:(non)qaWF}.
  Furthermore, if also the condition \ref{sec:weightfunctions}(3) is fulfilled, then
  the classes $\mathcal{D}^{[\om]}$ also admit the description
  \begin{align*}
   	\mathcal{D}^{(\om)}(U)
   	&= \big\{f \in \mathcal{D}(U) : ~\forall \rh>0 : \|f\|^\om_{U,\rho} <\infty \big\},
   	\\
   	\mathcal{D}^{\{\om\}}(U)
   	&= \big\{f \in \mathcal{D}(U) : ~\exists \rh>0 : \|f\|^\om_{U,\rho} <\infty  \big\}.
   \end{align*}
   This was the starting point for the introduction of the local classes $\mathcal{E}^{[\om]}(U)$
   and the development of a corresponding theory of ultradistributions
   in \cite{BMT90}.
\end{remark}

\subsection{Associated weight matrix}

We associate with a weight function $\om$ a one-parameter family $\mathfrak{W} = \{W^x\}_{x>0}$
of weight sequences $W^x = (W^x_k)$ defined by
\[
	W^x_k := \exp(\tfrac{1}{x} \ph^*(x k)), \quad k \in \NN.
\]
The family $\mathfrak{W}$ is called the \emph{associated weight matrix}\index{weight matrix!associated} of $\om$.
As usual for weight sequences, we write
\[
  W^x_k = \vt^x_0 \vt^x_1 \vt^x_2 \cdots \vt^x_k =  k!\, w^x_k.
\]
Notice that, for any $\rh>0$,
\begin{equation} \label{eq:connectionSN}
	\|f\|^\om_{K,\rho}
	= \sup_{\alpha \in \mathbb{N}^n} \frac{\|f^{(\alpha)}\|_K}{\exp(\tfrac{1}{\rh} \ph^*(\rh |\al|))}
	= \sup_{\alpha \in \mathbb{N}^n} \frac{\|f^{(\alpha)}\|_K}{W^\rh_{|\al|}}
	= \|f\|^{W^\rh}_{K,1}.
\end{equation}

\begin{lemma} \label{lem:propertiesWF}
  We have:
	\begin{enumerate}
		\item Each $W^x$ is a weight sequence.
		\item $\vartheta^x \le \vartheta^y$ if $x \le y$, which entails $W^x \le W^y$.
		\item For all $x>0$ and all $j,k \in \NN$,  $W^x_{j+k} \le W^{2x}_{j} W^{2x}_{k}$ and $w^x_{j+k} \le w^{2x}_{j} w^{2x}_{k}$.
		\item For all $x>0$ and all $k \in \NN_{\ge 2}$, $\vartheta^x_{2k} \le  \vartheta^{4x}_{k}$.
    \item $\forall \rh>0 ~\exists H\ge 1 ~\forall x >0 ~\exists C \ge 1 ~\forall k \in \NN : \rh^k W^x_k \le C W^{Hx}_k$.
	\end{enumerate}
\end{lemma}

\begin{proof}
	(1)
  The convexity of $\ph^*$ implies that each $\vt^x$ is increasing.
  And
	$(W^x_k)^{1/k} \nearrow \infty$
  follows from
  $t \mapsto \frac{\ph^*(t)}{t} \nearrow \infty$

	(2)  Let $0<x \le y$. Then
	\[
		\vartheta^x_k  = \exp\Big(\frac{\ph^*(x k) - \ph^*(x (k-1))}{x} \Big)
		\le \exp\Big(\frac{\ph^*(y k) - \ph^*(y (k-1))}{y} \Big) = \vartheta^y_k
	\]
	again by the convexity of $\ph^*$.

	(3) The convexity of $\ph^*$ and $\ph^*(0)=0$ implies
	\begin{equation*}
	\ph^*(t)+\ph^*(s) \le \ph^*(t+s) \le \tfrac{1}{2}\ph^*(2t)+ \tfrac{1}{2}\ph^*(2s) , \quad t,s\ge 0.
	\end{equation*}
	As a consequence
	\[
		W^x_j W^x_k \le W^x_{j+k} \le W^{2x}_j W^{2x}_k, \quad j,k \in \NN,
	\]
	and equivalently
	\[
		w^x_j w^x_k \le \binom{j+k}{j} w^x_{j+k} \le w^{2x}_j w^{2x}_k, \quad j,k \in \NN,
	\]

	(4) The inequality $\vartheta^x_{2k} \le   \vartheta^{4x}_{k}$ is equivalent to
  	\begin{equation*}
    	\frac{\ph^*(2k x) - \ph^*((2k-1)x) }{x}  \le \frac{\ph^*(4k x) - \ph^*(4(k-1)x)}{4x},
  	\end{equation*}
  	which follows from the convexity of $\ph^*$ if $k \ge 2$.

  	(5) The statement follows from the following inequality
  \begin{equation*}
    \exists L\ge 1 ~\forall t \ge 0 ~\forall s \in \NN :
     L^s \ph^*(t)+ s L^s t \le  \ph^*(L^s t) + \sum_{i=1}^s L^i,
  \end{equation*}
  if we choose $s$ such that $e^s\ge \si$ and set $t:=\rh k$,
  $H:=L^s$ and $C := \exp(\tfrac{1}{H \rh} \sum_{i=1}^s L^i)$.
  By \ref{sec:weightfunctions}(1), there exists $L_1\ge1$ such that $\om(2t) \le L_1 \om(t)+L_1$ for all $t\ge0$ and hence
  there is $L\ge1$ such that $\ph(t+1) \le L \ph(t)+L$ for all $t\ge0$.
  Thus, for $t \ge 0$,
  \begin{align*}
    \ph^*(Lt) + L = \sup_{s\ge 0} (Lts - (\ph(s)-L)) \ge \sup_{s\ge 1} (Lts - L \ph(s-1)) = L \ph^*(t)+Lt,
  \end{align*}
  and the desired inequality follows by iteration.
\end{proof}

\subsection{Weight matrix description of Braun--Meise--Taylor classes}

\begin{theorem}[\cite{RainerSchindl12}] \label{thm:WMdescription}
	Let $\om$ be a weight function with associated weight matrix $\mathfrak{W} = \{W^x\}_{x>0}$.
	Let $U \subseteq \RR^n$ be open and $K \subseteq U$ compact.
	\begin{enumerate}
		\item For each $x>0$ we have the continuous inclusions
		$\mathcal{E}^{\{W^x\}}(U) \subseteq \mathcal{E}^{\{\om\}}(U)$
		and $\mathcal{E}^{(\om)}(U) \subseteq \mathcal{E}^{(W^x)}(U)$.
		\item As locally convex spaces
		\begin{align*}
			\mathcal{E}^{(\om)}(K) &= \on{proj}_{x>0} \mathcal{E}^{(W^x)}(K)
      = \on{proj}_{x>0} \on{proj}_{\rh>0} \mathcal{E}^{W^x}_\rh(K),
			\\
			\mathcal{E}^{\{\om\}}(K) &=  \on{ind}_{x>0} \mathcal{E}^{\{W^x\}}(K)
      = \on{ind}_{x>0} \on{ind}_{\rh>0} \mathcal{E}^{W^x}_\rh(K),
			\\
			\mathcal{E}^{(\om)}(U) &= \on{proj}_{K \subseteq_{cp} U} \on{proj}_{x>0} \mathcal{E}^{(W^x)}(K) =  \on{proj}_{x>0} \mathcal{E}^{(W^x)}(U),
			\\
			\mathcal{E}^{\{\om\}}(U) &= \on{proj}_{K \subseteq_{cp} U} \on{ind}_{x>0} \mathcal{E}^{\{W^x\}}(K).
		\end{align*}
	\end{enumerate}
\end{theorem}

\begin{proof}
  (1)	Let $x>0$ be fixed.
  If $f \in \mathcal{E}^{\{W^x\}}(U)$,
  then for each compact subset $K \subseteq U$ there exists $\rh>0$ such that $\|f\|^{W^x}_{K,\rh} < \infty$.
  By \Cref{lem:propertiesWF}(5) and \eqref{eq:connectionSN}, there exist constants $H, C \ge 1$ such that
  \[
  \infty >
  C \|f\|^{W^x}_{K,\rh} \ge
  \|f\|^{W^{Hx}}_{K,1} = \|f\|^{\om}_{K,Hx},
  \]
  whence $f \in \mathcal{E}^{\{\om\}}(U)$.

  Assume that $f \in \mathcal{E}^{(\om)}(U)$.
  Let $x>0$ and $\rh >0$ be fixed. By \Cref{lem:propertiesWF}(5),
  there exist constants $H, C \ge 1$ such that
  $W^x_k \le C \rh^k W^{Hx}_k$ for all $k$.
  For each compact subset $K \subseteq U$ we have
  $\|f\|^{\om}_{K,\frac{x}{H}} < \infty$, and, thus, using \eqref{eq:connectionSN},
  \[
  \infty > C \|f\|^{\om}_{K,\frac{x}{H}} = C \|f\|^{W^{\frac{x}{H}}}_{K,1} \ge \|f\|^{W^x}_{K,\rh}.
  \]
  Since $\rh>0$ was arbitrary, we may conclude that $f \in \mathcal{E}^{(W^x)}(U)$.

  (2) now follows from (1),
  since we have the continuous inclusions $\mathcal{E}^{(\om)}(K) \supseteq \on{proj}_{x>0} \mathcal{E}^{(W^x)}(K)$ and
  $\mathcal{E}^{\{\om\}}(K) \subseteq \on{ind}_{x>0} \mathcal{E}^{\{W^x\}}(K)$ by definition.
\end{proof}

\begin{corollary}
	Let $\om$ be a weight function. Then:
	\begin{enumerate}
		\item $\mathcal{E}^{[\om]}(U)$ is a ring.
		\item If $C^\om(U)\subseteq \mathcal{E}^{[\om]}(U)$ and $\psi : \RR^m \supseteq V \to U$ is real analytic,
		then the pullback $\psi^* : \mathcal{E}^{[\om]}(U) \to \mathcal{E}^{[\om]}(V)$ is well-defined.
		\item $\mathcal{E}^{[\om]}(U)$ is stable under differentiation.
	\end{enumerate}
\end{corollary}

\begin{proof}
	(1) and (2) follow from the fact that each $W^x$ is a weight sequence; cf.\ \Cref{sec:stability}.
	(3) is a consequence of \Cref{lem:propertiesWF}(3).
\end{proof}

\subsection{Braun--Meise--Taylor vs.\ Denjoy--Carleman classes}

\Cref{thm:WMdescription} shows that every Braun--Meise--Taylor class can be represented as a union or an intersection
of Denjoy--Carleman classes.
The question arises as to when Braun--Meise--Taylor and Denjoy--Carleman classes coincide.

\begin{theorem} \label{thm:comparison}
	Let $\om$ be a weight function with associated weight matrix $\mathfrak{W} = \{W^x\}_{x>0}$.
	Let $U \subseteq \RR^n$ be open and non-empty. The following conditions are equivalent:
\begin{enumerate}
	\item $\exists H\ge 1 ~\forall t\ge 0: 2 \om(t) \le \om(Ht) + H$.
	\item The weight sequences $W^x$ and $W^y$ are equivalent for all $x,y>0$.
	\item $\mathcal{E}^{\{\om\}}(U) = \mathcal{E}^{\{W^x\}}(U)$ for all $x>0$.
	\item $\mathcal{E}^{(\om)}(U) = \mathcal{E}^{(W^x)}(U)$ for all $x>0$.
	\item $W^x$ has moderate growth for some (equivalently, for all) $x>0$.
\end{enumerate}
\end{theorem}

The equivalence of (1) and (2) can be seen by straightforward computations.
Then $(2) \Leftrightarrow (3) \Leftrightarrow (4)$ follow from \Cref{thm:WMdescription}.
For the equivalence with (5) see \cite{RainerSchindl12}.
Conversely, one may ask when for a given weight sequence $M$ there exists a weight function $\om$
such that $\mathcal{E}^{[M]}(U) = \mathcal{E}^{[\om]}(U)$.

\begin{theorem}[{\cite[Theorem 14]{BMM07}}]
	Let $M=(M_k)$ be a derivation-closed weight sequence with $m_k^{1/k} \to \infty$.
	The following conditions are equivalent:
	\begin{enumerate}
		\item There exists a weight function $\om$ such that for all open $U \subseteq \RR^n$ we
		have $\mathcal{E}^{[\om]}(U) = \mathcal{E}^{[M]}(U)$ as vector spaces and/or as locally convex spaces.
		\item The weight sequence $M$ has moderate growth and satisfies
		\[
			\exists Q \in \NN_{\ge 1}:
			\liminf_{k \to \infty} \frac{\mu_{Qk}}{\mu_k} >1.
		\]
		\item The weight sequence $M$ has moderate growth, the associated function $\om_M(t) = \sup_{k \in \N} \log \frac{t^k}{M_k}$ is a weight function, and
		(1) holds with $\om=\om_M$.
	\end{enumerate}
\end{theorem}

\begin{example} \label{ex:WF}
  (1) Let $s\ge 1$. Then $\ga_s(t) := \max\{0,t^{\frac{1}{s}}-1\}$ is a weight function and we have the identity
  \[
    \mathcal{E}^{[\ga_s]}(U) = \mathcal{E}^{[G^s]}(U),
  \]
  where $G^s$ denotes the Gevrey sequence $G^s_k=k!^s$. See also \Cref{ex:strongWF}.

  (2) For $s> 1$, $\om_s(t) := \max\{0,(\log t)^s\}$ is a weight function with associated weight matrix $\fW^s = \{W^{s,x}\}_{x>0}$ given by
  $W^{s,x}_k = \exp(C_s \, x^{1/(s-1)}\, k^{s/(s-1)})$.
  The classes $\cE^{[\om_s]}$ are \emph{not} Denjoy--Carleman classes.
\end{example}

\subsection{Weight matrices and associated ultradifferentiable classes}

The weight matrix description of Braun--Meise--Taylor classes motivates the definition of ultradifferentiable classes by general weight matrices,
providing a general framework for ultradifferentiability.

\begin{definition}
  A \emph{weight matrix}\index{weight matrix} is by definition a family $\mathfrak{M}$ of weight sequences
  $M \in \mathfrak{M}$ which is totally ordered with respect to the natural order relation on sequences (i.e., for all $M,N \in \mathfrak{M}$ we have $M \le N$ or $N \le M$).
\end{definition}

Let $\mathfrak{M}$ be a weight matrix. Let $U \subseteq \RR^n$ open and $K \subseteq U$ compact. We define
\begin{align*}
	\mathcal{E}^{\{\mathfrak{M}\}}(K)
	&:= \bigcup_{M \in \mathfrak{M}} \mathcal{E}^{\{M\}}(K),
	\\
	\mathcal{E}^{(\mathfrak{M})}(K)
	&:= \bigcap_{M \in \mathfrak{M}} \mathcal{E}^{(M)}(K),
	\\
	\mathcal{E}^{\{\mathfrak{M}\}}(U)
	&:= \bigcap_{K \subseteq_{cp} U} \bigcup_{M \in \mathfrak{M}} \mathcal{E}^{\{M\}}(K) = \bigcap_{K \subseteq_{cp} U} \mathcal{E}^{\{\mathfrak{M}\}}(K),
	\\
	\mathcal{E}^{(\mathfrak{M})}(U)
	&:= \bigcap_{K \subseteq_{cp} U} \bigcap_{M \in \mathfrak{M}} \mathcal{E}^{(M)}(K) = \bigcap_{K \subseteq_{cp} U} \mathcal{E}^{(\mathfrak{M})}(K)
  = \bigcap_{M \in \mathfrak{M}} \mathcal{E}^{(M)}(U)
\end{align*}
and endow these spaces with their natural locally convex topologies.
The class $\mathcal{E}^{(\mathfrak{M})}(U)$ is said to be of \emph{Beurling type}\index{Beurling type}
and the class $\mathcal{E}^{\{\mathfrak{M}\}}(U)$ of \emph{Roumieu type}\index{Roumieu type}.
Again we use $\mathcal{E}^{[\mathfrak{M}]}(U)$ as placeholder for either $\mathcal{E}^{(\mathfrak{M})}(U)$
or $\mathcal{E}^{\{\mathfrak{M}\}}(U)$.

For each weight matrix $\fM$ there exists a countable weight matrix $\fM'$ such that $\cE^{[\fM]}(U) = \cE^{[\fM']}(U)$
as locally convex spaces; see \cite[Lemma 2.5]{FurdosNenningRainer}.

Note that all $\mathcal{E}^{[\mathfrak{M}]}(U)$ are rings and if $C^\om(U) \subseteq \mathcal{E}^{[\mathfrak{M}]}(U)$
then the class $\mathcal{E}^{[\mathfrak{M}]}$ is stable by analytic change of coordinates.

\subsection{Inclusion relations}

The inclusion relations for the classes $\mathcal{E}^{[\mathfrak{M}]}$ follow easily from those for Denjoy--Carleman classes.
For weight matrices $\mathfrak{M}$ and $\mathfrak{N}$ we define
\begin{align*}
	\mathfrak{M} (\preceq) \mathfrak{N}
	\quad &:\Leftrightarrow \quad \forall N \in \mathfrak{N} ~\exists M \in \mathfrak{M} : M \preceq N,
	\\
	\mathfrak{M} \{\preceq\} \mathfrak{N}
	\quad &:\Leftrightarrow \quad \forall M \in \mathfrak{M} ~\exists N \in \mathfrak{N} : M \preceq N,
	\\
	\mathfrak{M} (\preceq\} \mathfrak{N}
	\quad &:\Leftrightarrow \quad \exists M \in \mathfrak{M} ~\exists N \in \mathfrak{N} : M \preceq N,
	\\
	\mathfrak{M} \{\lhd) \mathfrak{N}
	\quad &:\Leftrightarrow \quad \forall M \in \mathfrak{M} ~\forall N \in \mathfrak{N} : M \lhd N.
\end{align*}
Moreover, we set
\begin{align*}
	\mathfrak{M} [\approx] \mathfrak{N} \quad &:\Leftrightarrow \quad
	\mathfrak{M} [\preceq] \mathfrak{N} [\preceq] \mathfrak{M}.
\end{align*}

\begin{lemma}[\cite{RainerSchindl12}] \label{lem:WMincl}
	Let $\mathfrak{M}$ and $\mathfrak{N}$ be two weight matrices. Let $U \subseteq \RR^n$ be open and non-empty. Then:
	\begin{enumerate}
		\item $\mathfrak{M} [\preceq] \mathfrak{N}$ if and only if
		$\mathcal{E}^{[\mathfrak{M}]}(U) \subseteq \mathcal{E}^{[\mathfrak{N}]}(U)$.
		\item $\mathfrak{M} (\preceq\} \mathfrak{N}$ if and only if
		$\mathcal{E}^{(\mathfrak{M})}(U) \subseteq \mathcal{E}^{\{\mathfrak{N}\}}(U)$.
		\item $\mathfrak{M} \{\lhd) \mathfrak{N}$ if and only if
		$\mathcal{E}^{\{\mathfrak{M}\}}(U) \subseteq \mathcal{E}^{(\mathfrak{N})}(U)$.
	\end{enumerate}
	All inclusions are continuous. For the ``if'' part it suffices to have the inclusions on $U\subseteq \R$.
\end{lemma}

\begin{corollary} \label{cor:WMincl}
  A class of Roumieu type and a class of Beurling type
	can never coincide.
\end{corollary}

\begin{proof}
  Observe that
 $\mathfrak M (\preceq\} \mathfrak N \{\lhd) \mathfrak M$ is impossible
 for any two weight matrices $\fM$ and $\fN$ and apply \Cref{lem:WMincl}.
\end{proof}

In the special case of Braun--Meise--Taylor classes we obtain

\begin{proposition} \label{prop:Omincl}
	Let $\om$ and $\si$ be weight functions.  Let $U \subseteq \RR^n$ be open and non-empty. Then:
	\begin{enumerate}
		\item $\si(t) = O(\om(t))$ as $t \to \infty$ if and only if
		$\mathcal{E}^{[\om]}(U) \subseteq \mathcal{E}^{[\si]}(U)$.
		\item $\si(t) = o(\om(t))$ as $t \to \infty$ if and only if
		$\mathcal{E}^{\{\om\}}(U) \subseteq \mathcal{E}^{(\si)}(U)$.
		\item $\si(t) = O(\om(t))$ as $t \to \infty$ if and only if
		$\mathcal{E}^{(\om)}(U) \subseteq \mathcal{E}^{\{\si\}}(U)$.
	\end{enumerate}
	All inclusions are continuous. For the ``if'' part it suffices to have the inclusions on $U \subseteq \R$.
\end{proposition}

Note that $\om$ and $\si$ are said to be \emph{equivalent} if
$\si(t) = O(\om(t))$ and $\om(t) = O(\si(t))$ as $t \to \infty$.

\begin{proof}
	Let $\mathfrak{W}=\{W^x\}_{x>0}$ and $\mathfrak{S}=\{S^x\}_{x>0}$ be the associated weight matrices of
	$\om$ and $\si$, respectively.
	Let us prove (1); (2) and (3) can be shown similarly.
	In view of \Cref{lem:WMincl} it suffices to prove
	\begin{equation} \label{eq:Omincl}
	 	 \text{$\si(t) = O(\om(t))$ as $t \to \infty$ if and only if $\mathfrak{W} [\preceq] \mathfrak{S}$.}
	 \end{equation}
	Assume that $\si(t) = O(\om(t))$ as $t \to \infty$. Then there exists $H \ge 1$ such that $\si(t) \le H \om(t) + H$
	for all $t \ge 0$. This implies $H \ph_\om^*(t) \le \ph_\si^*(Ht)+H$ for the Young conjugates of $t \mapsto \si(e^t)$ and $t \mapsto \om(e^t)$. In particular, for $t = x k$ we find
	\[
		W^x_k = \exp(\tfrac{1}{x} \ph_\om^*(xk)) \le \exp(\tfrac{1}{Hx}\ph_\si^*(Hxk)+ \tfrac{1}{x}) = e^{1/x} S^{Hx}_k.
	\]
	That means there is $H \ge 1$ such that for all $x>0$ we have $W^x \le e^{1/x} S^{Hx}$
	which implies $\mathfrak{W} (\preceq) \mathfrak{S}$ as well as $\mathfrak{W} \{\preceq\} \mathfrak{S}$.

	For the converse direction we assume $\mathfrak{W} \{\preceq\} \mathfrak{S}$; for the assumption $\mathfrak{W} (\preceq) \mathfrak{S}$
  the arguments are analogous.
	This means, using \Cref{lem:propertiesWF}(5),
	\[
		\forall x>0 ~\exists y>0 ~\exists C>0 ~\forall k\in \NN :
		\tfrac{1}{x} \ph_\om^*(xk) \le \tfrac{1}{y} \ph_\si^*(yk) + C.
	\]
	By the convexity of $\ph_\om^*$ and $\ph_\si^*$ we may conclude that
	\[
		\forall x>0 ~\exists y>0 ~\exists D>0 ~\forall t\ge 0 :
		\tfrac{1}{x} \ph_\om^*(xt) \le \tfrac{1}{y} \ph_\si^*(yt) + C
	\]
	 (for different $y$ and $C$). Then
	 \begin{align*}
	 	\tfrac{1}{y} \ph_\si(t) = \sup_{s\ge 0} ( ts - \tfrac{1}{y}\ph_\si^*(ys)) \le
	 	 \sup_{s\ge 0} ( ts - \tfrac{1}{x}\ph_\om^*(xs)) + C = 	\tfrac{1}{x}\ph_\om(t)+C,
	 \end{align*}
	 that is
	 \begin{equation*}
	 	\tfrac{1}{y} \si(t) \le \tfrac{1}{x}\om(t)+C
  \end{equation*}
	 which implies $\si(t) = O(\om(t))$ as $t \to \infty$.
\end{proof}

\begin{corollary}
	Let $\om$ be a weight function. Let $U\subseteq \RR^n$ be open and non-empty.
	\begin{enumerate}
			\item $\mathcal{H}(\mathbb{C}^n) \subseteq \mathcal{E}^{(\om)}(U)$ if and only if
			$\om(t) = O(t)$ as $t \to \infty$.
		\item $C^\omega(U) \subseteq \mathcal{E}^{(\om)}(U)$ if and only if
		$\om(t) = o(t)$ as $t \to \infty$.
		\item $C^\omega(U) \subseteq \mathcal{E}^{\{\om\}}(U)$ if and only if
		$\om(t) = O(t)$ as $t \to \infty$.
	\end{enumerate}
\end{corollary}

\begin{proof}
	This follows from \Cref{ex:WF}(1) and \Cref{prop:Omincl}.
\end{proof}

\subsection{Intersection and union of all non-quasianalytic Gevrey classes}

Let $\mathfrak G := \{G^s\}_{s > 1}$ be the weight matrix consisting of all non-quasianalytic Gevrey sequences $G^s_k = k!^s$.
It turns out that $\cE^{(\fG)}$ and $\cE^{\{\fG\}}$ are neither Denjoy--Carleman nor Braun--Meise--Taylor classes.

\begin{theorem}[{\cite{RainerSchindl12}}]
  	The weight matrix $\fG= \{G^s\}_{s > 1}$ has the following properties:
  	\begin{enumerate}
  		\item  We have $L \lhd G^s \lhd G^t$ for all $1<s<t$, where $L=(L_k)$ is the weight sequence $L_k := k^k (\log(k + e))^{2k}$.
  		\item  Let $U \subseteq \R^n$ be open and $K \subseteq U$ compact. As locally convex spaces
      \begin{align*}
        \mathcal E^{(\mathfrak G)}(K)
  			&= \on{proj}_{s>1} \mathcal E^{(G^s)}(K)
  			= \on{proj}_{s>1} \mathcal E^{\{G^s\}}(K),
        \\
        \mathcal E^{\{\mathfrak G\}}(K)
  			&= \on{ind}_{s>1} \mathcal E^{\{G^s\}}(K)
  			= \on{ind}_{s>1} \mathcal E^{(G^s)}(K),
        \\
        \mathcal E^{(\mathfrak G)}(U)
  			&= \on{proj}_{s>1} \mathcal E^{(G^s)}(U)
  			= \on{proj}_{s>1} \mathcal E^{\{G^s\}}(U),
        \\
        \mathcal E^{\{\mathfrak G\}}(U)
  			&= \on{proj}_{K \subseteq_{cp} U} \on{ind}_{s>1} \mathcal E^{\{G^s\}}(K)
  			= \on{proj}_{K \subseteq_{cp} U} \on{ind}_{s>1} \mathcal E^{(G^s)}(K).
      \end{align*}
  		\item $\mathcal E^{(\mathfrak G)}$ and $\mathcal E^{\{\mathfrak G\}}$ are non-quasianalytic.
  		\item Neither $\mathcal E^{(\mathfrak G)}$ nor
  		$\mathcal E^{\{\mathfrak G\}}$ coincides with a Denjoy--Carleman or a Braun--Meise--Taylor class (as sets).
  	\end{enumerate}
\end{theorem}

\begin{proof}
  (1) If $1<s<t$, then
	\[
		\Big(\frac{G^s_k}{G^t_k}\Big)^{1/k} = k!^{\frac{s-t}k} \le k^{s-t} \to 0
	\]
	and
	\[
		\Big(\frac{L_k}{G^s_k}\Big)^{1/k} = \frac{k}{k!^{s/k}} (\log(k + e))^{2}
		\le e k^{1-s}  (\log(k + e))^{2} \to 0.
	\]

	(2) This follows from the definition and from the fact that $G^s \lhd G^t$ if
	$s<t$.

	(3) is a consequence of (1), since the weight sequence $L$ is non-quasianalytic. It also follows from \Cref{thm:WMqa}.

	(4)
	Suppose that there is a weight sequence $M$ such that
	$\mathcal{E}^{(M)}(\RR)= \mathcal{E}^{(\mathfrak G)}(\RR)$.
	Then \Cref{lem:WMincl} implies $M (\preceq) \mathfrak G (\preceq) M$,
	i.e., $M \preceq G^s$
	for all $s>1$ and there exists $t>1$ such that $G^t \preceq M$.
	Thus $M$ is equivalent to $G^t$ and hence
	\[
	\mathcal{E}^{(M)}(\RR)
	= \mathcal{E}^{(G^t)}(\RR)
	\supsetneq \mathcal{E}^{(\mathfrak G)}(\RR)
	=\mathcal{E}^{(M)}(\RR),
	\]
	a contradiction.
  Similarly,
	if $\mathcal{E}^{\{M\}}(\RR)= \mathcal{E}^{\{\mathfrak G\}}(\RR)$ for
	some weight sequence $M$, then
	we conclude $M \{\preceq\} \mathfrak G \{\preceq\} M$,
	i.e., $M \preceq G^s$
	for some $s>1$ and $G^t \preceq M$ for all $t>1$.
	So $M$ is equivalent to $G^s$ which leads to the contradiction
	\[
	\mathcal{E}^{\{M\}}(\RR)
	= \mathcal{E}^{\{G^s\}}(\RR)
	\subsetneq \mathcal{E}^{\{\mathfrak G\}}(\RR)
	=\mathcal{E}^{\{M\}}(\RR).
	\]

	Now assume that there exists a weight function $\om$ such that
	$\mathcal{E}^{(\om)}(\RR)= \mathcal{E}^{(\mathfrak G)}(\RR)$.
	Let $\mathfrak W = \{W^x\}_{x>0}$ be the weight matrix associated with $\om$.
	Then $\mathfrak W (\preceq) \mathfrak G (\preceq) \mathfrak W$,
	i.e., for all $t>1$ there exists $x>0$ such that $W^x \preceq G^t$ and
	for all $y>0$ there exists $s>1$ such that $G^s \preceq W^y$.
	It follows that
	\begin{equation*}
	 	\forall y>0 ~\exists s>1 ~\exists x>0 : W^x \preceq G^s \preceq W^y.
	 \end{equation*}
	 As a consequence $\mathcal{E}^{[W^x]}(\RR) \subseteq \mathcal{E}^{[G^s]}(\RR)
	 \subseteq \mathcal{E}^{[W^y]}(\RR)$.
	 But, by \Cref{ex:WF}(1), we have $\mathcal{E}^{[G^s]}(\RR)
	 = \mathcal{E}^{[\ga_s]}(\RR)$, where $\ga_s(t)= t^{1/s}$.
	 Consequently,
	 \[
	 	\mathcal{E}^{(\om)}(\RR) \subseteq \mathcal{E}^{(W^x)}(\RR)
	 	\subseteq \mathcal{E}^{(G^s)}(\RR) = \mathcal{E}^{(\ga_s)}(\RR)
	 \]
	 and thus $\ga_s(t) = O(\om(t))$ as $t \to \infty$, by \Cref{prop:Omincl}.
	 On the other hand,
	 \[
	 	\mathcal{E}^{\{\ga_s\}}(\RR) = \mathcal{E}^{\{G^s\}}(\RR)
	 	\subseteq \mathcal{E}^{\{W^y\}}(\RR)
	 	\subseteq \mathcal{E}^{\{\om\}}(\RR)
	 \]
	 and hence $\om(t) = O (\ga_s(t))$ as $t \to \infty$.
	 But this leads to a contradiction:
	 \[
	 	\mathcal{E}^{(G^s)}(\RR)=\mathcal{E}^{(\ga_s)}(\RR) = \mathcal{E}^{(\om)}(\RR)
	 	= \mathcal{E}^{(\mathfrak G)}(\RR) \subsetneq \mathcal{E}^{(G^s)}(\RR).
	 \]
   In analogy, one shows that $\mathcal{E}^{\{\om\}}(\RR)= \mathcal{E}^{\{\mathfrak G\}}(\RR)$ for some weight function $\om$
   is impossible.

  In view of \Cref{cor:WMincl} the proof of (4) is complete.
\end{proof}

\begin{remark}
  The classes $\cE^{[\fG]}$ enjoy good properties that the single Gevrey classes do not have,
  as investigated in \cite{ChaumatChollet98}.
  In fact, if $M$ is a strongly log-convex weight sequence of moderate growth such that the ``power''
  $M^{(a)} := (k!\, m_k^a)$ is non-quasianalytic for every $a>0$ (e.g.\ strongly regular sequences or
  $M_k:=k!\, \exp(\sum_{j=1}^k (\log j)^\de)$ for $\de \in (0,1]$),
  then the class $\bigcap_{a>0} \cB^{\{M^{(a)}\}}$ admits versions of
  \begin{enumerate}
    \item Whitney's extension theorem,
    \item {\L}ojasiewicz's theorem on regularly situated compact sets,
    \item Weierstrass' division and preparation theorems,
    \item Whitney's spectral theorem.
  \end{enumerate}
  Let us give a short argument for (1):
  for any jet $F$ of class $\bigcap_{a>0} \cB^{\{M^{(a)}\}}$ on a compact set $K$
  one can find a strongly log-convex non-quasianalytic weight sequence $L$ of moderate growth
  such that $L \lhd M^{(a)}$ for all $a>0$ and $F$ is of class $\cE^{\{L\}}$; see \cite[Proposition 5]{ChaumatChollet98}.
  Then, if $L_k = k!\, \ell_k$ and $\la_k^*:= \frac{\ell_k}{\ell_{k-1}}$, we find
  \[
     \sum_{j\ge k} \frac{1}{j (\la_j^*)^2} \le \frac{1}{\la_k^*} \sum_{j\ge k} \frac{1}{j \la_j^*} \lesssim \frac{1}{\la_k^*},
  \]
  and hence $F$ has an extension $f$ in $\cB^{\{L^{(2)}\}}(\R^n)$, by \Cref{thm:CCmixed}.
  But
  \[
  \Big(\frac{L_k^{(2)}}{M_k^{(a)}}\Big)^{1/k} = \Big(\frac{\ell_k^{2}}{m_k^{a}}\Big)^{1/k} = \Big(\frac{\ell_k}{m_k^{a/2}}\Big)^{2/k} = \Big(\frac{L_k}{M_k^{(a/2)}}\Big)^{2/k} \to 0,
  \]
  for all $a>0$ so that $f \in \bigcap_{a>0} \cB^{\{M^{(a)}\}}$.
\end{remark}

\subsection{Quasianalyticity and non-quasianalyticity}

The weight matrix description of Braun--Meise--Taylor classes, i.e.\ \Cref{thm:WMdescription},
allows us to deduce the characterization of (non-)quasianalyticity from the Denjoy--Carleman theorem \ref{thm:DC}.
For a different proof relying on H\"ormander's $L^2$-method and a Paley--Wiener theorem see \cite{BMT90}.

\begin{theorem} \label{thm:WMqa}
	Let $\mathfrak{M}$ be a weight matrix. Then:
	\begin{enumerate}
		\item $\mathcal{E}^{\{\mathfrak{M}\}}$ is quasianalytic if and only if all $M \in \mathfrak{M}$ are quasianalytic.
		\item $\mathcal{E}^{(\mathfrak{M})}$ is quasianalytic if and only if some $M \in \mathfrak{M}$ is quasianalytic.
	\end{enumerate}
\end{theorem}

\begin{proof}
	(1) If some $M \in \mathfrak{M}$ is non-quasianalytic, then there exists a
	non-trivial function with compact support of class $\mathcal{E}^{\{M\}}$, thus of of class $\mathcal{E}^{\{\mathfrak{M}\}}$.
	Conversely, any $\mathcal{E}^{\{\mathfrak{M}\}}$-function with compact support is a
	$\mathcal{E}^{\{M\}}$-function for some $M \in \mathfrak{M}$.

	(2) If some $M \in \mathfrak{M}$ is quasianalytic, then clearly $\mathcal{E}^{(\mathfrak{M})}$ is quasianalytic since
	$\mathcal{E}^{(\mathfrak{M})} \subseteq \mathcal{E}^{(M)}$.
	Conversely, assume that all $M \in \mathfrak{M}$ are non-quasianalytic. We may assume that the weight matrix
	$\mathfrak{M}$ is countable (\cite[Lemma 2.5]{FurdosNenningRainer}),
  i.e.,  $\mathfrak{M} = \{M^n\}_{n\in\NN}$, where $M^n \ge M^{n+1}$ for all $n$.
	Set $\al^n_k := (M^n_k)^{-1/k}$. Then, for each $n$ the sequence $\al^n$ is decreasing and $\sum_k a^n_k <\infty$, by \Cref{lem:Carleman}.
  Moreover $\al^n \le \al^{n+1}$.

	We claim that there is a decreasing positive sequence $\al$ such that $\sum_k \al_k <\infty$
	and for each $n$ there exists $k_n$ such that $\al_k \ge \al^n_k$ for all $k \ge k_n$.
  Then we may use \Cref{lem:smaller} and \Cref{lem:Carleman} as in the proof of \Cref{thm:DC} to see that
	$\mathcal{E}^{(M)}$ with $M_k := a_k^{-k}$, thus also $\mathcal{E}^{(\mathfrak{M})}$, is non-quasianalytic.

	It remains to show the claim; the proof is based on ideas from \cite[Proposition 4.7]{Schindl15}.
	We define recursively two increasing sequences of integers $(p_j)_{j \in \N}$ and $(q_j)_{j \in \N}$.
	Set $p_0=q_0=0$ and let
	\begin{itemize}
		\item $p_j$ be the minimal integer such that $p_j > q_{j-1}$ and $\sum_{k>p_j} \al^{j+1}_k \le 2^{-j}$,
		\item $q_j$ be the minimal integer such that $\al^{j}_{p_j} > \al^{j+1}_{q_j+1}$.
	\end{itemize}
	This makes sense since $\al^n_k \to 0$ as $k \to \infty$ for each $n$.
	Moreover it implies $q_j\ge p_j$.
	Then we define
	\[
		\al_k :=
		\begin{cases}
			\al^{j}_k & \text{ if } q_{j-1}< k \le p_j,
			\\
			\al^{j}_{p_j} & \text{ if } p_j < k \le q_j.
		\end{cases}
	\]
	By construction $\al$ is decreasing.
	By the minimality of $q_j$ we have $\al^{j}_{p_j} \le \al^{j+1}_{k}$ for $p_j < k \le q_j$. Thus
	\begin{align*}
		\sum_k \al_k
		&= \sum_j \Big(\sum_{p_{j}< k \le q_{j}} \al^{j}_{p_j} + \sum_{q_j < k \le p_{j+1}} \al^{j+1}_{k}\Big)
		\le \sum_j \Big(\sum_{p_j < k \le p_{j+1}} \al^{j+1}_{k}\Big)
		\le  \sum_j 2^{-j}.
	\end{align*}
	For fixed $n$ we have $a_k \ge a^{n}_k$ for all $k \ge q_{n-1}$. The claim is proved.
\end{proof}

\begin{theorem} \label{thm:(non)qaWF}
	Let $\om$ be a weight function with associated weight matrix $\mathfrak{W}= \{W^x\}_{x>0}$.
	The following conditions are equivalent:\index{quasianalytic}
	\begin{enumerate}
		\item $\mathcal{E}^{\{\om\}}$ is quasianalytic.
		\item $\mathcal{E}^{(\om)}$ is quasianalytic.
		\item The weight function $\om$ is \emph{quasianalytic}\index{weight function!quasianalytic}, i.e.,
			\[
				\int_1^\infty \frac{\om(t)}{t^2}\, dt = \infty.
			\]
		\item The weight sequence $W^x$ is quasianalytic for all $x>0$.
		\item The weight sequence $W^x$ is quasianalytic for some $x>0$.
	\end{enumerate}
\end{theorem}

\begin{proof}
  It suffices to show the equivalence of (3), (4), and (5). The rest follows from \Cref{thm:WMqa}.
  Now $W^x$ is quasianalytic if and only if the associated function $\om_{W^x}(t) = \sup_{k \in \N} \log \frac{t^k}{W^x_k}$ satisfies
  \(\int_1^\infty \frac{\om_{W^x}(t)}{t^2}\, dt = \infty\); cf.\
  \cite[Lemma 4.1]{Komatsu73}. Each $\om_{W^x}$ is equivalent to $\om$; cf.\ \cite[Lemma 5.7]{RainerSchindl12}.
\end{proof}

Note that a \emph{non-quasianalytic} weight function $\om$ necessarily
satisfies $\om(t) = o(t)$ as $t \to \infty$,
	i.e.,
	$C^\om(U) \subseteq \mathcal{E}^{(\om)}(U)$.
Indeed, since $\om$ is increasing,
	\begin{align*}
	 \frac{\om(t)}{t} = \int_t^\infty \frac{\om(t)}{s^2}\, ds \le \int_t^\infty \frac{\om(s)}{s^2}\, ds \to 0 \quad \text{ as } t \to \infty.
	\end{align*}

\subsection{Stability properties}

We state without proof a characterization of stability under composition for Braun--Meise--Taylor classes;
for the general classes $\cE^{[\fM]}$ we refer to \cite{RainerSchindl12}.

\begin{theorem}[\cite{FernandezGalbis06}, \cite{RainerSchindl12}]
  Let $\om$ be a weight function satisfying $\om(t)=O(t)$ as $t \to \infty$.
 	The following conditions are equivalent.
 	\begin{enumerate}
 		\item $\exists C>0 ~\exists t_0>0 ~\forall \la\ge 1 ~\forall t \ge t_0 :  \om(\la t) \le C \la \om(t)$.
 		\item $\om$ is equivalent to a concave weight function.
 		\item $\om$ is equivalent to a subadditive weight function.
 		\item $\mathcal{E}^{\{\om\}}$ is stable under composition.
 		\item $\mathcal{E}^{(\om)}$ is stable under composition.
 	\end{enumerate}
 	The condition $\om(t)=O(t)$ as $t \to \infty$ is only used in the directions
 	(4) $\Rightarrow$ (1)
 	and (5) $\Rightarrow$ (1).
\end{theorem}

That subadditivity is a sufficient condition for stability under composition is seen as follows:
 Let $\mathfrak{W} = \{W^x\}_{x>0}$ be the associated weight matrix.
	We have
	\begin{align*}
		W^x_k &= \exp(\tfrac{1}{x} \ph^*(xk))
    = \exp \sup_{t \ge 1} (k \log(t) - \tfrac{1}{x} \om(t))
		= \sup_{t \ge 1} \big(t^k e^{- \tfrac{1}{x} \om(t)}\big).
	\end{align*}
	Thus subadditivity of $\om$ implies
	\begin{align*}
		w^x_j w^x_k &=
		\sup_{t,s \ge 1} \Big(\frac{t^js^k}{j!k!} e^{- \tfrac{1}{x} (\om(t)+\om(s))}\Big)
    \le
		\sup_{t,s \ge 1} \Big(\frac{(t+s)^{j+k}}{(j+k)!} e^{- \tfrac{1}{x} \om(t+s)}\Big)
		= w^x_{j+k}.
	\end{align*}
  Together with \Cref{lem:propertiesWF}(3) we see that,
	for all $\al_i \in \NN_{>0}$ with $\al_1 + \cdots + \al_j = k$,
	\begin{align*}
		w^x_j w^x_{\al_1} \cdots w^x_{\al_j}
		\le w_1^j w^{2x}_j w^{2x}_{\al_1-1} \cdots w^{2x}_{\al_j-1}
		\le w_1^j w^{2x}_k
	\end{align*}
	which implies $(W^x)^\o \preceq W^{2x}$ for all $x>0$. From this stability under composition follows easily (cf.\ \Cref{sec:stability}).

\begin{remark} \label{rem:WMconcave}
  Let $\om$ be a weight function with $\om(t)=o(t)$ as $t \to \infty$.
  That $\om$ is equivalent to a concave weight function is furthermore equivalent to any of the following conditions:
  \begin{itemize}
   \item There is a weight matrix $\mathfrak S$ consisting of strongly log-convex weight sequences such that $\mathcal E^{\{\om\}} = \mathcal E^{\{\mathfrak{S}\}}$.
   \item There is a weight matrix $\mathfrak S$ consisting of strongly log-convex weight sequences such that $\mathcal E^{(\om)} = \mathcal E^{(\mathfrak{S})}$.
   \item $\mathcal E^{\{\om\}}$ can be described by almost analytic extensions.
   \item $\mathcal E^{(\om)}$ can be described by almost analytic extensions.
 \end{itemize}
  For all this (including the meaning of \emph{almost analytic extensions}) we refer to \cite[Theorem 4.8]{FurdosNenningRainer} and
  \cite[Theorem 11]{Rainer:2020aa}.
  Furthermore, these conditions are equivalent to the classes $\cE^{[\om]}$ to be stable under inverse/implicit functions and solving ODEs (in the sense described in \Cref{sec:stability}),
  respectively;
  see \cite{RainerSchindl14}.
\end{remark}

\section{Extension in Braun--Meise--Taylor classes} \label{sec:BMT}

\subsection{Whitney ultrajets}

Let $A\subseteq \R^n$ be a closed non-empty set. Let $\om$ be a weight function.
A Whitney jet $F=(F^\al)_{\al \in \N^n} \in \cE(A)$
is called a \emph{$\om$-Whitney ultrajet of Beurling type} on $A$
if for all compact subsets $K \subseteq A$ and all integers $m \ge 1$
we have
\begin{equation} \label{Omjet1}
 	\|F\|^\om_{K,1/m}:=\sup_{x \in K} \sup_{\al \in \N^n} |F^\al(x)| \exp\big(-m \vh^*\big(\tfrac{|\al|}{m}\big)\big) < \infty
 \end{equation}
 and
 \begin{equation} \label{Omjet2}
 	| F |^\om_{K,1/m}  := \sup_{\substack{x,y \in K\\ x \ne y}} \sup_{p \in \N } \sup_{|\al| \le p}  |(R^p_x F)^\al (y)|
 	\frac{(p+1-|\al|)!}{|x-y|^{p+1-|\al|}} \exp\big(-m \vh^*\big(\tfrac{p+1}{m}\big)\big) <\infty.
 \end{equation}
 We denote by $\cE^{(\om)}(A)$
 the locally convex space of all $\om$-Whitney ultrajets $F$ of Beurling type on $A$ equipped with the projective limit topology
 with respect to the system of seminorms $\|F\|^\om_{K,1/m} + |F|^\om_{K,1/m}$.
The space of \emph{$\om$-Whitney ultrajets of Roumieu type} on $A$
 is
 \[
 \cE^{\{\om\}}(A) := \{F \in \cE(A) : \forall K \subseteq_{cp} A ~\exists m\in \N_{\ge 1} : \|F\|^\om_{K,m} + |F|^\om_{K,m}<\infty \}
 \]
supplied with its natural locally convex topology.
In view of \Cref{lem:propertiesWF}(5), we have
\begin{align*}
   \cE^{(\om)}(A) &= \on{proj}_{K \subseteq_{cp} A} \on{proj}_{m>0} \cE^{(W^{1/m})}(K),
   \\
   \cE^{\{\om\}}(A) &= \on{proj}_{K \subseteq_{cp} A} \on{ind}_{m>0} \cE^{\{W^{m}\}}(K),
\end{align*}
where $\fW = \{W^x\}_{x>0}$ is the associated weight matrix of $\om$.

\subsection{Strong weight functions}
We shall see that $\cE^{[\om]}$ admits extension theorems preserving the class if and only if $\om$ is a \emph{strong} weight function.

\begin{definition}
  A non-quasianalytic weight function $\om$ is called \emph{strong} if
  \begin{equation} \label{def:strong}
    \E C>0 \A t>0 : \int_1^\infty \frac{\om(ut)}{u^2} \, du \le C \om(t) + C.
  \end{equation}
\end{definition}

\begin{lemma}[{\cite[Propositions 1.3 and 1.7]{MeiseTaylor88}}] \label{lem:MeiseTaylor}
  Let $\om : [0,\infty) \to [0,\infty)$ be an increasing function with $\om(0) =0$ and $\om(t) \to \infty$ as $t \to \infty$.
	The following conditions are equivalent:\footnote{If $\om$ satisfies these equivalent conditions, then $\om(t)=O(t^\al)$ as $t \to \infty$ for some $0<\al<1$;
	see \cite[Corollary 1.4]{MeiseTaylor88}. That means the class contains a Gevrey class (of Roumieu or Beurling type, respectively).}
	\begin{enumerate}
    \item $\om$ satisfies \eqref{def:strong}.
    \item There exist constants $K>H>1$ such that $\om(Kt) \le H \om(t)$ for all sufficiently large $t$.
	 	\item The increasing concave function
    \begin{equation}
      \ka(t):= \int_1^\infty \frac{\om(ut)}{u^2}\, du, \quad t>0,
    \end{equation}
    satisfies \eqref{def:strong} and $\om \le \ka \le C \om +C$ for some $C>0$.
    \item The harmonic extension
    \begin{equation}
      P_\om(x+ iy):= \begin{cases}
        \frac{|y|}{\pi} \int_\R \frac{\om(t)}{(t-x)^2 + y^2} \, dt & \text{ if } |y|>0,
        \\
        \om(x) & \text{ if } y=0,
    \end{cases}
    \end{equation}
    satisfies $P_\om(z) = O(\om(z))$ as $|z| \to \infty$, where $\om(z) := \om(|z|)$.
	 \end{enumerate}
\end{lemma}

\begin{example} \label{ex:strongWF}
  (1) $\om_s(t) = \max(0,(\log t)^s)$ is a strong weight function for each $s>1$.

  (2) Let $M=(M_k)$ be a strongly regular weight sequence.
  Then $\om_M$ satisfies \eqref{def:strong}, by \cite[Proposition 4.4]{Komatsu73}.
  By \Cref{lem:MeiseTaylor}(3), we have $\om_M \le \ka \le C \om_M +C$. Since $\ka$ is subadditive,
  \[
    \om_M(2t) \le \ka(2t)\le 2 \ka(t) \le 2C \om_M(t) + 2C,
  \]
  i.e., $\om_M$ satisfies \ref{sec:weightfunctions}(1) and hence is a weight function (the other conditions are always
  fulfilled by $\om_M$).
  Moreover, $2\om_M(t) \le \om_M(Dt) +D$, as $M$ has moderate growth, see \cite[Proposition 3.6]{Komatsu73}.
  Then \Cref{thm:comparison} applied to $\om_M$ shows that $\cE^{[M]} = \cE^{[\om_M]}$,
  since $M_k = \sup_{t\ge 0} \frac{t^k}{\exp(\om_M(t))} = e^{\vh^*(k)} =W^1_k$ (cf.\ \cite[Proposition 3.2]{Komatsu73}). Consequently, the extension results for strong weight functions
  comprise those for strongly regular weight sequences.
\end{example}

\subsection{The singleton and other sets with nice geometry}

It was shown in \cite{MeiseTaylor88} that Whitney ultrajets of class $\cE^{(\om)}$ admit extension from sets with nice geometry, including the singleton $\{0\}$.

\begin{theorem}[{\cite[Theorem 3.10]{MeiseTaylor88}}]
  Let $\om$ be a non-quasianalytic weight function.
  The following conditions are equivalent:
  \begin{enumerate}
    \item $\om$ is strong.
    \item $j^\infty_{\{0\}} : \cE^{(\om)}(\R^n) \to \cE^{(\om)}(\{0\})$ is surjective.
    \item $j^\infty_{K} : \cE^{(\om)}(\R^n) \to \cE^{(\om)}(K)$ is surjective for all compact convex $K \subseteq \R^n$ with non-empty interior.
    \item $j^\infty_{\ol \Om} : \cE^{(\om)}(\R^n) \to \cE^{(\om)}(\ol \Om)$ is surjective for all bounded open $\Om \subseteq \R^n$ with real analytic boundary.
  \end{enumerate}
\end{theorem}

Let us sketch an argument for the equivalence of (1) and (2).
The Fourier--Laplace transform is a linear topological isomorphism between $\cE^{(\om)}(\R^n)_b'$ and the weighted space of entire functions
\[
  \cA_1 (\C^n) := \Big\{f \in \cH(\C^n) : \E j\in \N : \sup_{z \in \C^n} |f(z)| e^{-j (|\Im z| + \om(z))}<\infty \Big\}
\]
equipped with its natural inductive limit topology; cf.\ \cite{BMT90}.
The dual $\cE^{(\om)}(\{0\})_b'$ can be identified with $(\La^{(\om)}_n)_b'$, where
\[
  \La^{(\om)}_n := \Big\{(c_\al) \in \C^{\N^n} : \A m \in \N : \sum_{\al} |c_\al| e^{-m \sum_{j=1}^n \vh^*(\frac{\al_j}{m})} < \infty\Big\}.
\]
Then the map $(\La^{(\om)}_n)_b' \ni (c_\al) \mapsto (z \mapsto \sum_\al c_\al (-iz)^\al)$
is a linear topological isomorphism between $(\La^{(\om)}_n)_b'$ and
\[
  \cA_2 (\C^n) := \Big\{f \in \cH(\C^n) : \E j\in \N : \sup_{z \in \C^n} |f(z)| e^{-j  \om(z)}<\infty \Big\}
\]
and we have the commutative diagram
\[
  \xymatrix{
    (\La^{(\om)}_n)_b' \ar[rr]^{(j^\infty_{\{0\}})^t} \ar[d]_{\cong} && \cE^{(\om)}(\{0\})_b' \ar[d]^{\cong}
    \\
    \cA_2 (\C^n) \ar[rr] && \cA_1 (\C^n)
  }
\]
where the bottom arrow is the inclusion map.
So (2) holds if and only if the inclusion $\cA_2(\C^n) \to \cA_1(\C^n)$ is an injective topological homomorphism.
It is shown in \cite{MeiseTaylor88} that the latter holds if and only if $P_\om(z) = O(\om(z))$ as $|z| \to \infty$;
the proof relies on the Phragm\'en--Lindel\"of principle and H\"ormander's $L^2$-estimates for the solution of the $\ol \p$-problem.

\begin{remark}
    An analogous result holds in the Roumieu case $\cE^{\{\om\}}$, see \cite{BonetMeiseTaylor89}.
\end{remark}

Concerning the existence of extension operators in the above cases we have

\begin{theorem}[\cite{Meise:1989un}] \label{thm:nicegeometry}
  Let $\om$ be a strong weight function.
  Then
  \begin{enumerate}
    \item $j^\infty_{\ol \Om} : \cE^{(\om)}(\R^n) \to \cE^{(\om)}(\ol \Om)$ is split surjective for all bounded open $\Om \subseteq \R^n$ with real analytic boundary.
    \item $j^\infty_{\{0\}} : \cE^{(\om)}(\R^n) \to \cE^{(\om)}(\{0\})$ is split surjective if and only if
    \begin{equation} \label{eq:DNweight}
      \A C>1 \E \de>0 \E t_0>0 \A t\ge t_0 :
        \om^{-1}(Ct)\om^{-1}(\de t) \le (\om^{-1}(t))^2.
    \end{equation}
  \end{enumerate}
\end{theorem}

A weight function $\om$ satisfying \eqref{eq:DNweight} is called a \emph{(DN)-weight}.

The proof of (1) is based on the splitting theorem of Vogt and Wagner \cite{Vogt:1980th} for the short exact sequence
of nuclear Fr\'echet spaces
\[
  \xymatrix{
  0 \ar[r] & \cD^{(\om)}(\ol B\setminus \Om) \ar[r] & \cD^{(\om)}(\ol B) \ar[r]^{j^\infty_{\ol \Om}} & \cE^{(\om)}(\ol \Om) \ar[r] & 0
  }
\]
where $B$ is a large open ball containing $\ol \Om$. By assumption, $K :=\ol B\setminus \Om$ is
the closure of a bounded open set with real analytic boundary
and, for such $K$, the space  $\cD^{(\om)}(K)$ has property ($\Om$), see \cite[Corollary 2.9]{Meise:1989un}.
On the other hand $\cE^{(\om)}(\ol \Om)$ has property (DN), by \cite[Proposition 5.7]{Meise:1989tj}.

For the singleton, we note that $\om$ satisfies \eqref{eq:DNweight} if and only if $\cE^{(\om)}(\{0\})$ has property (DN),
by \cite[Theorem 2.17 and Proposition 3.1]{Meise:1987tp}. Thus, in dimension $n=1$ the splitting theorem
can be applied as above. The  case $n>1$ follows by a tensor product argument; cf.\ \cite[3.1]{Meise:1989un}.
Since $\cD^{(\om)}(\ol B)$ has (DN), by \cite[Lemma 1.10(b)]{Meise:1989tj}, so does $\cE^{(\om)}(\{0\})$ if
the sequence splits, since (DN) is inherited by closed subspaces.
See also \cite[Corollary 3.12]{MeiseTaylor88}.

\begin{example}
  (1) If $M$ is a strongly regular weight sequence, then $\om_M$ is a (DN)-weight.

  (2) The weight functions $\om_s(t) = \max\{0,(\log t)^s\}$, $s>1$, are strong, but not (DN)-weights.
\end{example}

\begin{remark}
  The function $\om(t) = \max\{0,\log t\}$ is not a weight function, since condition \ref{sec:weightfunctions}(2) is violated.
  Nevertheless it satisfies \eqref{def:strong} and violates \eqref{eq:DNweight}. Then each seminorm $\|\cdot\|^\om_{K,\rh}$
  vanishes identically.
  So, formally, one may identify  $\cE^{(\om)}(U)$ with $C^\infty(U)$.
  Indeed,
  the theory of extension operators in the Braun--Meise--Taylor setting of Beurling type on the one hand and in
  the $C^\infty$ setting on the other hand
  have many similarities.
\end{remark}

\begin{remark}
  As for Denjoy--Carleman classes
  the Borel map is never onto in the quasianalytic case:
  Let $\om$ be a quasianalytic weight function such that $\om(t) = o(t)$ as $t \to \infty$, i.e.,
  the real analytic class is strictly contained in $\cE^{(\om)}$.
  Then there exist elements in $\La^{(\om)}$ that are not contained in the Borel image
  $\Bmap \cE^{\{\si\}}_0$ of \emph{any} quasianalytic weight function $\si$.
  This can be deduced from \Cref{seqRoumieu} using the description of Braun--Meise--Taylor classes
  given in \Cref{thm:WMdescription}. For details see \cite{RainerSchindl15}.
\end{remark}

\subsection{Optimal cutoff functions}

As in the Denjoy--Carleman setting, to address
the extension problem for general closed sets one needs cutoff functions with certain \emph{optimal} estimates.
Such functions are obtained by constructing certain entire functions and applying the Paley--Wiener theorem.
This in turn is based on H\"ormanders estimates for the solution of the $\ol \p$-problem and ultimately
boils down to finding subharmonic functions on $\C$ with suitable upper and lower bounds.

The optimal bounds are here expressed in terms of the conjugate $\om^\star$:
for a function $\omega : [0,\infty) \to [0,\infty)$ with $\om(t) = o(t)$ as $t \to \infty$ one defines
\begin{equation}\label{omegaconjugate}
\omega^{\star}(t):=\sup_{s\ge 0} \big(\omega(s)-st\big), \quad t>0.
\end{equation}
Then $\omega^\star$ is decreasing, continuous, and convex
with $\omega^{\star}(t) \to \infty$ as $t \to 0$, see \cite[Remark 1.5]{PetzscheVogt84}.
Since $\om(t) = o(t)$ as $t \to \infty$, $\om^\star(t)$ is finite for all $t$.
If $\omega$ is concave and increasing, then, by \cite[Proposition 1.6]{PetzscheVogt84},
\begin{equation}\label{omegaconjugate1}
\omega(t)=\inf_{s>0} \big(\omega^{\star}(s)+st\big), \quad t>0.
\end{equation}

\begin{theorem}[{\cite[Proposition 2.2]{BBMT91}}] \label{thm:WFcutoffR}
  Let $\om$ be a strong concave weight function.
  For each $n \in \N_{\ge 1}$ there exist $m \in \N_{\ge 1}$, $C>0$ and $0<r_0 <\frac{1}{2}$ such that
  for all $0<r<r_0$ there exist $\vh_{n,r} \in C^\infty(\R)$ with the following properties:
  \begin{enumerate}
    \item $0 \le \vh_{n,r} \le 1$.
    \item $\vh_{n,r}|_{[-r,r]}=1$ and $\on{supp} \vh_{n,r} \subseteq [-\frac{9}8r,\frac{9}8r]$.
    \item We have
    \[
      \|\vh_{n,r}\|^\om_{\R,m} \le C \exp\Big(\frac{1}{n} \om^\star(nr)\Big).
    \]
  \end{enumerate}
\end{theorem}

The functions $\vh_{n,r}$ are useful for the Roumieu case. The following cutoff functions
are needed in the Beurling case.

\begin{theorem}[{\cite[Corollary 2.6]{Franken:1993tn}}] \label{thm:WFoptimalcutoff}
  Let $\om$ be a strong concave weight function.
  There exist functions $(\vh_r)_{r>0}$ in $\cE^{(\om)}(\R)$ such that:
  \begin{enumerate}
    \item $0 \le \vh_{r} \le 1$.
    \item $\vh_{r}|_{[-r,r]}=1$ and $\on{supp} \vh_{r} \subseteq [-\frac{9}8r,\frac{9}8r]$.
    \item For each $m \in \N_{\ge 1}$ there exist $C>0$ and $n \in\N_{\ge 1}$ such that for all $r>0$
    \[
      \|\vh_{r}\|^\om_{\R,\frac{1}{m}} \le C \exp(n \om^\star(r)).
    \]
  \end{enumerate}
\end{theorem}

In combination with \Cref{prop:Wcubes} we easily obtain corresponding partitions of unity
subordinate to families of Whitney cubes for a given closed set $A$.

\begin{remark}
  The existence of such \emph{optimal} cutoff functions is equivalent to the fact that $\om$ is a strong
  weight function and additionally to the exactness of certain $\ol \p$-complexes, see \cite{Langenbruch:1994vt}.
  Recall that a strong weight function is equivalent to a strong concave weight function, by \Cref{lem:MeiseTaylor}.
\end{remark}

\subsection{Extension of Whitney ultrajets}

\begin{theorem}[{\cite{BBMT91}}] \label{thm:OmWhitney}
  Let $\om$ be a non-quasianalytic weight function.
  Then the following conditions are equivalent:
  \begin{enumerate}
    \item $\om$ is strong.
    \item $j^\infty_A : \cE^{\{\om\}}(\R^n) \to \cE^{\{\om\}}(A)$ is surjective for every closed subset $A \subseteq \R^n$.
    \item $j^\infty_A : \cE^{(\om)}(\R^n) \to \cE^{(\om)}(A)$ is surjective for every closed subset $A \subseteq \R^n$.
  \end{enumerate}
\end{theorem}

The sufficiency of strongness of $\om$ for (2) is proved in analogy to \Cref{thm:WRoumieu}:
Without loss of generality one may assume that $A$ is compact. The
case of the singleton yields local extensions $f_x$, $x \in A$, which form a bounded set in $\cB^{\{\om\}}(\R^n)$, cf.\
\Cref{lem:BorelSilva}.
The existence of an optimal partition of unity for a
family of Whitney cubes for $A$ allows to define the required extension by a formula similar to \eqref{eq:extensionformula}.

The Beurling case (3) can be reduced to the Roumieu case (2) by a reduction lemma (see \cite{BBMT91}) which is similar in
spirit to \Cref{lem:CC17}.

The necessity of (1) for (2) as well as for (3) follows from the special case of the singleton.
But is was shown in \cite{Abanin01} that $\om$ must be strong, if there is \emph{any} non-empty compact set $K \subseteq \R^n$
such that $j^\infty_K : \cE^{\{\om\}}(\R^n) \to \cE^{\{\om\}}(K)$
or $j^\infty_K : \cE^{(\om)}(\R^n) \to \cE^{(\om)}(K)$ is surjective.

\subsection{Extension operators}

We saw in \Cref{thm:nicegeometry} that
compact sets with real analytic boundary admit extension operators
in the Beurling case $\cE^{(\om)}$ for all strong weight functions $\om$.
The singleton $\{0\}$, on the other hand, admits an extension operator if and only if $\om$ is a strong (DN)-weight.
The existence of optimal cutoff functions of Beurling type (\Cref{thm:WFoptimalcutoff})
make it possible to extend this result to all closed sets.

\begin{theorem}[{\cite[Theorem 1]{Franken:1993tn}}]
  Let $\om$ be a strong weight function.
  Then the following conditions are equivalent:
  \begin{enumerate}
    \item $\om$ is a (DN)-weight.
    \item $j^\infty_A : \cE^{(\om)}(\R^n) \to \cE^{(\om)}(A)$ is split surjective for every closed subset $A \subseteq \R^n$.
  \end{enumerate}
\end{theorem}

The proof is analogous to the one of \Cref{thm:WSsplit}. Indeed, the case of the singleton guarantees that local extension operators
$E_x$, $x \in A$, exist and $\{E_x : x \in A\}$ is locally equicontinuous in $L(\cE^{(\om)}(A),\cE^{(\om)}(\R^n))$.
And as in \Cref{prop:localproperty} one shows that admitting an extension operator is a local property.

If $\om$ is strong but not a (DN)-weight, then
the existence of an extension operator depends on the set.
In fact, if $K \subseteq \R^n$ is a non-empty compact set
and $L$ a compact neighborhood of $K$, then
\begin{equation}\label{Omshortexactsequence}
 \xymatrix{
 0 \ar[r] & \cD^{(\om)}(L,K)~ \ar@{^{(}->}[r] & \cD^{(\om)}(L) \ar[r]^{j^\infty_K} & \cE^{(\om)}(K) \ar[r] & 0
 }
\end{equation}
is an exact sequence of Fr\'echet spaces, by \Cref{thm:OmWhitney},
where
\[
 \cD^{(\om)}(L,K) := \{f \in \cE^{(\om)}(\R^n) : \on{supp} f \subseteq L,\, j^\infty_K f = 0\}
\]
and $\cD^{(\om)}(L)$
carry the subspace topology of $\cE^{(\om)}(\R^n)$.
It is shown in \cite[Proposition 4.4]{Franken:1993tn} that
$\cD^{(\om)}(L,K)$ has the property ($\Om$); the proof is based on the result of Meise and Taylor \cite[Corollary 2.9]{Meise:1989un} that $\cD^{(\om)}(K)$ has property ($\Om$) if $K$
has real analytic boundary and on the existence of optimal cutoff functions (\Cref{thm:WFoptimalcutoff}).
Thus the splitting theorem of Vogt and Wagner \cite{Vogt:1980th} gives

\begin{theorem}[{\cite[Theorem 2]{Franken:1993tn}}]
  Let $\om$ be a strong weight function and let $K \subseteq \R^n$ be a non-empty compact set.
  Then the following conditions are equivalent:
  \begin{enumerate}
    \item $\cE^{(\om)}(K)$ has property (DN).
    \item $j^\infty_K : \cE^{(\om)}(\R^n) \to \cE^{(\om)}(K)$ is split surjective.
  \end{enumerate}
\end{theorem}

Compact sets that do not admit extension operators are the
\emph{flat cusps}
\[
  K_f=\{(x,y) \in [0,1]^2 : |y| \le |f(x)|\},
\]
where $f \in \cE^{(\om)}(\R)$ with $j^\infty_{\{0\}} f = 0$.
If $\om$ is not a (DN)-weight, then $\cE^{(\om)}(K_f)$ does not have (DN), and
$j^\infty_K : \cE^{(\om)}(\R^n) \to \cE^{(\om)}(K)$ is not split surjective;
see \cite{Franken:1994uh}.

\begin{remark}
  Recall that on compact sets with the Markov property,
  functions, that can be approximated by polynomials in an $[M]$-rapid manner,
  admit extensions to functions of Denjoy--Carleman type on the ambient space by a continuous linear map.
  But this extension involves a loss of regularity; cf.\ \Cref{sec:polyapprox}.
  In the framework of Braun--Meise--Taylor classes there is a class of weight functions, namely strong weight functions $\om$ such that\footnote{Note that \eqref{eq:tsquare} implies that $\om$ is strong, e.g.\ by \Cref{lem:MeiseTaylor}.}
  \begin{equation} \label{eq:tsquare}
    \E L>1 \A t\ge 0 : \om(t^2) \le L\om(t) +L,
  \end{equation}
  for which the analogous problem allows a solution without loss of regularity (at least in the Beurling case), see
  \cite{Franken:1995ta}. For instance,
    the weight functions $\om_s(t) = \max\{0,(\log t)^s\}$, $s>1$, satisfy \eqref{eq:tsquare}.
\end{remark}

\section{Extensions not preserving the ultradifferentiable class} \label{sec:BMTmixed}

As for Denjoy--Carleman classes it is natural to ask
whether the loss of regularity in the extension, for weight functions that are not strong,
can be determined and controlled.

\subsection{The singleton}

For the singleton we have

\begin{theorem}[\cite{BonetMeiseTaylor92}]
  Let $\om$ be a non-quasianalytic weight function, and $\si$ another weight function.
  Then the following conditions are equivalent:
  \begin{enumerate}
    \item $\exists C>0 \A t > 0 : \int_{1}^\infty \frac{\om(tu)}{u^2}\,du \le C\si(t) + C$.
    \item $\cE^{\{\si\}}(\{0\}) \subseteq j^\infty_{\{0\}} \cE^{\{\om\}}(\R^n)$.
    \item $\cE^{(\si)}(\{0\}) \subseteq j^\infty_{\{0\}} \cE^{(\om)}(\R^n)$.
  \end{enumerate}
\end{theorem}

This result was extended to compact convex sets with non-empty interior, by Langenbruch \cite{Langenbruch94}.
The proofs are based on descriptions of the duals of the spaces at hand as weighted spaces of
entire functions.

\subsection{Roumieu extension on arbitrary closed sets}
The general case was recently solved in the Roumieu case:

\begin{theorem}[{\cite{Rainer:2019ac,Rainer:2020aa}}] \label{thm:mixedRoumieuWF}
    Let $\om$ be a non-quasianalytic concave weight function and $\si$ a weight function with $\si(t)=o(t)$ as $t\to \infty$.
    Then the following conditions are equivalent:
    \begin{enumerate}
      \item $\exists C>0 \A t > 0 : \int_{1}^\infty \frac{\om(tu)}{u^2}\,du \le C\si(t) + C$.
      \item $\cE^{\{\si\}}(A) \subseteq j^\infty_{A} \cE^{\{\om\}}(\R^n)$ for each closed subset $A \subseteq \R^n$.
    \end{enumerate}
\end{theorem}

Let us comment on the proof (of (1) $\Rightarrow$ (2)).
For convenience we call a pair $(\si,\om)$ of weight functions as in the theorem and satisfying (1)
\emph{admissible}.
The two crucial ingredients are
\begin{itemize}
  \item cutoff functions with bounds that reflect the condition (1),
  \item the extension method of Dyn'kin by Taylor approximation to higher and higher degree as $A$ is approached.
\end{itemize}
Easy modifications in the proof of \Cref{thm:WFcutoffR} yield the desired cutoff functions:

\begin{theorem}[{\cite[Proposition 4.1]{Rainer:2019ac}}] \label{thm:WFcutoffRmixed}
  Let $(\si,\om)$ be an admissible pair of weight functions.
  For each $n \in \N_{\ge 1}$ there exist $m \in \N_{\ge 1}$, $C>0$ and $0<r_0 <\frac{1}{2}$ such that
  for all $0<r<r_0$ there exist $\vh_{n,r} \in C^\infty(\R)$ with the following properties:
  \begin{enumerate}
    \item $0 \le \vh_{n,r} \le 1$.
    \item $\vh_{n,r}|_{[-r,r]}=1$ and $\on{supp} \vh_{n,r} \subseteq [-\frac{9}8r,\frac{9}8r]$.
    \item We have
    \begin{equation} \label{eq:WFcutoffRbound}
      \|\vh_{n,r}\|^\om_{\R,m} \le C \exp\Big(\frac{1}{n} \si^\star(nr)\Big).
    \end{equation}
  \end{enumerate}
  Here $\si^\star$ is the conjugate of $\si$ defined in \eqref{omegaconjugate}.
\end{theorem}

Note that $\ka(t) := \int_1^\infty \frac{\om(ut)}{u^2}\, du$, $t>0$, is a (possibly quasianalytic) weight function
and $(\ka,\om)$ is an admissible pair.
If $(\si,\om)$ is another admissible pair, then $\ka(t) = O(\si(t))$ as $t \to \infty$, i.e., $\cE^{[\si]} \subseteq \cE^{[\ka]}$.
Moreover, $\ka$ is concave; cf.\ \Cref{lem:MeiseTaylor}. It follows that we may assume without loss of generality that $\si$ is concave.

In order to implement Dyn'kin's extension procedure one
uses the description of Braun--Meise--Taylor spaces of functions and jets by the associated weight matrix, see
\Cref{thm:WMdescription}.
In view of \Cref{rem:WMconcave} we may assume that there is a weight matrix $\fS = \{S^x\}_{x>0}$
such that $\cE^{[\si]} = \cE^{[\fS]}$,
where each $S^x$ is strongly log-convex and
\[
  \sup_{j,k} \Big(\frac{s^x_{j+k}}{s^{2x}_js^{2x}_k}\Big) =H< \infty
\]
which entails
\begin{equation} \label{eq:hmixed}
  h_{s^x}(t) \le h_{s^{2x}}(Ht)^2, \quad  t>0, ~ x>0.
\end{equation}
Setting $t^x_k := \min_{0 \le j \le k} s^{2x}_j s^{2x}_{k-j}$ we obtain another
weight matrix $\fT = \{T^x\}_{x>0}$
such that $\cE^{[\si]} = \cE^{[\fT]}$,
where each $T^x$ is strongly log-convex and
\[
  \Big(\frac{t^x_k}{t^x_{k-1}}\Big)_{k \ge 1}
  = \Big(\frac{s^{2x}_1}{s^{2x}_{0}}, \frac{s^{2x}_1}{s^{2x}_{0}},\frac{s^{2x}_2}{s^{2x}_{1}}, \frac{s^{2x}_2}{s^{2x}_{1}},
  \frac{s^{2x}_3}{s^{2x}_{2}}, \frac{s^{2x}_3}{s^{2x}_{2}},\ldots\Big)
\]
which implies (cf.\ \eqref{counting2})
\begin{equation} \label{eq:Gamixed}
  2\Ga_{s^{2x}} = \Ga_{t^x}.
\end{equation}

Now, if $F=(F^\al)_\al$ is a Whitney ultrajet of class $\cE^{\{\si\}}$ on some compact set $K \subseteq \R^n$,
then there exists $x>0$ such that $F \in \cE^{\{T^x\}}(K)$.
Let $(Q_j)_{j\ge1}$ be a family of Whitney cubes for $K$ and $(\vh_{j,\ep})_{j\ge 1}$ a corresponding partition of
unity based on the cutoff functions from \Cref{thm:WFcutoffRmixed}. Let $x_j$ be the center of $Q_j$ and $\hat x_j \in K$
with $|\hat x_j - x_j| = d_K(x_j)$.
We define
\[
  f(z) := \begin{cases}
     \sum_{j \ge 1} \vh_{j,\ep} (z) T^{p(x_j)}_{\hat x_j} F(z) & \text{ if } z \in \R^n \setminus K,
     \\
     F^0(z) & \text{ if } z \in K,
  \end{cases}
\]
where
\[
  p(z) :=  2 \Ga_{s^{2x}}(L d_K(z))-1
\]
and $L$ is a positive constant.
Then is not difficult to mimic the steps sketched in \Cref{sec:mixedCC}
(in particular, choosing $\ep$ and $L$ suitably)
to see that $f$ indeed provides an extension of class $\cE^{\{\om\}}$.
The properties \eqref{eq:hmixed} and \eqref{eq:Gamixed} serve as substitutes of \eqref{hmg} and \eqref{eq32}.
The connection between the bound \eqref{eq:WFcutoffRbound} and the functions $h_{s^x}$ is established
by the observation (see \cite[Corollary 3.11]{Rainer:2019ac})
\[
  \A x >0 \E 0 <c \le 1 \A t>0 : e^{c\si^\star (t)} \le  \frac{e}{h_{s^x}(ct)}.
\]
For a detailed presentation we refer to \cite{Rainer:2020aa}.

\begin{remark} \label{rem:extopWFmixed}
  The proof shows that, if $\fS$ and $\fW$ are weight matrices with $\cE^{\{\fS\}} = \cE^{\{\si\}}$ and
  $\cE^{\{\fW\}} = \cE^{\{\om\}}$, respectively, then
  for each $S \in \fS$ and $a>0$ there exist $W \in \fW$, $b>0$, and an extension operator
  $\cE^{S}_a(K) \to \cE^{W}_b(\R^n)$.
\end{remark}

\subsection{Beurling extension on arbitrary closed sets}

We intend to reduce the Beurling to the Roumieu case (in a way similar to \Cref{thm:mixedBeurlingWS}).
So, for an admissible pair of weight functions $(\si,\om)$
and for a Whitney ultrajet $F$ of class $\cE^{(\si)}$,
we would like to find
an admissible pair of weight functions $(\tilde \si,\tilde \om)$
such that $F$ is also of Roumieu class $\cE^{\{\tilde \si\}}$ and  $\om(t) = o(\tilde \om(t))$ as $t \to \infty$.
Then we could infer from \Cref{thm:mixedRoumieuWF} that $F$ has an extension of class $\cE^{\{\tilde \om\}}$
and hence of class $\cE^{(\om)}$, thanks to  $\om(t) = o(\tilde \om(t))$ as $t \to \infty$.

Unfortunately, it is not clear how to transfer the condition
\[
  \E C>0 \A t > 0 : \int_{1}^\infty \frac{\om(tu)}{u^2}\,du \le C\si(t) + C
\]
to the pair $(\tilde \si,\tilde \om)$ (in such a way that also all the other requirements are fulfilled).
But a stronger condition can be transferred:
let $\om$ be a non-quasianalytic concave weight function and $\si$ a weight function with $\si(t)=o(t)$ as $t\to \infty$.
We say that the pair $(\si,\om)$ is \emph{strongly admissible} if
\begin{equation} \label{eq:stronglyadmissible}
  \exists r \in (0,1) \E C>0 \A t > 0 : \int_{1}^\infty \frac{\om(tu)}{u^{1+r}}\,du \le C\si(t) + C.
\end{equation}
It is obvious that a strongly admissible pair is admissible.
On the other hand $(\om_{\al -1},\om_\al)$, where $\om_\al(t) = t (\log t)^{-\al}$ and $\al>1$,
is admissible, but not strongly admissible; see \cite[Example 11]{Rainer:2020ab}.

\begin{lemma}[{\cite[Lemma 13]{Rainer:2020ab}}] \label{lem:lemmareduction}
  Let $(\si,\om)$ be a strongly admissible pair of weight functions
  and $f : [0,\infty) \to [0,\infty)$ any function satisfying $\si(t) = o(f(t))$ as $t \to \infty$.
  Then there exists a strongly admissible pair of weight functions $(\tilde \si,\tilde \om)$
  such that
  \[
  	\text{$\om(t) = o(\tilde \om(t))$, $\si(t) = o(\tilde \si(t))$, and
  	$\tilde \si(t) = o(f(t))$ as $t \to \infty$}.
  \]
\end{lemma}

For the proof of the lemma it is crucial that
\eqref{eq:stronglyadmissible} is equivalent to
\begin{equation}
		\label{strongercondition}
   		 \exists C>0 ~\exists K>H> 1   ~\exists t_0 \ge 0 ~\forall t\ge t_0
    	~\forall j\in \N_{\ge1} :
    	\om(K^jt) \le C H^j \si(t);
\end{equation}
see \cite[Proposition 7]{Rainer:2020ab}.
As a consequence (invoking also \Cref{lem:MeiseTaylor}),
the pair $(\om,\om)$ is strongly admissible if and only if $\om$ is strong; see \cite[Lemma 8]{Rainer:2020ab}.

Thanks to \Cref{lem:lemmareduction}, the reduction scheme alluded to above then gives

\begin{theorem}[{\cite[Theorem 2]{Rainer:2020ab}}]
Let $(\si,\om)$ be a strongly admissible pair of weight functions.
Then for every closed $A \subseteq \R^n$
we have $\cE^{(\si)}(A) \subseteq j^\infty_A(\cE^{(\om)}(\R^n))$.
\end{theorem}

It is an open question if the conclusion holds for admissible pairs $(\si,\om)$.

\subsection{Extension operators}
As a by-product we obtain an extension operator on certain subspaces of $\cE^{(\si)}(A)$ with values in $\cE^{(\om)}(\R^n)$.

\begin{theorem}[{\cite[Theorem 5]{Rainer:2020ab}}]
Let $(\si,\om)$ be a strongly admissible pair of weight functions.
If $\ta$ is a weight function with $\si(t) = o(\ta(t))$ as $t\to \infty$, then for each non-empty closed subset $A \subseteq \R^n$
there is an extension operator $\cE^{\{\ta\}}(A) \to \cE^{(\om)}(\R^n)$.
\end{theorem}

Applying \Cref{lem:lemmareduction} to $f = \ta$, yields a strongly admissible pair $(\tilde \si,\tilde \om)$
satisfying
$\om(t) = o(\tilde \om(t))$, $\si(t) = o(\tilde \si(t))$, and
  $\tilde \si(t) = o(\ta(t))$ as $t \to \infty$.
Let $K$ be a compact subset of $\R^n$.
Then we have continuous inclusions $\cE^{\{\ta\}}(K) \hookrightarrow \cE^{(\tilde \si)}(K)$ and
$\cE^{\{\tilde \om\}}(\R^n) \hookrightarrow \cE^{(\om)}(\R^n)$. Composing these maps with
\[
\xymatrix{
\cE^{(\tilde \si)}(K) ~\ar@{^{(}->}[r] &\cE^{\tilde S}_1(K) \ar[r] & \cE^{\tilde W}_b(\R^n) ~\ar@{^{(}->}[r] & \cE^{\{\tilde \om\}}(\R^n)
}
\]
where the middle arrow is the extension operator from \Cref{rem:extopWFmixed},
gives an extension operator $\cE^{\{\ta\}}(K) \to \cE^{(\om)}(\R^n)$.
Here $\tilde S \in \tilde \fS$, $\tilde W \in \tilde \fW$, and $\tilde \fS$, $\tilde \fW$ are weight matrices associated with $\tilde \si$, $\tilde \om$,
respectively.
If $A$ is a closed subset of $\R^n$, then we may use a suitable partition of unity
in order to construct the required extension operator.

\begin{remark}
  All extensions in \Cref{sec:BMT,sec:BMTmixed}
  can be chosen to be analytic in the complement of $A$.
  This follows from a result of Schmets and Validivia \cite{Schmets:1999aa} or
  by adapting the proof of Langenbruch \cite[Theorem 13]{Langenbruch03}.
\end{remark}


\def\cprime{$'$}
\providecommand{\bysame}{\leavevmode\hbox to3em{\hrulefill}\thinspace}
\providecommand{\MR}{\relax\ifhmode\unskip\space\fi MR }
\providecommand{\MRhref}[2]{%
  \href{http://www.ams.org/mathscinet-getitem?mr=#1}{#2}
}
\providecommand{\href}[2]{#2}

\end{document}